\documentclass[11pt,a4paper]{elsarticle}

\usepackage[utf8]{inputenc}

\usepackage{amsmath}
\usepackage{amsthm}
\usepackage{enumerate}
\usepackage{mathdots}
\usepackage{amsfonts}
\usepackage{multicol}
\usepackage{amsmath}
\usepackage{amssymb}
\usepackage{fancybox}
\usepackage[usenames,dvipsnames]{color}
\usepackage{graphicx}
\usepackage{tikz}
\usepackage[left=3cm,top=4cm,right=3cm,bottom=3.5cm]{geometry}
\usepackage{arydshln}

\newcommand{\C}{\mathbb{C}}

\newcommand{\efe}{\mathbb{F}}
\newcommand{\FF}{\mathbb{F}}
\newcommand{\F}{\mathbb{F}}

\newcommand{\la}{\lambda}

\def\rank{\mathop{\rm rank}\nolimits}
\newcommand{\wh}{\widehat}

\newtheorem{theo}{Theorem}[section]

\newtheorem{deff}[theo]{Definition}

\newtheorem{prop}[theo]{Proposition}

\newtheorem{lem}[theo]{Lemma}

\newtheorem{cor}[theo]{Corollary}

\newtheorem{rem}[theo]{Remark}

\newtheorem{example}[theo]{Example}

\DeclareMathOperator{\diag}{diag}
\DeclareMathOperator{\rev}{rev}

\begin{document}

	\title{Local Linearizations of Rational Matrices with Application to Rational Approximations of Nonlinear Eigenvalue Problems}
	\author[uc3m]{Froil\'{a}n M. Dopico\fnref{fn1}}
	\ead{dopico@math.uc3m.es}
	
	\author[upv]{Silvia Marcaida\fnref{fn2}}
	\ead{silvia.marcaida@ehu.eus}
	
	\author[uc3m]{Mar\'{i}a C. Quintana\fnref{fn1}}
	\ead{maquinta@math.uc3m.es}
	
	\author[ucl]{Paul~Van~Dooren\fnref{fn3}}
	\ead{paul.vandooren@uclouvain.be}

	\address[uc3m]{Departamento de Matem\'aticas,
		Universidad Carlos III de Madrid, Avda. Universidad 30, 28911 Legan\'es, Spain.}
	\address[upv]{Departamento de Matem\'{a}tica Aplicada y Estadística e Investigación Operativa,
		Universidad del Pa\'{\i}s Vasco UPV/EHU, Apdo. Correos 644, Bilbao 48080, Spain.}
	\address[ucl]{Department of Mathematical Engineering, Universit\'{e} catholique de Louvain, Avenue Georges Lema\^itre 4, B-1348 Louvain-la-Neuve, Belgium.}
	
	\fntext[fn1]{Supported by ``Ministerio de Econom\'ia, Industria y Competitividad (MINECO)'' of Spain and ``Fondo Europeo de Desarrollo Regional (FEDER)'' of EU through grants MTM2015-65798-P and MTM2017-90682-REDT. The research of M. C. Quintana is funded by the “contrato predoctoral” BES-2016-076744 of MINECO.}
	\fntext[fn2]{Supported by ``Ministerio de Econom\'ia, Industria y Competitividad (MINECO)'' of Spain and ``Fondo Europeo de Desarrollo Regional (FEDER)'' of EU through grants MTM2017-83624-P and MTM2017-90682-REDT, and by UPV/EHU through grant GIU16/42.}
	\fntext[fn3]{This work was partially developed while Paul Van Dooren held a ``Chair of Excellence UC3M - Banco de Santander'' at Universidad Carlos III de Madrid in the academic year 2017-2018.}
	
		\begin{abstract}
		This paper presents a definition for local linearizations of rational matrices and studies their properties. This definition allows us to introduce matrix pencils associated to a rational matrix that preserve its structure of zeros and poles in subsets of any algebraically closed field and also at infinity. Moreover, such definition includes, as particular cases, other definitions that have been used previously in the literature. In this way, this new theory of local linearizations captures and explains rigorously the properties of all the different pencils that have been used from the 1970's until 2019 for computing zeros, poles and eigenvalues of rational matrices. Particular attention is paid to those pencils that have appeared recently in the numerical solution of nonlinear eigenvalue problems through rational approximation.
	\end{abstract}


	\maketitle
	
	\section{Introduction}
	Rational matrices, i.e., matrices whose entries are rational functions of a scalar variable, are a classical topic inside matrix theory that has received a lot of attention since the 1950s, as a consequence of their fundamental role in linear systems and control theory \cite{McMi1,McMi2}. Classical references on rational matrices and their applications to these areas are, for instance, the pioneering monographs \cite{Kailath,Rosen}. The most relevant structural data of a rational matrix are its zeros and poles, together with their partial multiplicities or structural indices, and its minimal indices, which exist only when the matrix is singular, i.e., rectangular or square with identically zero determinant. These structural data are very important in the applications mentioned above, which motivated in the 1970s a considerable research activity on the development of numerical algorithms for computing them, see \cite{vandooren1981} and the references therein. Among the different algorithms developed for this purpose in the 1970-80s, the most reliable ones were based on constructing a matrix pencil, i.e., a matrix polynomial of degree $1$, containing exactly {\em all} the information about the structural data of the considered rational matrix \cite{vandooren1981,VVK79}, and then applying to this matrix pencil backward stable algorithms, developed also in the 1970s, for computing the eigenvalues and/or other structural data of general pencils \cite{molerstewart1973,vandooren1979}.
	
	The pencils mentioned in the previous paragraph are among the first examples of linearizations of rational matrices. Such pencils are, in fact, particular instances of minimal polynomial system matrices of the considered rational matrix, a key concept introduced by Rosenbrock \cite{Rosen} that allows us, among other things, to include simultaneously all the information about the zeros and the poles of a rational matrix into a polynomial matrix.
	
	Recently, rational matrices have received considerable attention from the  different perspective of what are called rational eigenvalue problems (REPs).
	Such REPs may arise directly from applications \cite{mehrmanvoss2004}, as approximations of other nonlinear eigenvalue problems (NLEPs) (see, for instance, \cite{nlep,automatic,Saad,van-beeumen-et-al-2018}), and, even more, REPs have also been used to approximate polynomial eigenvalue problems (PEPs) in order to take advantage of certain low rank structures \cite{lu-huang-bai-su-2015}. Since NLEPs are nowadays a very active area of research (see the recent survey \cite{guttel-tisseur-2017} and the references therein), REPs and rational matrices are currently a hot topic inside applied and numerical linear algebra. In this scenario, it is of interest to establish in the next paragraphs connections and differences between how rational matrices are viewed in the classic areas of linear systems and control theory and in the modern one of NLEPs, since, unfortunately, some modern and pioneering references on NLEPs seem to ignore classic results on rational matrices.
	
	First, let us review the definition of REPs. Given a regular rational matrix $G(\la)$, the corresponding REP is defined as computing numbers $\la_0$ and nonzero vectors $x$ such that $G(\la_0) x = 0$. These $\la_0$ and $x$ are called eigenvalues and (right) eigenvectors of $G(\la)$, respectively, a terminology inherited from other matrix eigenvalue problems but that has never been used in standard references on rational matrices \cite{Kailath,Rosen}. Observe that the definition of REP assumes implicitly that $G(\la_0)$ is defined at $\la_0$, i.e., none of their entries become infinite. Thus, using the classic definitions for the structural data of rational matrices, we can say that $\la_0$ is a zero of $G(\la)$ but not a pole, and we can see REPs as particular cases of the computational problems on rational matrices investigated in the 1970-80s.
	
Second, we emphasize that rational approximations of NLEPs are only reliable in a certain target set. Moreover, in many works \cite{nlep,automatic,Saad,van-beeumen-et-al-2018}, the matrix defining the NLEP is assumed to be analytic in the target region, and such region does not contain the poles of the rational matrix defining the approximating REP. In particular, the poles are already known from the approximation process. This means that for those rational matrices coming from approximating these NLEPs, the poles are of no interest (since they are known), and only those zeros (eigenvalues) in the target set have to be computed. In addition, the structure at infinity (see \cite{Kailath} for a definition) is also of no interest. This is in stark contrast with the situation for rational matrices arising in linear systems and control theory, which, usually, are transfer functions of time invariant linear systems and, therefore, all the finite and infinite structure of zeros and poles related to the transfer function is of interest and has to be computed \cite{vandooren1981}.
	
	As said before, some influential modern references on solving numerically NLEPs via rational approximations ignore classic results on rational matrices. Probably, this is a consequence of the differences mentioned in the previous paragraph and, also, of the fact that rational matrices coming from approximating NLEPs may appear represented in forms different from the most standard ones in linear system and control theory. This lack of connections with classical results is unfortunate, but has had also the positive effect of producing new results on and approaches to rational matrices. For instance, on the unfortunate side, it is surprising that the idea of solving REPs via linearizations was not used in modern references until the key paper \cite{su-bai-2011} was published, despite the fact it had been intensively used much earlier (see \cite{vandooren1981} and the references therein), and it is one of the most reliable methods for solving REPs. On the positive side, \cite{su-bai-2011} introduced a new companion-like linearization of any rational matrix that is very useful in computations. For this purpose, \cite{su-bai-2011} expressed the rational matrix as the sum of a polynomial matrix and a state-space realization and approached the problem with the spirit of linearizations of polynomial matrices \cite{GoLaRo82}, instead of using the classical point of view of polynomial system matrices. (However, it is worth highlighting that, in Example \ref{ex_subai}, we will see that the linearization in \cite{su-bai-2011} is nothing else than a polynomial system matrix of the considered rational matrix. We will see in Section \ref{guttel} that the same happens for the linearizations in \cite{nlep}.)
	
	Another point to be remarked is that reference \cite{su-bai-2011} started a confusing practice, common to several references dealing with linearizations of rational matrices that approximate NLEPs. Namely, to term as ``linearizations'' pencils which are proved to contain only partial information about the corresponding rational matrix. For example, the papers \cite{nlep,automatic,Saad,su-bai-2011}, which are excellent from the numerical point of view, only prove (at most) that the algebraic and geometric multiplicities of the eigenvalues are preserved in the ``linearization'', but nothing is proved about the partial multiplicities. This is in contrast with the standard definition of (strong) linearization of polynomial matrices \cite{GoLaRo82,spectral}, which guarantees that linearizations contain {\em all} the information about the eigenvalues of polynomial matrices (including at infinity in the strong case), as well as with the linear minimal polynomial system matrices used as linearizations of rational matrices in \cite{vandooren1981,VVK79}, which contain {\em all} the information about poles and zeros of the rational matrices.
	
	The partial results proved in \cite{su-bai-2011} were among the motivations of the development of a rigorous definition and theory of strong linearizations of arbitrary (regular or singular, square or rectangular) rational matrices in \cite{strong}. Moreover, infinitely many examples of such strong linearizations have been constructed in \cite[Section 5.2]{strong} through the family of so-called strong block minimal bases linearizations of rational matrices. In simple words, the main idea of the theory in \cite{strong} is to combine minimal polynomial system matrices of rational matrices with the theory of linearizations of polynomial matrices \cite{spectral,BKL,GoLaRo82} in the following sense: strong linearizations of a rational matrix $G(\la)$ are linear minimal polynomial system matrices of rational matrices $\widehat{G} (\la)$ that may be different from $G(\la)$, but that are related to it via unimodular polynomial matrices, biproper rational matrices, and direct sums with identities. In this way such strong linearizations contain {\em all} the information about poles and zeros of the considered rational matrices and extend the ``linearizations'' used in \cite{vandooren1981,VVK79}, which correspond to the particular case when  $\widehat{G} (\la) = G (\la)$. Related works about linearizations containing {\em all} the pole-zero information of a rational matrix (in some cases not at infinity) are \cite{AlBe16,dasalam2019-1,dasalam2019,dopmarquin2019}.
	
	However, the definitions of linearization and strong linearization in \cite{strong} do not capture always the pencils defined in \cite{nlep,automatic,Saad,su-bai-2011} for two reasons. First, the pencils in \cite{nlep,automatic,Saad,su-bai-2011} do not always satisfy the minimality requirements of the definitions in \cite{strong}. Second, and related to the first fact, some of these pencils may not content all the information about the poles of the rational matrix (neither the information of those zeros that are also poles), and a zero of the linearization could be a pole of the rational matrix but not a zero. But, we stress that this is not a drawback in the setting of \cite{nlep,automatic,Saad,su-bai-2011} because, as explained before, in these cases the poles are of no interest, and only the eigenvalues in a certain target set have to be computed. This motivates us to develop in this paper a theory of what we call local linearizations of rational matrices, where the word local means that the linearization is only guaranteed to contain all the information about those zeros and poles of the rational matrix which are located in a certain set.
	
	The theory of local linearizations of rational matrices captures all the pencils that have been used (as far as we know) in the literature for solving REPs arising from approximating NLEPs. As illustration, we will apply in this paper this theory to the pencils in \cite{nlep,Saad,su-bai-2011} in several different ways. The application to the pencils in \cite{automatic} is postponed to \cite{local2} with the goal of limiting the length of this paper. In addition, we will see that the definition of local linearizations include the definitions of linearizations and strong linearizations of arbitrary rational matrices presented in \cite{strong}, just by considering as set the whole underlying field and including infinity in the strong case. As a consequence, local linearizations also include the pencils originally used in \cite{vandooren1981,VVK79}. Thus, this new local theory is a flexible tool that generalizes and includes most of the previous results available in the literature in this area. This is in part possible due to a new and more flexible treatment of polynomial system matrices at infinity.
	
 The theory of local linearizations of rational matrices is based on the extension of Rosenbrock's fundamental concept of minimal polynomial system matrix to a local perspective. Such extension is performed in a very simple and applicable manner that avoids as much as possible the use of abstract algebraic concepts. This is in contrast with related local approaches as the one in \cite{cullen1986} and the references therein, which, in addition, are focused on the underlying local equivalence relationships rather than on the properties of polynomial system matrices.
The local linearization approach connects the concept of linearization with classical results as the local Smith form of polynomial matrices (see, for instance, \cite[Section S1.5]{GoLaRo82}) and the local Smith--McMillan form of rational matrices (see \cite[Theorem II.9]{Integral} and \cite{vandooren-laurent-1979}).
	
	The paper is organized as follows. Section \ref{sec.preliminaries} summarizes some basic results that will be used in the rest of the paper. Locally minimal polynomial system matrices are defined and studied in Section \ref{sec.polysysmat}. Section \ref{sec_local} presents the main definitions and properties of local linearizations of rational matrices. Section \ref{sec-full-rank-pencils} introduces the so-called block full rank pencils, which are linearizations of rational matrices that do not contain any information about the poles, and are closely related to the block minimal bases linearizations of polynomial matrices recently presented in \cite{BKL}. The application of the local theory to the pencils in \cite{nlep} is analyzed in depth and from two perspectives in Section \ref{guttel}. Finally, Section \ref{sect:con} discusses the conclusions and some lines of future research. Several examples that illustrate the theoretical results are scattered throughout the paper. They are often based on the pencils introduced in \cite{Saad,su-bai-2011}.

\section{Preliminaries} \label{sec.preliminaries}

We assume throughout this paper that $\F$ is an algebraically closed field that does not include infinity. As usual, $\efe[\la]$ denotes the ring of polynomials with coefficients in $\efe$ and $\efe(\la)$ the field of rational functions or, equivalently, the field of fractions of $\efe[\la]$. A rational function $r(\la)=\frac{n(\la)}{d(\la)}$ is said to be 
proper if $\deg(n(\la))\leq\deg(d(\la)),$
strictly proper if $\deg(n(\la))<\deg(d(\la)),$ and
biproper if $\deg(n(\la))=\deg(d(\la))$, where $\deg(\cdot)$ stands for ``degree of''.

$\efe^{p\times m}$, $\efe[\la]^{p\times m}$ and $\efe(\la)^{p\times m}$ denote the sets of $p\times m$ matrices with elements in $\efe,$ $\efe[\la]$ and $\efe(\la),$  respectively. The elements of $\efe[\la]^{p\times m}$ are called polynomial matrices  or matrix polynomials. In the sequel we will use both terms. A unimodular matrix is a square polynomial matrix with polynomial inverse or, equivalently,  a square polynomial matrix with nonzero constant determinant. Moreover, the elements of $\efe(\la)^{p\times m}$ are called rational matrices. A (strictly) proper rational matrix is a rational matrix whose entries are (strictly) proper rational functions. A biproper matrix is a square proper matrix with proper inverse or, equivalently, a square proper matrix whose determinant is a biproper rational function. The normal rank of a polynomial or rational matrix $G(\la)$ is the size of its largest nonidentically zero minor and is denoted by $\rank G(\la)$. See \cite{Kailath} and \cite{Vard} for more information on these and other concepts related to polynomial and rational matrices.
	
As a first step to define local linearizations of rational matrices, we present local notions and results about rational matrices. We denote the point at infinity as $\infty.$

\begin{deff}
Let $R(\la)\in\F(\la)^{p\times m}$. Let $\la_{0}\in\efe$, and $\Sigma\subseteq\F$ be nonempty.	
\begin{itemize}
	\item [\rm(i)] $R(\la)$ is \textit{defined or bounded at $\la_{0}$} if  $R(\la_0)\in\F^{p\times m}.$
	\item [\rm(ii)] $R(\la)$ is \textit{defined or bounded at $\infty$} if  $R(1/\la)$ is defined at $0.$
	\item [\rm(iii)] $R(\la)$ is \textit{defined or bounded in $\Sigma$} if  $R(\la_0)\in\F^{p\times m}$ for all $\la_0\in\Sigma$.
\end{itemize}
\end{deff}

Notice that a rational matrix being defined at $\la_{0}\in\efe$ is equivalent to having a Taylor expansion around $\la_0.$ Moreover, a rational matrix is defined at infinity if and only if is proper.

\begin{deff}\label{def:regular}
Let $R(\la)\in\F(\la)^{m\times m}$.	Let $\la_{0}\in\efe$, and $\Sigma\subseteq\F$ be nonempty.
\begin{itemize}
	\item [\rm (i)] $R(\la)$ is \textit{regular or invertible at $\la_{0}$} if it is defined at $\la_0$ and $\det R(\la_0)\neq 0.$
	\item [ \rm(ii)] $R(\la)$ is \textit{regular or invertible at $\infty$} if $R(1/\la)$ is regular at $0.$
	\item [ \rm (iii)] $R(\la)$ is \textit{regular or invertible in $\Sigma$} if it is regular at each $\la_0\in\Sigma.$
\end{itemize}

\end{deff}
A rational matrix $R(\la)$ is said to be regular if it is regular for some $\la_0\in\F.$ That is, if $R(\la)$ is square and $\det R(\la)\not\equiv 0.$ Note that $R(\la)$ is regular at $\la_0\in\efe$ if and only if both $R(\la)$ and $R(\la)^{-1}$ have a Taylor expansion around $\la_0.$ Moreover, biproper matrices are those rational matrices that are regular at infinity, while unimodular matrices are those rational matrices that are regular in $\efe$.

In regard to the previous definitions, we introduce some equivalence relations defined in the set of rational matrices \cite{AmMaZa13, AmMaZa15}.

\begin{deff}\label{def:equivalent}
Let $G(\la), H(\la)\in\F(\lambda)^{p\times m}$. Let $\la_{0}\in\efe$, and $\Sigma\subseteq\F$ be nonempty.
\begin{itemize}
	\item [\rm(i)]
  $G(\la)$ and $H(\la)$ are equivalent at $\la_0$ if there exist rational matrices $R_1(\la)\in\F(\lambda)^{p\times p}$ and $R_2(\la)\in\F(\lambda)^{m\times m}$ both regular at $\la_0$ such that $R_1(\la)G(\la)R_2(\la)=H(\la).$
	\item [\rm(ii)] $G(\la)$ and $H(\la)$ are equivalent at $\infty$ if there exist rational matrices $R_1(\la)\in\F(\lambda)^{p\times p}$ and $R_2(\la)\in\F(\lambda)^{m\times m}$ both regular at $\infty$ such that $R_1(\la)G(\la)R_2(\la)=H(\la).$
	\item [\rm(iii)] $G(\la)$ and $H(\la)$ are equivalent in $\Sigma$ if there exist rational matrices $R_1(\la)\in\F(\lambda)^{p\times p}$ and $R_2(\la)\in\F(\lambda)^{m\times m}$ both regular in $\Sigma$ such that $R_1(\la)G(\la)R_2(\la)=H(\la).$
\end{itemize}
\end{deff}

Note that if $\Sigma=\F$ is considered in Definition \ref{def:equivalent}$\rm(iii)$, then $R_{1}(\lambda)$ and $R_{2}(\la)$ are both unimodular, and the standard definition of unimodular equivalence is recovered.

We now introduce the definition of the local Smith--McMillan form of a rational matrix at a point (finite and infinite). The  notion of the Smith--McMillan form of a rational matrix was first studied by McMillan in \cite{McMi1,McMi2} and, then, in other works as \cite{Kailath,Integral,Rosen,Vard, Factorization}.  The local Smith--McMillan form is a particular case of the very general (and abstract) result \cite[Theorem II.9]{Integral}. A description valid for rational matrices over the complex field can be found in \cite{vandooren-laurent-1979}, and a complete and rigorous modern treatment in \cite{AmMaZa15}.
Let $G(\lambda)\in\F(\lambda)^{p\times m}$ be any rational matrix of normal rank $r$. Let $\lambda_{0}\in\efe.$ Then $G(\la)$ is equivalent at $\la_0$ to a matrix of the form
\begin{equation}\label{localsm}
\left[\begin{array}{cc}
\diag\left((\la-\la_{0})^{\nu_{1}},\ldots, (\la-\la_{0})^{\nu_{r}}\right)&0 \\
0& 0_{(p-r)\times (m-r)}
\end{array}\right],
\end{equation}
where $\nu_{1}\leq\cdots\leq\nu_{r}$ are integers. The integers $\nu_{1},\ldots,\nu_{r}$ are uniquely determined by $G(\la)$ and $\la_{0}$, and are called the invariant orders at $\la_0$ of $G(\la)$. The matrix in \eqref{localsm} is called the local Smith--McMillan form of $G(\la)$ at $\la_{0}.$ Moreover, $G(\la)$ is equivalent at $\infty$ to a matrix of the form
\begin{equation}\label{localsm_inf}
	\left[\begin{array}{cc}
	\diag\left(\frac{1}{{\la}^{\mu_{1}}},\ldots,\frac{1}{{\la}^{\mu_{r}}}\right)&0 \\
	0& 0_{(p-r)\times (m-r)}
	\end{array}\right]
\end{equation}
where $\mu_{1}\leq\cdots\leq\mu_{r}$ are integers. These integers $\mu_{1},\ldots,\mu_{r}$ are uniquely determined by $G(\la)$, and are called the invariant orders at infinity of $G(\la)$. The matrix in \eqref{localsm_inf} is called the Smith--McMillan form of $G(\la)$ at $\infty.$




In order to define zeros and poles we need to distinguish between positive and negative invariant orders \cite{Kailath,Vard}. When we say that a rational matrix has $\nu_1\leq\cdots\leq \nu_k<0=\nu_{k+1}=\cdots=\nu_{u-1}<\nu_u\leq\cdots\leq \nu_r$ as invariant orders at $\la_0$ (infinity) we mean that $k$ may take values from 0 to $r$ and $u$ from $1$ to $r+1$.
For instance, if $k=0$ all the invariant orders are nonnegative; if, in addition, $u=1$ then they are all positive, but if $k=0$ and $u=r+1$ they are all 0.
 	
\begin{deff}
Let $G(\la)\in\F(\lambda)^{p\times m}$ and $\la_{0}\in\F$. Let
$\nu_1\leq\cdots\leq \nu_k<0=\nu_{k+1}=\cdots=\nu_{u-1}<\nu_u\leq\cdots\leq \nu_r$ be the invariant orders at $\la_{0}$ of $G(\la)$. Then $\la_{0}$ is said to be a pole of $G(\la)$ with \textit{partial multiplicities} $-\nu_k,\ldots, -\nu_1,$ and a zero of $G(\la)$ with \textit{partial multiplicities} $\nu_u,\ldots, \nu_r.$ In particular, the positive integers $-\nu_k,\ldots, -\nu_1$ and $\nu_u,\ldots, \nu_r$ are called the pole and zero partial multiplicities of $G(\la)$ at $\la_0,$ respectively. Moreover, $(\la-\la_0)^{-\nu_i}$ for $i=1,\ldots,k$ are called  the pole elementary divisors of $G(\la)$ at $\la_0$, while $(\la-\la_0)^{\nu_i}$ for $i=u,\ldots,r$ are called the zero elementary divisors of $G(\la)$ at $\la_0.$ Finally, the pole (zero) algebraic multiplicity of $\la_0$ is the sum of its pole (zero) partial multiplicities, and the pole (zero) geometric multiplicity of $\la_0$ is the number of its pole (zero) partial multiplicities.
\end{deff}	

If $G(\la)$ is a polynomial matrix then the polynomials $(\la-\la_0)^{\nu_i}$ with $\nu_i\neq 0$ are simply called elementary divisors of $G(\la)$ at $\la_0,$ and the nonzero integers $\nu_i\neq 0$ are all positive and are called partial multiplicities of $G(\la)$ at $\la_0.$

\begin{deff}
Let $G(\la)\in\F(\lambda)^{p\times m}$.	Let $\mu_1\leq\cdots\leq \mu_{\ell}<0=\mu_{{\ell}+1}=\cdots=\mu_{t-1}<\mu_t\leq\cdots\leq \mu_r$ be the invariant orders at $\infty$ of $G(\la)$. Then $\infty$ is said to be a pole of $G(\la)$ with \textit{partial multiplicities} $-\mu_{\ell},\ldots, -\mu_{1},$ and a zero of $G(\la)$ with \textit{partial multiplicities} $\mu_t,\ldots, \mu_r.$ In particular, the integers $-\mu_{\ell},\ldots, -\mu_1$ and $\mu_t,\ldots, \mu_r$ are called the pole and zero partial multiplicities of $G(\la)$ at $\infty,$ respectively.
\end{deff}

Some modern references, see for instance \cite{AlBe16,nlep,su-bai-2011}, also consider (finite) eigenvalues of rational matrices, a concept that is not mentioned at all in classical references of rational matrices. According to these modern references, we introduce the following definition.
\begin{deff} Let $G(\la)\in\F(\la)^{p\times m}$ be a rational matrix. A finite eigenvalue of $G(\la)$ is any $\la_0\in\F$ such that $\rank G(\la_0)< \rank G(\la),$\footnote{Note that here $\rank G(\la)$ denotes the normal rank of $G(\la),$ while $\rank G(\la_0)$ is the rank of the constant matrix $G(\la_0).$} with $G(\la_0)\in \F^{p\times m}.$ That is, $\la_0$ is a finite zero of $G(\la)$ but not a pole.
\end{deff}
 Observe that if $G(\la)\in\F(\la)^{p\times p}$ is regular, an eigenvalue of $G(\la)$ is any $\la_0\in\F$ such that there exists a nonzero vector $x\in\F^{p}$ satisfying $G(\la_0)x=0$ with $G(\la_0)\in \F^{p\times p},$ which is the standard definition of REP (Rational Eigenvalue Problem).

As a consequence of \cite[Theorem 2.3]{AmMaZa15} (see \cite[Section 2]{AmMaZa13} for more details) we can also present the Smith--McMillan form of a rational matrix in a nonempty subset of $\efe$, say $\Sigma$. Let $G(\la)\in\efe(\la)^{p\times m}$ with normal rank $r$. Then $G(\la)$ is equivalent in $\Sigma$ to a matrix of the form
\begin{equation}\label{eq:globSM}
\left[\begin{array}{cc}
\diag\left(\frac{\epsilon_1(\la)}{\psi_1(\la)},\ldots, \frac{\epsilon_r(\la)}{\psi_r(\la)}\right)&0 \\
0& 0_{(p-r)\times (m-r)}
\end{array}\right]
\end{equation}
where, for $i=1,\ldots, r$, $\frac{\epsilon_i(\la)}{\psi_i(\la)}$ are nonzero irreducible rational functions, $\epsilon_i(\la)$ and $\psi_i(\la)$ are monic (leading coefficient equal to 1) polynomials which are either constants or whose roots are in $\Sigma$ and $\epsilon_1(\la)\mid\cdots\mid\epsilon_r(\la)$ while $\psi_r(\la)\mid\cdots\mid\psi_1(\la)$, where $\mid$ stands for divisibility. We refer to (\ref{eq:globSM}) as the Smith--McMillan form in $\Sigma$ of $G(\la)$. When we take $\Sigma=\efe$, we obtain the (finite) Smith--McMillan form of $G(\la)$, i.e., the classical Smith--McMillan form of $G(\la)$.  In this case, if $G(\la)$ is polynomial then $\psi_1(\la)=\cdots =\psi_r(\la)=1,$ $\epsilon_1(\la),\ldots,\epsilon_r(\la)$ are the invariant polynomials of $G(\la)$, and (\ref{eq:globSM}) is called the Smith normal form of $G(\la)$.

Notice that the Smith--McMillan form of a rational matrix in a nonempty set $\Sigma\subseteq\F$ is invariant under multiplication by regular rational matrices in $\Sigma,$ i.e., under equivalence in $\Sigma.$ Analogously, the Smith--McMillan form at $\infty$ is invariant under multiplication by biproper matrices, i.e., under equivalence at $\infty.$

The next result shows that the equivalence of rational matrices in nonempty sets is a local property.

\begin{prop}\label{local property} Let $\Sigma\subseteq\F$ be nonempty. Two rational matrices of the same size
are equivalent in $\Sigma$ if and only if they are equivalent at each $\la_0\in\Sigma.$
\end{prop}

\begin{proof} If two rational matrices are equivalent in $\Sigma$ then, by Definitions \ref{def:equivalent} and \ref{def:regular}, it is straightforward that they are equivalent at each $\la_0\in\Sigma$. For the converse, suppose that $G(\lambda)$ and $H(\la)$ are equivalent at each $\la_0\in\Sigma.$ Then, $G(\lambda)$ and $H(\la)$ have the same local Smith--McMillan forms at each $\la_0\in\Sigma.$ In particular, $G(\lambda)$ and $H(\la)$ have the same pole and zero elementary divisors  at each $\la_0\in\Sigma.$ Let us consider $M_G(\la)$ and $M_H(\la)$ as the global Smith--McMillan forms of $G(\la)$ and $H(\la),$ respectively. Thus, there exist unimodular matrices $U_{i}^{G}(\la),$ $U_i^{H}(\la)$ for $i=1,2,$ such that $G(\la)= U_{1}^{G}(\la) M_G(\la)U_{2}^{G}(\la)$, $H(\la) = U_{1}^{H}(\la)M_H(\la)U_{2}^{H}(\la)$, and we can write
	\begin{equation*}
	\begin{split}
	 M_G(\la) &=\diag\left(
	f_1(\la)g_1(\la) ,\ldots,
	f_r(\la)g_r(\la),
	0_{(p-r)\times (m-r)}\right)\text{, and}\\ M_H(\la) & =\diag\left(
	f_1(\la)h_1(\la) ,\ldots,
	f_r(\la)h_r(\la),0_{(p-r)\times (m-r)}\right),
	\end{split}
	\end{equation*}
where $f_i(\la)$ are rational functions which are either equal to one or have poles and zeros in $\Sigma,$ while $g_{i}(\la)$ and $h_{i}(\la)$ are rational functions that do not have neither poles nor zeros in $\Sigma.$ Let us define $R(\la):=\diag\left(\dfrac{h_1(\la)}{g_1(\la)},\ldots,\dfrac{h_r(\la)}{g_r(\la)}, I_{m-r} \right).$ Hence, $M_H(\la)=M_G(\la)R(\la).$ Therefore, we deduce that $H(\la)=U_{1}^{H}(\la) U_{1}^{G}(\la)^{-1}G(\la) U_{2}^{G}(\la)^{-1}R(\la)U_{2}^{H}(\la),$ and $G(\la)$ and $H(\la)$ are equivalent in $\Sigma$ since the matrices $U_{1}^{H}(\la) U_{1}^{G}(\la)^{-1}$ and $U_{2}^{G}(\la)^{-1}R(\la)U_{2}^{H}(\la)$ are regular in $\Sigma.$
\end{proof}

\section{Polynomial system matrices minimal in subsets of $\F$ and at infinity}\label{sec.polysysmat}
Polynomial system matrices are a classical tool for studying rational matrices. They were introduced by Rosenbrock and are analyzed in detail in \cite{Rosen}. Among them, minimal polynomial system matrices have been used in many problems dealing with rational matrices because they allow to extract all the information about finite poles and zeros. Recently, they have played a fundamental role in developing a rigorous theory of linearizations and strong linearizations of rational matrices \cite{strong}. In this section, we extend the concept of minimal polynomial system matrices from the classical global scenario to a local one.
Some of the definitions in this section can also be found in \cite{cullen1986} expressed in an abstract algebraic language.

\subsection{Polynomial system matrices minimal in subsets of $\F$}
In this section we introduce polynomial system matrices of rational matrices that are locally minimal, and study their properties. Consider the fact that any rational matrix  $G(\la)\in\FF(\la)^{p\times m}$  can be written as
\begin{equation*}\label{s1.eqrealizG}
G(\la)=D(\la)+C(\la)A( \la)^{-1}B(\la)
\end{equation*}
for some polynomial matrices  $A(\la)\in\FF[\la]^ {n\times n},$   $B(\la)\in\FF[\la]^{n\times m}$, $C(\la)\in\FF[\la]^{p\times n}$ and $D(\la)\in\FF[\la]^{p\times m}$
with $A(\la)$ nonsingular if $n>0$ (see \cite{Rosen}). Then the matrix polynomial
\begin{equation}\label{eq:polsysmat}
P(\la)=\begin{bmatrix}
A(\la) & B(\la)\\
-C(\la) & D(\la)
\end{bmatrix}
\end{equation}
is called a polynomial system matrix of $G(\la)$ \cite{Rosen}. That is, $G(\la)$ is the Schur complement of $A(\la)$ in $P(\la)$. In that case, $A(\la)$ is called the state matrix of $P(\la)$ and $G(\la)$ is the transfer function matrix of
$P(\la).$ If $n=0,$ we assume that the matrices $A(\la),$ $B(\la)$ and $C(\la)$ are empty, and $P(\la)=G(\la)=D(\la)$ is a polynomial matrix. We emphasize that the definition of polynomial system matrix of a rational matrix includes a specific partition. Sometimes in this paper a certain polynomial matrix is partitioned in different ways giving rise to different polynomial system matrices of (possibly) different rational matrices. In such cases, we often use expressions as ``$P(\la)$ is a polynomial system matrix of $G(\la)$ with state matrix $A(\la)$'' in order to avoid ambiguities, where the words ``of $G(\la)$'' may be omitted because $P(\la)$ and $A(\la)$ determine $G(\la)$. In the case $n=0$ mentioned above, we will use ``$P(\la)$ is a polynomial system matrix with empty state matrix''. We stress that although in \eqref{eq:polsysmat} the state matrix is in the $(1,1)$-block, it might be a different submatrix of $P(\la)$. In general, the fundamental property defining a polynomial system matrix is that the rational matrix is the Schur complement of the state matrix.

We remark that the relation between the normal ranks of $P(\la)$ and its transfer function matrix $G(\la)$ is
\begin{equation}\label{ranks_rel}
\rank P(\la)= n + \rank G(\la),
\end{equation}
since we can write $P(\la)$ as
\begin{equation*}
P(\la)=\begin{bmatrix}
I_n & 0\\
-C(\la)A(\la)^{-1} & I_p
\end{bmatrix}\begin{bmatrix}
A(\la) & 0\\
0 & G(\la)
\end{bmatrix}\begin{bmatrix}
I_n & A(\la)^{-1}B(\la)\\
0 & I_m
\end{bmatrix}.
\end{equation*}

Next, we introduce two of the main definitions of this work.

\begin{deff}[Polynomial system matrix minimal at a point in $\F$]\label{def_minimalpolsysmat} Let $\la_0\in\efe.$ The polynomial system matrix $P(\la)$ in \eqref{eq:polsysmat}, with $n>0,$ is said to be minimal at $\la_0$ if
	$$\rank\begin{bmatrix} A(\la_0) \\ C(\la_0)\end{bmatrix}=\rank\begin{bmatrix} A(\la_0) & B(\la_0) \end{bmatrix}=n.$$

\end{deff}

\begin{rem} If $P(\la)$ is a polynomial system matrix as in \eqref{eq:polsysmat}, with $n>0,$ then $$\rank\begin{bmatrix} A(\la) \\ C(\la)\end{bmatrix}=\rank\begin{bmatrix} A(\la) & B(\la) \end{bmatrix}=n$$
	since $A(\la)$ is nonsingular. Thus $P(\la)$ is minimal at $\la_0$ if and only if $\la_0$ is neither an eigenvalue of	$\begin{bmatrix} A(\la) \\ C(\la)\end{bmatrix}$ nor of $\begin{bmatrix} A(\la) & B(\la) \end{bmatrix}.$
\end{rem}

\begin{deff}[Polynomial system matrix minimal in a subset of $\F$]\label{def_minimalpolsysmatsubset} Let $\Sigma\subseteq\efe$ be nonempty. The polynomial system matrix $P(\la)$ in \eqref{eq:polsysmat}, with $n>0,$ is minimal in $\Sigma$ if $P(\la)$ is minimal at each point $\la_{0}\in\Sigma.$
\end{deff}

 Observe that Definitions \ref{def_minimalpolsysmat} and \ref{def_minimalpolsysmatsubset} extend to points and subsets of $\F$ the classical definition of minimal, or with least order, polynomial system matrices introduced in \cite{Rosen}. Rosenbrock's definition coincides with Definition \ref{def_minimalpolsysmatsubset} when $\Sigma =\F.$

 \begin{rem} \label{rem:convention} \rm For convenience, if $n=0$ in \eqref{eq:polsysmat}, we adopt the agreement that $P(\la)$ is minimal at every point $\la_{0}\in\F.$
 \end{rem}

 In the next example, we illustrate Definition \ref{def_minimalpolsysmatsubset} with a rational matrix and a polynomial system matrix taken from the recent reference \cite{Saad} dealing with numerical algorithms for solving NLEPs via rational approximation. We advance that we will use the matrices in Example \ref{ex_saad} several times for illustrating different concepts introduced in this paper as well as for establishing a first connection between the theory developed in this paper and NLEPs. In this respect, we emphasize that \cite{Saad} does not mention at all polynomial system matrices, and that the same happens with references \cite{nlep,su-bai-2011}.

 \begin{example}\label{ex_saad}\rm Let $G(\la)$ be a rational matrix of the form
 		\begin{equation}\label{rational_saad}
G (\la) = -B_0 + \la A_0  + \frac{B_1}{\la-\sigma_1} + \cdots + \frac{B_s}{\la-\sigma_s} \in \C(\la)^{p\times p},
 		\end{equation}
 	with $A_0, B_0,\ldots,B_s \in \mathbb{C}^{p \times p},$ $\sigma_1,\ldots,\sigma_s\in\C,$ and $\sigma_i \ne \sigma_j$ if $i \ne j$. Let us consider the linear polynomial matrix
 		$$P(\la) =
 	\left[
 	\begin{array}{cccc|c}
 	(\la - \sigma_1) I & & & & I \\
 	&(\la - \sigma_2) I & & & I \\
 	& & \ddots & & \vdots \\
 	&& &(\la - \sigma_s) I & I \\ \hline
 	-B_1 & -B_2 &\cdots & -B_s & \la A_0 - B_0\\
 	\end{array}
 	\right].$$These matrices are introduced in \cite{Saad} to tackle a NLEP $T(\la) v  = 0,$ in a certain region $\Omega \subseteq \mathbb{C},$ where the matrix $T(\la)$ is of the form
 	$
 		T(\la) = -B_0 + \la A_0 + f_1(\la) A_1 + \cdots + f_q(\la) A_q,
 	$
 		with $A_0,A_1,\ldots,A_q \in \mathbb{C}^{p \times p}$ and
 		$f_i : \Omega \subseteq \mathbb{C} \longrightarrow \mathbb{C},$ $i=1,\ldots , q,$ being scalar functions nonlinear in the variable $\la$ and holomorphic in $\Omega.$ For solving a NLEP of this form, the nonlinear matrix $T(\la)$ is approximated in $\Omega$ by a rational matrix $G(\la)$ as in \eqref{rational_saad}, and $P(\la)$ is considered to linearize $G(\la).$ It is easy to see that $P(\la)$ is, in fact, a linear polynomial system matrix of $G(\la),$ by setting the matrix $\diag((\la-\sigma_1)I,\ldots,(\la-\sigma_s)I)$ as state matrix $A(\la)$ in \eqref{eq:polsysmat}. Moreover, without any assumption, $P(\la)$ is minimal in $\Sigma:=\mathbb{C} \setminus \{\sigma_1 , \ldots , \sigma_s \}.$ In particular, and according to \cite{Saad}, $\Omega$ is a subset of $\Sigma.$ Therefore, $P(\la)$ is minimal in the target set $\Omega.$ For completeness, notice that a polynomial system matrix as $P(\la)$ is minimal in $\C$ if and only if all the matrices $B_1, \ldots, B_s$ are nonsingular. We also emphasize that the form of the rational matrix $G(\la)$ in \eqref{rational_saad} is very particular because it is the sum of a linear polynomial matrix and strictly proper rational matrices with linear denominators, which simplifies considerably working with it from different perspectives. We will consider later more complicated examples.
 \end{example}

The next result provides the pole and zero elementary divisors of a rational matrix $G(\la)$ at any finite point $\la_0\in\efe$ from any polynomial system matrix of $G(\la)$ minimal at $\la_0.$ This result is the counterpart of \cite[Chapter 3, Theorem 4.1]{Rosen} for polynomial system matrices minimal at a finite point instead of polynomial system matrices of least order.

\begin{theo}\label{th:Rosenlocal}
Let $\la_0\in\efe.$ Let $G(\la)\in\FF(\la)^{p\times m}$  and let
	\begin{equation}\label{eq.Pdelambda}
	P(\la)=\begin{bmatrix}
	A(\la) & B(\la)\\-C(\la) &D(\la)
	\end{bmatrix}\in\FF[\la]^{(n+p)\times (n+m)}
	\end{equation}
	be a polynomial system matrix minimal at $\la_0$ whose transfer
	function matrix is $G(\la).$ Then the elementary divisors of $A(\la)$ at $\la_0$ are the pole elementary divisors of $G(\la)$ at $\la_0,$ and the elementary divisors of $P(\la)$ at $\la_0$ are the zero elementary divisors of $G(\la)$ at $\la_0.$
	\end{theo}
	
	\begin{proof}
		Let us consider the Smith normal form of $\begin{bmatrix}
		A(\la) & B(\la) \end{bmatrix}.$ Namely, $$U(\la) \begin{bmatrix}
		A(\la) & B(\la) \end{bmatrix} V(\la)= \begin{bmatrix}
		S(\la) & 0 \end{bmatrix},$$ with $U(\la)$ and $V(\la)$ unimodular matrices. Observe that $S(\la)\in\F[\la]^{n\times n}$ is invertible as a rational matrix since $\rank\begin{bmatrix} A(\la) & B(\la) \end{bmatrix}=n.$ We set $H_{1}(\la):=S(\la)^{-1}U(\la).$
		Since $P(\la)$ is minimal at $\la_0,$ $S(\la)$ has no zeros at $\la_0.$ Therefore, $H_{1}(\la)$ is regular at $\la_0.$ Moreover, $\begin{bmatrix}
			H_{1}(\la)A(\la) & H_1(\la)B(\la) \end{bmatrix}$ is a polynomial matrix, as it is equal to $\begin{bmatrix}
				I_n & 0 \end{bmatrix}V(\la)^{-1},$ has full row rank, and has no zeros in $\F.$ Now, let us consider the Smith normal form of the polynomial matrix $\begin{bmatrix}
		H_{1}(\la)A(\la) \\-C(\la)
		\end{bmatrix}.$ Namely, $$\widetilde{U}(\la)\begin{bmatrix}
		H_{1}(\la)A(\la) \\-C(\la)
		\end{bmatrix}\widetilde{V}(\la)=\begin{bmatrix}
		\widetilde{S}(\la) \\0
		\end{bmatrix},$$ with $\widetilde U(\la)$ and $\widetilde V(\la)$ unimodular matrices. Observe that $\widetilde S(\la)\in\F[\la]^{n\times n}$ is invertible as a rational matrix since $H_{1}(\la)$ is invertible and $\rank\begin{bmatrix} A(\la) \\ C(\la)\end{bmatrix}=n.$  We set $H_{2}(\la):=\widetilde{V}(\la)\widetilde{S}(\la)^{-1}.$ Moreover, the matrix $\begin{bmatrix}
		H_{1}(\la)A(\la)H_{2}(\la) \\-C(\la)H_{2}(\la)
		\end{bmatrix}$ is also polynomial, as it is equal to $\widetilde{U}(\la)^{-1}\begin{bmatrix}
			I_n \\0
		\end{bmatrix}$, has full column rank, and has no zeros in $\F.$ Since $P(\la)$ is minimal at $\la_0$ and $H_1(\la)$ is regular at $\la_0,$ $\widetilde{S}(\la)$ has not zeros at $\la_0.$ Therefore, $H_2(\la)$ is regular at $\la_0.$ Let us define now the polynomial system matrix
			\begin{equation*}
			\widetilde{P}(\la):=\begin{bmatrix}
			H_1(\la) & 0\\0 & I_{p}
			\end{bmatrix}\begin{bmatrix}
			A(\la) & B(\la)\\-C(\la) &D(\la)
			\end{bmatrix}\begin{bmatrix}
			H_2(\la) & 0\\0 & I_{m}
			\end{bmatrix}=\begin{bmatrix}
			H_1(\la)A(\la)H_2(\la) & H_1(\la)B(\la)\\-C(\la)H_2(\la) &D(\la)
			\end{bmatrix}.
			\end{equation*}
		We claim that $\widetilde{P}(\la)$ is a minimal polynomial system matrix in $\F$ or in the classical sense of Rosenbrock \cite{Rosen}. For that, it remains to prove that the matrix $$Z(\la):=\begin{bmatrix}
		H_{1}(\la)A(\la)H_{2}(\la)  &  H_{1}(\la)B(\la)
		\end{bmatrix}$$ has full row rank for all $\la\in\F.$ Let us suppose that there exists $\la_1\in\F$ such that $\rank Z(\la_1) < n.$ On the one hand,
		we know that $$\rank\begin{bmatrix}
		H_{1}(\la_1)A(\la_1)\widetilde{V}(\la_1) & H_{1}(\la_1)B(\la_1)
		\end{bmatrix} = n,$$ since the Smith normal form of $\begin{bmatrix}
		H_{1}(\la)A(\la) & H_1(\la)B(\la) \end{bmatrix}$ is equal to $\begin{bmatrix}
		I_n & 0 \end{bmatrix}$ and $\widetilde{V}(\la)$ is unimodular. On the other hand, we have that
		$$\rank\begin{bmatrix}
		H_{1}(\la_1)A(\la_1)\widetilde{V}(\la_1) & H_{1}(\la_1)B(\la_1)
		\end{bmatrix} =\rank\left( Z(\la_1) \begin{bmatrix}
		\widetilde{S}(\la_1) & 0\\0 & I_{m}
		\end{bmatrix}\right)\leq \rank Z(\la_1) < n,$$
		which is a contradiction. Therefore, $\widetilde{P}(\la)$ is a minimal polynomial system matrix. Its transfer
		function matrix is $G(\la).$ Then, by \cite[Chapter 3, Theorem 4.1]{Rosen}, we know that the zero elementary divisors of $G(\la)$ are the elementary divisors of $\widetilde{P}(\la),$ and that the pole elementary divisors of $G(\la)$ are the elementary divisors of $H_1(\la)A(\la)H_2(\la).$ Finally, the result follows by taking into account that the matrices $P(\la)$ and $\widetilde{P}(\la)$ are equivalent at $\la_0,$ and that the matrices $A(\la)$ and $H_1(\la)A(\la)H_2(\la)$ are also equivalent at $\la_0,$ since $H_{1}(\la)$ and $H_{2}(\la)$ are both regular at that point.
	\end{proof}
	
	Theorem \ref{th:Rosenlocal} can be extended to any subset of $\F$ in a natural way, by applying this theorem to every point of that subset.
	
	\begin{theo}\label{th:Rosenlocalsubset}
		Let $\Sigma\subseteq\efe$ be nonempty. Let $G(\la)\in\FF(\la)^{p\times m}$ and let
		\begin{equation*}\label{eq.Pdelambda2}
		P(\la)=\begin{bmatrix}
		A(\la) & B(\la)\\-C(\la) &D(\la)
		\end{bmatrix}\in\FF[\la]^{(n+p)\times (n+m)}
		\end{equation*}
		be a polynomial system matrix minimal in $\Sigma$ whose transfer
		function matrix is $G(\la).$ Then the elementary divisors of $A(\la)$ in $\Sigma$ are the pole elementary divisors of $G(\la)$ in $\Sigma,$ and the elementary divisors of $P(\la)$ in $\Sigma$ are the zero elementary divisors of $G(\la)$ in $\Sigma.$
	\end{theo}
	
	\begin{example}\rm If Theorem \ref{th:Rosenlocalsubset} is applied to the matrices $G(\la)$ and $P(\la)$ and the set $\Sigma $ in Example \ref{ex_saad}, we obtain immediately that (without any hypothesis) the eigenvalues of $P(\la)$ in $\Sigma$ coincide exactly with the zeros of $G(\la)$ in $\Sigma$, with exactly the same multiplicities (geometric, algebraic and partial). Observe also that all the zeros of $G(\la)$ in $\Sigma$ are, in fact, eigenvalues of $G(\la)$ because the only potential poles of $G(\la)$ are $\sigma_1,\ldots , \sigma_s$. This result is stronger than Lemma 3.1 and Corollary 3.2 in \cite{Saad} from two perspectives: \cite{Saad} deals with determinants and, so, only gives information on algebraic multiplicities, and the requests in \cite{Saad} impose the additional hypothesis that $A_0$ is nonsingular. Note that, under the assumption that all the matrices $B_1,\ldots,B_s$ are nonsingular, we obtain that $P(\la)$ (and $A(\la)$) allows us to obtain the complete information on finite poles and zeros (including all the multiplicities) of $G(\la)$ in $\C.$
		
	\end{example}
	
	\subsection{Polynomial system matrices minimal at infinity}
	 Theorems \ref{th:Rosenlocal} and \ref{th:Rosenlocalsubset} characterize polynomial system matrices that contain the information of the invariant orders at finite points of their transfer functions. The extension of these results for including the information at infinity is an old problem that has been considered in classical papers as, for instance, in \cite{Ver81,VVK79}. However, a satisfactory solution has been found, so far, only for polynomial system matrices with state matrix $A(\la)$ being a linear polynomial matrix and the other blocks $B(\la),$ $C(\la),$ $D(\la)$ being constant matrices. In other cases, recovering the information at infinity requires to embed the polynomial system matrix into a larger matrix. In this section, we propose a new approach for obtaining a counterpart of Theorem \ref{th:Rosenlocal} at infinity. This approach is motivated by the recent work \cite{strong}, but presents relevant differences with respect to \cite{strong}, and is based on the use of ``reversals'' and local equivalences of rational matrices.
	
	In order to develop our counterpart of Theorem \ref{th:Rosenlocal} at infinity, first, we introduce the notion of $g$-reversal of a rational matrix in Definition \ref{reversal}, where $g$ is any integer. In this definition we will use, for a particular value of $g,$ the well-known fact that any rational matrix $G(\la)\in\FF(\la)^{p\times m}$ can be uniquely written as
	\begin{equation}\label{eq.polspdec}
	G(\la)=Q(\la)+G_{sp}(\la)
	\end{equation}
	where $Q(\la)\in\FF[\la]^{p\times m}$ is a polynomial matrix and $G_{sp}(\la)\in \F(\lambda)^{p\times m}$ is a
	strictly proper rational matrix. The equation \eqref{eq.polspdec} follows from the Euclidean division for polynomials applied to each entry of $G(\la).$ The matrices $Q(\la)$ and $G_{sp} (\la)$ are called the polynomial part and the strictly proper part of $G(\la)$, respectively. A polynomial matrix $Q(\la)$ is said to have degree $d$ if $d$ is the largest exponent of the variable $\la$ of its entries with nonzero coefficient. In such a case, $d$ is denoted by $\deg (Q(\lambda)).$

	\begin{deff}[$g$-reversal of a rational matrix]\label{reversal}
		Let $G(\lambda)\in\F(\lambda)^{p\times m}$ be a rational matrix, and let $g$ be an integer. We define the $g$-reversal of $G(\la)$ as the rational matrix
		\begin{equation*}
		\rev_{g} G(\lambda):=\lambda^{g}G\left(\dfrac{1}{\lambda}\right).
		\end{equation*}
		Let $G(\lambda)$ be expressed as in \eqref{eq.polspdec}. If $g= \deg (Q(\lambda))$ whenever $G(\la)$ is not strictly proper, or $g=0$ if $G(\la)$ is strictly proper, then the $g$-reversal is called the \textit{reversal of $G(\la)$} and it is often denoted by just $\rev G(\la).$ 	
	\end{deff}

	Note that if $Q(\la)$ in \eqref{eq.polspdec} is a constant matrix, including the zero matrix, then $\rev G(\lambda)=G\left(1/\lambda\right)$. Definition \ref{reversal} extends the definition of $g$-reversal for polynomial matrices (see, for instance, \cite[Definition 2.12]{spectral}). However, we emphasize that in the definition of $g$-reversal of a polynomial matrix considered previously in the literature, $g$ is always taken larger than or equal to the degree of the polynomial matrix, while in Definition \ref{reversal} we only ask for $g$ to be an integer.
	
	Given a polynomial system matrix $P(\la)$ as in \eqref{eq:polsysmat}, we have that
	\begin{equation*}
	\rev P(\la)=\begin{bmatrix}
	\rev_d A(\la) & \rev_d B(\la)\\-\rev_d C(\la) &\rev_d D(\la)
	\end{bmatrix},
	\end{equation*}
	where $d$ is the degree of $P(\la),$ is also a polynomial matrix. Moreover, $\rev_d A(\la)$ is nonsingular since $A(\la)$ is nonsingular. Therefore, $\rev P(\la)$ is also a polynomial system matrix. We now introduce Definition \ref{def:mininfinity} about minimality at infinity of a polynomial system matrix.
	
	\begin{deff}[Polynomial system matrix minimal at infinity]\label{def:mininfinity} The polynomial system matrix $P(\la)$ in \eqref{eq:polsysmat} is minimal at $\infty$ if $\rev P(\la)$ is minimal at $0.$
	\end{deff}

\begin{example}\label{ex_saad_inf}\rm The polynomial system matrix $P(\la)$ with transfer function matrix $G(\la)$ in Example \ref{ex_saad} is minimal at $\infty$ since	{\small	  \[
\mbox{rev} P(\la) =
		\left[
		\begin{array}{cccc|c}
		(1 - \la \sigma_1) I & & & & \la I \\
		&(1 - \la \sigma_2) I & & & \la I \\
		& & \ddots & & \vdots \\
		&& &(1 - \la \sigma_s) I & \la I \\ \hline
		-\la B_1 & -\la B_2 &\cdots & -\la B_s & A_0 - \la B_0\\
		\end{array}
		\right]
	\]}
	is, obviously, minimal at $0.$
	
\end{example}

	\begin{rem}\label{rem:minimalinf}\rm A polynomial system matrix $P(\la)$ as in \eqref{eq:polsysmat}, with $\deg (P(\la))=d$ and $n>0,$ is minimal at $\infty$ if and only if
		$$\rank\begin{bmatrix} \rev_d A(0) \\ \rev_d C(0)\end{bmatrix}=\rank\begin{bmatrix}  \rev_d A(0) & \rev_d B(0) \end{bmatrix}=n.$$ More precisely, let $A_d,$ $B_d,$ $C_d$ and $D_d$ be the matrix coefficients of $\la^d$ in $A(\la),$ $B(\la),$ $C(\la)$ and $D(\la),$ respectively. Then the fact of $P(\la)$ being minimal at $\infty$ is equivalent to
		 $$\rank\begin{bmatrix} A_d\\ C_d\end{bmatrix}=\rank\begin{bmatrix} A_d & B_d \end{bmatrix}=n.$$
		Notice that if $d=0$ then $P(\la)$ is a constant polynomial system matrix, and $A_0$ must be invertible. Therefore, in this case, the rank condition above is automatically satisfied, and $P(\la)$ is minimal at $\infty.$
	\end{rem}
	Theorem \ref{th:Rosenlocalinf} is essentially the counterpart of Theorem \ref{th:Rosenlocal} at infinity. We state it in terms of reversals and their elementary divisors at $0$ as we only have defined elementary divisors for finite points. The implications of Theorem \ref{th:Rosenlocalinf} on the structure at infinity are made explicit in Theorem \ref{th:Rosenlocalinf2}.
	
		\begin{theo}\label{th:Rosenlocalinf}
		Let $G(\la)\in\FF(\la)^{p\times m}$ and let
		\begin{equation*}
		P(\la)=\begin{bmatrix}
		A(\la) & B(\la)\\-C(\la) &D(\la)
		\end{bmatrix}\in\FF[\la]^{(n+p)\times (n+m)}
		\end{equation*}
		be a polynomial system matrix of degree $d$ minimal at $\infty$ whose transfer
		function matrix is $G(\la).$ Then the elementary divisors of $\rev_d A(\la)$ at $0$ are the pole elementary divisors of $\rev_d G(\la)$ at $0,$ and the elementary divisors of $\rev P(\la)$ at $0$ are the zero elementary divisors of $\rev_ d G(\la)$ at $0.$
	\end{theo}
	\begin{proof}It can be easily proved that the transfer function matrix of $\rev P(\la)$ is $\rev_d G(\la).$ The theorem then follows by applying Theorem \ref{th:Rosenlocal}, since $\rev P(\la)$ is minimal at $0.$ \end{proof}
	Once we have obtained the elementary divisors of the $d$-reversal of a rational matrix at $0$, from one of its polynomial system matrices of degree $d$ minimal at $\infty,$ we can then obtain its invariant orders at infinity as we state in Theorem \ref{th:Rosenlocalinf2}. For proving that, we use Lemma \ref{rem:invariantorders}.
	\begin{lem}\label{rem:invariantorders}\rm Let $G(\la)\in\F(\la)^{p\times m}$ with $\rank G(\la)=r,$ and let $g$ be an integer. Let $e_{1},\ldots,e_{r}$ be the invariant orders of $\rev_g G(\la)$ at $0,$ and let $q_{1},\ldots,q_{r}$ be the invariant orders at infinity of $G(\la).$ Then
		\begin{equation}
		e_{i}=q_{i}+g \quad i=1,\ldots,r.
		\end{equation}	
	\end{lem}

	\begin{proof}
		From the local Smith--McMillan form at infinity of $G(\la),$ there exist biproper rational matrices $B_{1}(\la)$ and $B_{2}(\la)$ such that
		$$  G(\la)=B_{1}(\la)\diag\left((1/\la)^{q_1},\allowbreak\ldots,(1/\la)^{q_r},0_{(p-r)\times (m-r)}\right)  B_{2}(\la).$$ Let us perform the transformation $\la \longmapsto 1/\la$ on the variable of the equation above. Thus,
		$$G(1/\la)=B_{1}(1/\la)\diag\left(\la^{q_1},\ldots,\la^{q_r},0_{(p-r)\times (m-r)}\right)  B_{2}(1/\la).$$
		By \cite[Lemma 6.9]{AmMaZa15}, $B_{1}(1/\la)$ and $B_{2}(1/\la)$ are regular at $0.$ We now multiply the previous equation by $\la^g,$ and we get that $q_{i}+g$ for $i=1,\ldots,r$  are the invariant orders of $\rev_g G(\la)$ at $0.$
	\end{proof}

		\begin{theo}\label{th:Rosenlocalinf2}
		Let $G(\la)\in\FF(\la)^{p\times m}$ with $\rank G(\la)=r$ and let
			\begin{equation*}\label{eq.Pdelambda3}
			P(\la)=\begin{bmatrix}
			A(\la) & B(\la)\\-C(\la) &D(\la)
			\end{bmatrix}\in\FF[\la]^{(n+p)\times (n+m)}
			\end{equation*}
			be a polynomial system matrix of degree $d$ minimal at $\infty$ whose transfer
			function matrix is $G(\la).$ Let $e_1\leq\cdots\leq e_s$ be the partial multiplicities of $\rev_d A(\la)$ at $0$ and let $\widetilde{e}_1\leq\cdots\leq\widetilde{e}_u$ be the partial multiplicities of $\rev P(\la)$ at $0.$ Then the invariant orders at infinity $q_{1}\leq q_{2}\leq\cdots\leq q_{r}$ of $G(\la)$ are
			 \begin{equation*}
			 (q_{1}, q_{2},\ldots , q_{r})=(-e_{s},-e_{s-1},\ldots ,-e_1,\underbrace{0,\ldots,0}_{r-s-u},\widetilde{e}_{1}, \widetilde{e}_{2},\ldots, \widetilde{e}_{u}) - (d,d,\ldots, d).
			 \end{equation*}
		\end{theo}

	\begin{proof}By Theorem \ref{th:Rosenlocalinf}, we know that $e_{i}$ and $\widetilde{e}_{j}$ with $i=1,\ldots,s$ and $j=1,\ldots,u$ are the pole and zero partial multiplicities of $\rev_d G(\la)$ at $0,$ respectively. Thus, the invariant orders of $\rev_d G(\la)$ at $0$ are $-e_s\leq-e_{s-1}\leq\cdots\leq -e_1 < \underbrace{0=\cdots=0}_{r-s-u} <\widetilde{e}_1\leq\cdots\leq\widetilde{e}_u.$ Then the use of Lemma \ref{rem:invariantorders} completes the proof.
	\end{proof}
	
	\begin{example}\rm
		By combining Theorem \ref{th:Rosenlocalinf2} and Example \ref{ex_saad_inf}, we see that $P(\la),$ in Example \ref{ex_saad}, contains the complete information about the invariant orders at $\infty $ of $G(\la)$ (without imposing any hypothesis). Note that, in this case, $d=1$ and that the $1$-reversal of the state matrix, i.e., $\rev_1 A(\la)=\diag ((1-\la\sigma_1)I,\ldots, (1-\la\sigma_s)I )$, has no partial multiplicities at $0$. This result on the relationship between the infinite structure of $G(\la)$ and the reversal of $P(\la)$ is not mentioned in \cite{Saad}. In this context, it is worth emphasizing that modern references on NLEPs and their rational approximations do not pay attention to the structure at $\infty $, while such structure plays an important role in many classic references of linear system theory and control \cite{Kailath, KaVo,Ver81,VVK79}.	\end{example}
	
	For polynomial system matrices that are minimal at infinity and, also, at every finite point, we state Definition \ref{def:strongly} about strong minimality. This definition has already been introduced in \cite[Definition 3.3]{linearsystems}. However, in \cite{linearsystems} the definition is given in terms of eigenvalues instead of minimality at every point, but both definitions are equivalent.
	
	\begin{deff}[Strongly minimal polynomial system matrix]\label{def:strongly} The polynomial system matrix $P(\la)$ in \eqref{eq:polsysmat} is strongly minimal if it is minimal at each point of $\F\cup\{\infty\}.$
	\end{deff}

We emphasize that, as a consequence of Theorems \ref{th:Rosenlocal} and \ref{th:Rosenlocalinf2}, strongly minimal polynomial system matrices contain all the information about the invariant orders of their transfer function matrices, both at finite points and at infinity.

\section{Local linearizations of rational matrices} \label{sec_local}

 In practice, one is often interested in studying the pole and zero structure of rational matrices not in the whole space $\F\cup \{\infty\}$  but in a particular region (see \cite{guttel-tisseur-2017,nlep,automatic,Saad}). For instance, this happens when a rational eigenvalue problem arises from approximating a nonlinear eigenvalue problem, since the approximation is usually reliable only in a target region not containing poles. As a consequence, the eigenvalues (those zeros that are not poles) of the approximating rational eigenvalue problem need to be computed only in that region. In this scenario, one can use local linearizations of the corresponding rational matrix which contain the information about the poles and zeros in the target region, but might not in the whole space $\F\cup\{\infty\}.$ In addition, they do not satisfy, in general, the conditions of the strong linearizations of rational matrices introduced in \cite{strong}. Thereby local linearizations provide extra flexibility in solving nonlinear eigenvalue problems.

 In this section, we give separately the definitions of linearizations of rational matrices in subsets of $\F$ and at infinity, study their properties and establish connections with the linearizations introduced in \cite{strong}. These linearizations will be useful in order to study the pole and zero structure of rational matrices in different sets containing or not infinity. In particular, and as an application of these definitions, we will study in Section \ref{guttel} the structure of the linearizations that appear in \cite{nlep}.

\subsection{ Linearizations in subsets of $\efe$}
In this subsection we introduce the definition of linearization of a rational matrix in a set not containing infinity and study some of its properties. We start by giving the definition of linearization at a finite point.
\begin{deff}[Linearization at a point in $\efe$]\label{def_pointstronglin}
	Let $G(\la) \in\F(\la)^{p\times m}$ and let $\la_{0}\in\efe.$ Let
	\begin{equation}\label{eq_lin}
	\mathcal{L}(\la)=\left[\begin{array}{cc}
	A_1 \la +A_0 &B_1 \la +B_0\\-(C_1 \la +C_0)&D_1 \la +D_0
	\end{array}\right]\in\F[\la]^{(n+q)\times (n+r)}
	\end{equation} be a linear polynomial system matrix and let
	$$\wh{G}(\la)=(D_1\la+D_0)+(C_1\la+C_0)(
	A_1\la+A_0)^{-1}(B_1\la+B_0)\in\F(\la)^{q\times r}$$ be its transfer function matrix. $\mathcal{L}(\la)$ is a linearization of $G(\la)$ at $\la_{0}$ if the following conditions hold:
	\begin{itemize}
			\item [\rm(a)] $\mathcal{L}(\la)$ is minimal at $\la_{0}$, and
			\item[\rm(b)]there exist nonnegative integers $s_1,s_2$ satisfying $s_1-s_2=q-p=r-m,$ and rational matrices $R_1(\la)\in\F(\la)^{(p+s_1)\times (p+s_1)}$ and
			$R_2(\la)\in\F(\la)^{(m+s_1)\times (m+s_1)}$ regular at $\la_{0}$ such that
			\begin{equation}\label{equivalencia}
R_1(\la)\diag(G(\la),I_{s_1})R_2(\la)=\diag(\wh{G}(\la),I_{s_2}).
			\end{equation}
	\end{itemize}
	
\end{deff}

\begin{rem}\label{extremecases} \rm  Notice that, in Definition \ref{def_pointstronglin}, the following two cases are allowed:
		\begin{enumerate}
			\item {$\wh{G}(\la) = G(\la)$}. Then we just have to check condition $\rm(a)$, since condition $\rm(b)$ is satisfied by setting $R_1(\la) = I_p$, $R_2(\la) = I_m,$ and $s_1=s_2=0$.
			\item $n=0.$ Then it is not necessary to take into account condition $\rm(a)$ (it is automatically satisfied by the agreement in Remark \ref{rem:convention}) and, therefore, we just have to check condition $\rm(b)$ with $\widehat{G} (\la) = D_1 \la + D_0 = \mathcal{L}(\la)$.
		\end{enumerate}
	We remark these extreme cases since they are important in applications, and make Definition \ref{def_pointstronglin} very general.
\end{rem}

We now extend, in a natural way, the notion of linearization at a finite point to linearization in subsets of $\F.$

\begin{deff}[Linearization in a subset of $\F$]\label{def:linsubset} Let $G(\la) \in\F(\la)^{p\times m}$ and let $\Sigma\subseteq\efe$ be nonempty. A linear polynomial system matrix
	$\mathcal{L}(\la)$ is a linearization of $G(\la)$ in $\Sigma$ if $\mathcal{L}(\la)$ is a linearization of $G(\la)$ at each point $\la_{0}\in\Sigma.$
\end{deff}

Since linearizations of rational matrices are, in particular, polynomial system matrices, their definition includes a specific partition. Thus, a fixed linear polynomial matrix (also called a matrix pencil) may be partitioned in different ways giving rise to different linearizations of the same or of different rational matrices, or in different subsets. To deal with different partitions, we will use expressions as ``$\mathcal{L} (\la)$ is a linearization of $G(\la)$ in $\Sigma$ with state matrix $A_1 \la + A_0$'' when it is necessary for avoiding any ambiguity. The expression ``$\mathcal{L} (\la)$ is a linearization of $G(\la)$ in $\Sigma$ with empty state matrix'' will cover the case $n=0$ in \eqref{eq_lin}.

In  condition \eqref{equivalencia}, one can always take $s_1=0$ or $s_2=0,$ according to $p\geq q$ and $m\geq r$ or $ q\geq p$ and $r\geq m$, respectively. This is a consequence of the local Smith--McMillan forms of $\diag(G(\la),I_{s_1})$ and $\diag(\wh{G}(\la),I_{s_2})$ being equivalent to each other at $\la_0$. In the rest of the results of this subsection, we will consider  $s:=s_1\geq 0$ and $s_2=0,$ since it corresponds to the most interesting situation in applications.

\begin{rem} \rm  If we have a linearization of $G(\la)$ in a set $\Sigma$ then, for each point $\mu\in\Sigma$, there exist rational matrices $R_1^\mu(\la)$ and
	$R_2^\mu(\la)$ regular at $\mu$ such that
	$R_1^\mu(\la)\diag(G(\la),I_{s})R_2^\mu(\la)\allowbreak=\wh{G}(\la).$ In principle, for different values of $\mu\in\Sigma,$ the rational matrices $R_1^\mu(\la)$ (respectively, $R_2^\mu(\la)$) may be different from each other, that is, $R_1^\mu(\la)$ (resp., $R_2^\mu(\la)$) depends on $\mu.$ However, Proposition \ref{local property} implies that the existence of $R_1^\mu(\la)$ and $R_2^\mu(\la)$ for each $\mu\in\Sigma$ is equivalent to the existence of two rational matrices $R_1(\la)$ and
	$R_2(\la)$ both regular in $\Sigma$ (and independent of $\mu$) such that
	$R_1(\la)\diag(G(\la),I_{s})R_2(\la)=\wh{G}(\la)$.
\end{rem}

\begin{rem}\label{equiv_finitepoins} \rm When $\Sigma=\efe,$ in Definition \ref{def:linsubset}, condition \eqref{equivalencia} is satisfied with $R_{1}(\la)$ and $R_{2}(\la)$ unimodular matrices. Therefore, a linearization in $\efe,$ or at every point of $\efe,$ is a linearization in the sense of \cite[Definition 3.2]{strong} and vice versa.
\end{rem}

 The next result gives the relation between the invariant orders at a finite point of a rational matrix $G(\la)$ and those of a rational matrix of the form $\diag(G(\la),I_{s}),$ with $s>0$. It is motivated by equation \eqref{equivalencia}.

\begin{lem}\label{lem:partialmultiplicities} Let $G(\la)\in\FF(\la)^{p\times m}$, $\la_{0}\in\F$ and let $\nu_1\leq\cdots\leq \nu_k<0=\nu_{k+1}=\cdots=\nu_{u-1}<\nu_u\leq\cdots\leq \nu_r$ be the invariant orders of $G(\la)$ at $\la_{0}.$ Consider $\widetilde{G}(\la):=\diag(G(\la),I_s)$ with $s>0.$ Then the invariant orders of $\widetilde{G}(\la)$ at $\la_{0}$ are $\widetilde{\nu}_1\leq\cdots\leq \widetilde{\nu}_k<0=\widetilde{\nu}_{k+1}=\cdots=\widetilde{\nu}_{u+s-1}<\widetilde{\nu}_{u+s}\leq\cdots\leq \widetilde{\nu}_{r+s},$ where $\widetilde{\nu}_i=\nu_i$ for $i=1,\ldots,k,$ and $\widetilde{\nu}_{j+s}=\nu_{j}$ for $j={u,\ldots,r}.$
\end{lem}

\begin{proof} Let $M(\la):=\diag\left(
		(\la-\la_{0})^{\nu_{1}},\ldots, (\la-\la_{0})^{\nu_{r}},0_{(p-r)\times (m-r)}
\right)$ be the local Smith--McMillan form of $G(\la)$ at $\la_0.$ Then, $G(\la)=R_{1}(\la)M(\la)R_2(\la)$ for some rational matrices $R_{1}(\la)$ and $R_2(\la)$ regular at $\la_0.$ Moreover,
 $\widetilde{G}(\la)=\diag\left(R_{1}(\la),I_s\right)\diag\left(M(\la),I_s\right)\diag\allowbreak\left(R_{2}(\la),I_s\right).$ Therefore, since the matrices $\diag\left(R_{1}(\la),I_s\right)$ and $\diag\left(R_{2}(\la),I_s\right)$ are regular at $\la_0,$ the local Smith--McMillan form of $\widetilde{G}(\la)$ at $\la_0$ is $\diag\left(M(\la),I_s\right)$ up to a permutation.
\end{proof}

Corollary \ref{lemmaspec} and Theorem \ref{theo:spectral} follow from Lemma \ref{lem:partialmultiplicities}. These results state the spectral information that one can obtain from local linearizations. More precisely, Theorem \ref{theo:spectral} is a spectral characterization of local linearizations in the spirit of \cite[Theorem 3.10]{strong}.

\begin{cor}\label{lemmaspec}
	Let $G(\la)\in\FF(\la)^{p\times m}$, $\la_{0}\in\F$ and let
	\[
	\mathcal{L}(\la)=\begin{bmatrix}
	A_1 \la +A_0 &B_1 \la +B_0\\-(C_1 \la +C_0)&D_1 \la +D_0
	\end{bmatrix}\in\efe[\la]^{(n+(p+s))\times (n+(m+s))}
	\]
	be a linear polynomial system matrix minimal at $\la_{0}.$ Let $\wh{G}(\la)$ be the transfer function matrix of $\mathcal{L}(\la).$
	Then $\mathcal{L}(\la)$ is a linearization of $G(\la)$ at $\la_{0}$ if and only if the following two conditions hold:
	\begin{enumerate}
		\item[\rm (a)] $\mbox{\rm rank} \,
		 \wh{G}(\la)  = \mbox{\rm rank} \,  G(\la) + s$, and
		
		\item[\rm (b)]  $G(\la)$ and $\wh{G}(\la)$ have exactly the same pole and zero elementary divisors at $\la_{0}.$
	\end{enumerate}
\end{cor}

\begin{proof}
	If $\mathcal{L}(\la)$ is a linearization of $G(\la)$ at $\la_{0}$ then (a) and (b) are satisfied by Lemma \ref{lem:partialmultiplicities}, since $\diag(G(\la),I_s)$ and $\wh{G}(\la)$ are equivalent at $\la_0.$ For the converse, suppose that $\nu_1\leq\cdots\leq \nu_k<0=\nu_{k+1}=\cdots=\nu_{u-1}<\nu_u\leq\cdots\leq \nu_r$ are the invariant orders of $G(\la)$ at $\la_{0}.$ From (a) and (b), the Smith--McMillan form at $\la_0$ of $\wh{G}(\la)$ must be
	$
	\diag\left((\la-\la_0)^{\nu_1},\ldots,(\la-\la_0)^{\nu_{u-1}},I_s,(\la-\la_0)^{\nu_u},\ldots,(\la-\la_0)^{\nu_r},0_{(p-r)\times(m-r)}\right)
	$. Observe that this is also the Smith--McMillan form at $\la_0$ of $\diag(G(\la),I_s)$, as proved in the previous lemma. Thus, $\diag(G(\la),I_s)$ and $\wh{G}(\la)$ are equivalent at $\la_0$.
\end{proof}

\begin{theo}[Spectral characterization of linearizations at a point in  $\efe$]\label{theo:spectral}
	Let $G(\la)\in\FF(\la)^{p\times m}$, $\la_{0}\in\F$ and let
	\[
	\mathcal{L}(\la)=\begin{bmatrix}
	A_1 \la +A_0 &B_1 \la +B_0\\-(C_1 \la +C_0)&D_1 \la +D_0
	\end{bmatrix}\in\efe[\la]^{(n+(p+s))\times (n+(m+s))}
	\]
	be a linear polynomial system matrix minimal at $\la_{0}.$ Then $\mathcal{L}(\la)$ is a linearization of $G(\la)$ at $\la_0$ if and only if the following three conditions hold:
	\begin{enumerate}
		\item[\rm (a)] $\mbox{\rm rank} \,   \mathcal{L}(\la)  =\mbox{\rm rank} \,  G(\la)+n+s $,
		
		\item[\rm (b)]  the pole elementary divisors of $G(\la)$ at $\la_0$ are the elementary divisors of $A_{1}\la+A_{0}$ at $\la_0,$ and
		
		\item[\rm (c)]  the zero elementary divisors of $G(\la)$ at $\la_0$ are the elementary divisors of $\mathcal{L}(\la)$ at $\la_0.$
	\end{enumerate}
\end{theo}

\begin{proof}
	Let $\wh{G}(\la)$ be the transfer function matrix of $\mathcal{L}(\la)$. By \eqref{ranks_rel}, $\mbox{\rm rank} \,   \wh{G}(\la)  =\mbox{\rm rank} \, \mathcal{L}(\la)-n  $. Moreover, by Theorem \ref{th:Rosenlocal}, the pole elementary divisors of $\wh{G}(\la)$ at $\la_0$ are the elementary divisors of $A_1 \la +A_0$ at $\la_0$, and the zero elementary divisors of $\wh{G}(\la)$ at $\la_0$ are the elementary divisors of $\mathcal{L}(\la)$ at $\la_0$. The result follows from Corollary \ref{lemmaspec}.
\end{proof}

It is immediate to obtain counterparts of Corollary \ref{lemmaspec} and Theorem \ref{theo:spectral} for linear polynomial system matrices minimal in sets $\Sigma\subseteq \F$ and for linearizations in $\Sigma.$ We omit to state such results for brevity.

The following proposition is a straightforward consequence of the definition of linearization in a subset of $\efe$ by taking $s_1=s_2=0,$ $R_1(\la)=I_{p},$ $R_{2}(\la)=I_{ m}$ and $G(\la)=\widehat{G}(\la),$ i.e., it corresponds to case 1 in Remark \ref{extremecases}. However, we emphasize this result since it gives a sufficient condition that is easy to verify in order to ensure that a linear polynomial system  matrix is a linearization of a rational matrix.

\begin{prop}\label{sametransfer} 	Let $\Sigma\subseteq\efe$ be nonempty. Let
	\begin{equation}
	\mathcal{L}(\la)=\left[\begin{array}{cc}
	A_1 \la +A_0 &B_1 \la +B_0\\-(C_1 \la +C_0)&D_1 \la +D_0
	\end{array}\right]\in\F[\la]^{(n+q)\times (n+r)}
	\end{equation} be a linear polynomial system matrix and let $\wh{G}(\la)$ be its transfer function matrix. If $\mathcal{L}(\la)$ is minimal in $\Sigma$ then $\mathcal{L}(\la)$ is a linearization of $\wh{G}(\la)$ in $\Sigma.$
\end{prop}

In plain words, any linear polynomial system matrix $\mathcal{L}(\la)$ is a linearization of its transfer function matrix in the sets where $\mathcal{L}(\la)$ is minimal.

\begin{example}\rm  Consider the matrices $G(\la)$ and $P(\la)$ and the set $\Sigma$ in Example \ref{ex_saad}, that were originally introduced in \cite{Saad}. By combining the discussion in Example \ref{ex_saad} with Proposition \ref{sametransfer}, we immediately obtain that $P(\la)$ is a linearization of $G(\la)$ in $\Sigma.$ With a bit more effort, it is also easy to obtain the following stronger result: $P(\la)$ is a linearization of $G(\la)$ in $\C\setminus \Pi $ where $\Pi:=\{\sigma_i :  B_i \text{ is singular for }1\leq i \leq s\}.$
\end{example}

As mentioned in Example \ref{ex_saad}, the form of the rational matrix $G(\la)$ in \eqref{rational_saad} is very particular since its polynomial part and the denominators in the strictly proper part are linear. Thus, we finish this section by discussing in Example \ref{ex_subai} a rational matrix with non linear polynomial part and with a general state space realization of the strictly proper part. For such general representation of rational matrices, an influential companion-like pencil associated to it was introduced in \cite{su-bai-2011}. We will analyze this pencil from three different perspectives.

\begin{example}\label{ex_subai} \rm It is well-known that any rational matrix can be written in the form:
	\begin{equation}\label{rational_polpart}
	G(\la)=D_q\la^q+\cdots + D_1 \la + D_0 + C(\la I_n - A)^{-1}B\in\F(\la)^{p\times m}.
	\end{equation}
By assuming $D_q\neq 0$ with $q\geq 2,$ from the expression above we define the pencil
\begin{equation} \label{lin_subai}
L(\la) =
\left[
\begin{array}{ccccc}
\la D_q + D_{q-1} & D_{q-2} & \cdots & D_0 & -C \\
-I_m & \la I_m & 0 & \cdots & 0 \\

& \ddots & \ddots & \vdots & \vdots \\
&  & -I_m & \la I_m & 0 \\
&  &  & B & \la I_n - A
\end{array}
\right]\, .
\end{equation}
This pencil was introduced in \cite{su-bai-2011} for regular rational matrices and is a particular case of the pencils considered in \cite[Theorem 5.11]{strong} (modulo some permutations). In fact, \cite[Theorem 5.11]{strong} proves that if $L(\la)$ is considered as a polynomial system matrix with state matrix $\la I-A,$ and $L(\la)$ is minimal in $\F,$ then $L(\la)$ is a strong linearization of $G(\la)$ in the sense of \cite[Definition 3.4]{strong} (we will revise this in subsection \ref{comparison}). Thus, under these conditions, $L(\la)$ contains all the information about the poles and zeros of $G(\la).$

We now consider $L(\la)$ from other two points of view different from the one in \cite{strong}. They will correspond to the two extreme cases described in Remark \ref{extremecases}. First, we consider the following regular submatrix of $L(\la),$ obtained by removing the first block row and the penultimate block column:
\begin{equation}\label{ex_statematrix}
A(\la) :=
\left[
\begin{array}{ccccc}
-I_m & \la I_m & 0 & \cdots & 0 \\

& \ddots & \ddots & & \vdots  \\
 &  & -I_m & \la I_m & 0 \\
& &  & -I_m  & 0 \\
&  &   & & \la I_n - A
\end{array}
\right]\in\F[\la]^{((q-1)m + n)\times ((q-1)m + n ) }\, ,
\end{equation}
and we see $L(\la)$ as a polynomial system matrix with state matrix $A(\la).$ That is, once the state matrix is chosen, the other matrices in \eqref{eq:polsysmat} are $D(\la):= D_0,$ $ C(\la):=[-\la D_q - D_{q-1} \allowbreak \quad  -D_{q-2} \quad \cdots \quad -D_1 \quad C ] $ and $B(\la)^{T}:=[  0 \quad \cdots \quad 0 \quad \la I_m  \quad B^{T} ]^{T}.$ With such partition of $L(\la),$ it is easy to see that the transfer function matrix of $L(\la)$ is precisely $G(\la),$ i.e., $D(\la)+ C(\la)A(\la)^{-1}B(\la)=G(\la).$ For that, just take into account that the two last block columns of $A(\la)^{-1}$ are $[  -\la^{q-2} I_m \,\, \cdots \,\, -\la I_m \quad -I_m \quad 0]^{T}$ and $[ 0 \,\, \cdots \,\, 0 \quad (\la I_n -A)^{-T} ]^{T}$. Then, Proposition \ref{sametransfer} guarantees that, without any extra hypothesis, $L(\la)$ is a linearization of $G(\la)$ in $\Omega:=\F\setminus \Lambda, $ where $\Lambda:=\{\la\in\F:\la \text{ is an eigenvalue of }A\}.$ With a bit more effort, it is also easy to see that if $\rank\begin{bmatrix} \la_0 I_n - A \\ C\end{bmatrix}=\rank\begin{bmatrix} \la_0 I_n - A & B \end{bmatrix}=n,$ for all $\la_0\in\Lambda,$ then $L(\la)$ is minimal in $\F$ and, thus, is a linearization of $G(\la)$ in $\F$ with state matrix $A(\la)$ in \eqref{ex_statematrix}. Observe that, if we do not impose any hypothesis of minimality in $\Lambda ,$ and $L(\la)$ is just a linearization in $\Omega,$ then we can not guarantee that $L(\la)$ has any information about the poles of $G(\la)$ since they are necessarily contained in $\Lambda.$ Moreover, the set $\Lambda$ might contain eigenvalues of $G(\la).$ This is not a problem in REPs coming from approximating NLEPs \cite{nlep,automatic,Saad} because, in such cases, the target set is outside $\Lambda .$ However, it is in classical applications of rational matrices \cite{Kailath}.

The second point of view is to consider $L(\la)$ as a linearization of $G(\la)$ in $\Omega$ with empty state matrix. To this purpose, we define the following rational matrices regular at $\Omega$:
\begin{align}\label{eq_V}
V(\la) & :=
\left[
\begin{array}{ccccccc}
\la^{q-1}I_m & 0 & -I_m  \\
\la^{q-2}I_m & 0 & & -I_m  \\
\vdots & \vdots & & & \ddots \\
\la I_m & 0 & & & & & -I_m \\
I_m & 0 & & & & & 0 \\
-(\la I_n - A)^{-1}B & I_n & 0 & & \cdots & & 0
\end{array}
\right]\, . \\ \label{eq_U}
U(\la) & :=
\left[
\begin{array}{cccccc}
I_p & - C & -\la D_q - D_{q-1} & -D_{q-2} & \cdots & -D_1 \\
0 & 0 & I_m & -\la I_m \\
\vdots & \vdots & & \ddots & \ddots\\
0 & 0 & & & I_m & -\la I_m \\
0 & 0 &  & & & I_m \\
0 & \la I_n - A & 0 & \cdots & & 0
\end{array}
\right]\, .
\end{align}
Then, we check that $L(\la)V(\la)=U(\la)\diag (G(\la),I_n,I_{m(q-1)}),$
which means that $L(\la)$ and $\diag (G(\la),I_n,I_{m(q-1)})$ are equivalent in $\Omega$ and, so, that $L(\la)$ is a linearization of $G(\la)$ in $\Omega$ with empty state matrix (recall Remark \ref{extremecases}(2)).
	
\end{example}

The two approaches described in Example \ref{ex_subai} for viewing $L(\la)$ in \eqref{lin_subai} as a linearization of $G(\la)$ in $\Omega$ can be extended with more effort to many other of the pencils described in \cite[Theorem 5.11]{strong}. We postpone these developments to future research to keep this paper concise.

\subsection{Linearizations at infinity and in sets containing infinity}
 Our definition of linearization of a rational matrix at infinity is based on the notion of $g$-reversal of a rational matrix introduced in Definition \ref{reversal}.

\begin{deff}[Linearization at infinity of grade $g$]\label{def:infinity} Let $G(\la)\in\F(\la)^{p\times m}.$ Let
\begin{equation}\label{pencil_inf}
\mathcal{L}(\la)=\left[\begin{array}{cc}
A_1 \la +A_0 &B_1 \la +B_0\\-(C_1 \la +C_0)&D_1 \la +D_0
\end{array}\right]\in\F[\la]^{(n+q)\times (n+r)}
\end{equation} be a linear polynomial system matrix and let
$$\wh{G}(\la)=(D_1\la+D_0)+(C_1\la+C_0)(
A_1\la+A_0)^{-1}(B_1\la+B_0)\in\F(\la)^{q\times r}$$ be its transfer function matrix. Let $g$ be an integer. $\mathcal{L}(\la)$ is a linearization of $G(\la)$ at $\infty$ of grade $g$ if the following conditions hold:
\begin{itemize}
	\item[\rm (a)]$\rev \mathcal{L}(\la)$ is minimal at $0,$ and
	\item[\rm (b)]there exist nonnegative integers $s_1,s_2,$ with $s_1-s_2=q-p=r-m,$ and rational matrices
	$Q_1(\la)\in\F(\la)^{(p+s_1)\times (p+s_1)}$ and $Q_2(\la)\in\F(\la)^{(m+s_1)\times (m+s_1)}$ regular at $0$ such that
	\begin{equation}\label{eqinfty}
	Q_1(\la)\diag(\rev_{g} G(\la),I_{s_1})Q_2(\la)=\diag(\rev_{\ell} \wh{G}(\la),I_{s_2}),
	\end{equation}
	where $\ell=\deg (\mathcal{L}(\la)).$
\end{itemize}
\end{deff}
Observe that Definition \ref{def:infinity} allows, for completeness, the possibility of $\ell=\deg(\mathcal{L}(\la))$ being equal to $0.$ We admit that this case has a very limited interest in applications, since it corresponds to $\mathcal{L}(\la)$ and $\rev_{\ell} \wh{G}(\la)= \wh{G}(\la)$ being constant matrices. However, it includes linearizations at $\infty$ of rational matrices $G(\la)$ such that, for some integer $g,$ $\rev_{g} G(\la)$ has all its invariant orders at zero equal to zero. Moreover, notice that, in any case, $\rev \mathcal{L}(\la)$ is also a linear polynomial system matrix since $\rev_\ell (A_1\la+A_0)$ is nonsingular. We then have the following characterization of linearizations at infinity.
\begin{prop}\label{prop_lininf}
	 A linear polynomial system matrix $\mathcal{L}(\la)$ as in \eqref{pencil_inf} is a linearization of a rational matrix $G(\la)$ at $\infty$ of grade $g$ if and only if $\rev \mathcal{L}(\la)$ is a linearization of $\rev_{g} G(\la)$ at $0.$
\end{prop}
\begin{proof}
 The proposition follows from the fact that $\rev_{\ell} \wh{G}(\la)$ with $\ell=\deg (\mathcal{L}(\la))$ is the transfer function matrix of $\rev \mathcal{L}(\la).$ Then we make use of Definition \ref{def_pointstronglin}.
\end{proof}
Conditions $\rm (a)$ and $\rm (b)$ in Definition \ref{def:infinity} can be stated in a different way as we show in Remarks \ref{cond_a} and \ref{cond_b}, respectively.
\begin{rem}\label{cond_a} \rm As a particular case of what is discussed in Remark \ref{rem:minimalinf}, condition $(a)$ in Definition \ref{def:infinity} is equivalent to $$\rank\begin{bmatrix} A_1 \\ C_1 \end{bmatrix}=\rank\begin{bmatrix} A_1 & B_1 \end{bmatrix}=n,$$ if $\mathcal{L}(\la)$ is nonconstant, i.e., if $\ell=1.$ If $\mathcal{L}(\la)$ is constant, i.e., $\ell=0.$ condition $(a)$ is automatically satisfied since $\mathcal{L}(\la)$ is a polynomial system matrix and, therefore, $A_{0}$ is invertible.
\end{rem}

\begin{rem}\label{cond_b} \rm
By \cite[Lemma 6.9]{AmMaZa15}, a rational matrix $Q(\la)$ is regular at $0$ if and only if $Q(1/\la)$ is biproper. Therefore, condition $(b)$ in Definition \ref{def:infinity} is equivalent to the matrices $\diag((1/\la)^{g}G(\la),I_{s_1})$ and $\diag((1/\la)^{\ell}\wh{G}(\la),I_{s_2})$ being equivalent at infinity according to Definition \ref{def:equivalent}. More precisely, a linear polynomial system matrix $\mathcal{L}(\la)$ as in \eqref{pencil_inf} is a linearization of a rational matrix $G(\la)$ at $\infty$ of grade $g$ if and only if
\begin{itemize}
	\item[(a)]$ \mathcal{L}(\la)$ is minimal at $\infty,$ and
	\item[(b)]there exist nonnegative integers $s_1,s_2,$ with $s_1-s_2=q-p=r-m,$ and biproper matrices
	$B_1(\la)\in\F(\la)^{(p+s_1)\times (p+s_1)}$ and $B_2(\la)\in\F(\la)^{(m+s_1)\times (m+s_1)}$ such that
	\begin{equation}\label{eqinfty2}
	B_1(\la)\diag((1/\la)^{g}G(\la),I_{s_1})B_2(\la)=\diag((1/\la)^{\ell}\wh{G}(\la),I_{s_2}).
	\end{equation}
\end{itemize}
\end{rem}

We state in Theorem \ref{theo:spectralinf} a characterization of linearizations at infinity analogous to the one in Theorem \ref{theo:spectral} for linearizations at finite points. In this characterization, we consider the most usual situation $s_1:=s\geq 0$ and $s_2=0,$ assuming $q\geq p$ and $r\geq m.$ The proof of Theorem \ref{theo:spectralinf} is omitted since it follows immediately from Theorem \ref{theo:spectral} and Proposition \ref{prop_lininf}.

\begin{theo}[Spectral characterization of linearizations at infinity]\label{theo:spectralinf}
	Let $G(\la)\in\FF(\la)^{p\times m}$ and let
	\[
	\mathcal{L}(\la)=\begin{bmatrix}
	A_1 \la +A_0 &B_1 \la +B_0\\-(C_1 \la +C_0)&D_1 \la +D_0
	\end{bmatrix}\in\efe[\la]^{(n+(p+s))\times (n+(m+s))}
	\]
	be a linear polynomial system matrix such that $\rev\mathcal{L}(\la)$ is minimal at $0$ and let $\ell=\deg (\mathcal{L}(\la)).$ Then $\mathcal{L}(\la)$ is a linearization of $G(\la)$ at $\infty$ of grade $g$ if and only if the following three conditions hold:
	\begin{enumerate}
		\item[\rm (a)] $\mbox{\rm rank} \, \mathcal{L}(\la)=\mbox{\rm rank} \, G(\la)+n+s  $,
		
		\item[\rm (b)]  the pole elementary divisors of $\rev_g G(\la)$ at $0$ are the elementary divisors of $\rev_{\ell}(A_{1}\la+A_{0})$ at $0,$ and
			
		\item[\rm (c)]  the zero elementary divisors of $\rev_g G(\la)$ at $0$ are the elementary divisors of $\rev \mathcal{L}(\la)$ at $0.$
		
	\end{enumerate}
\end{theo}

Next, we study in Proposition \ref{prop:invariantorders} how to recover the invariant orders at infinity of rational matrices from linearizations at infinity of grade $g.$
\begin{prop} \label{prop:invariantorders} Let $G(\la)\in\FF(\la)^{p\times m}$ with $\rank G(\la)=r,$ and let
	\begin{equation*}
	\mathcal{L}(\la)=\left[\begin{array}{cc}
	A_1 \la +A_0 &B_1 \la +B_0\\-(C_1 \la +C_0)&D_1 \la +D_0
	\end{array}\right]\in\F[\la]^{(n+(p+s))\times (n+(m+s))}
	\end{equation*} be a linearization at infinity of grade $g$ of $G(\la)$ with $\ell=\deg (\mathcal{L}(\la)).$
Let $e_{1}\leq\cdots\leq e_{t}$ be the partial multiplicities of $\rev_\ell (A_1\la+A_0)$ at $0$, and let $\widetilde{e}_{1}\leq\cdots\leq\widetilde{e}_{u}$ be the partial multiplicities of $\rev \mathcal{L}(\la)$ at $0$.  Then the invariant orders at infinity $q_{1}\leq q_{2}\leq\cdots\leq q_{r}$ of $G(\la)$ are
\begin{equation*}
(q_{1}, q_{2},\ldots , q_{r})=(-e_{t},-e_{t-1},\ldots ,-e_1,\underbrace{0,\ldots,0}_{r-t-u},\widetilde{e}_{1}, \widetilde{e}_{2},\ldots, \widetilde{e}_{u}) - (g,g,\ldots ,g).
\end{equation*}
		
\end{prop}
\begin{proof} This proof is analogous to the one for Theorem \ref{th:Rosenlocalinf2}. It follows just from combining Theorem \ref{theo:spectralinf} and Lemma \ref{rem:invariantorders}.
\end{proof}

The following result is the counterpart of Proposition \ref{sametransfer} but for linearizations at infinity. It shows when a linear polynomial system matrix is a linearization at infinity of its transfer function matrix. The proof is immediate and, therefore, omitted.

\begin{prop}\label{sametransferinf}	Let
	\begin{equation}
	\mathcal{L}(\la)=\left[\begin{array}{cc}
	A_1 \la +A_0 &B_1 \la +B_0\\-(C_1 \la +C_0)&D_1 \la +D_0
	\end{array}\right]\in\F[\la]^{(n+q)\times (n+r)}
	\end{equation} be a linear polynomial system matrix and let $\wh{G}(\la)$ be its transfer function matrix. Then the following statements hold:
	\begin{itemize}
		\item[\rm (a)] If $\rank\begin{bmatrix} A_1 \\ C_1 \end{bmatrix}=\rank\begin{bmatrix} A_1 & B_1 \end{bmatrix}=n$ then $\mathcal{L}(\la)$ is a linearization of $\wh{G}(\la)$ at $\infty$ of grade $1.$
		\item[\rm (b)] If $\mathcal{L}(\la)$ is constant then $\mathcal{L}(\la)$ is a linearization of $\wh{G}(\la)$ at $\infty$ of grade $0.$
	\end{itemize}
\end{prop}

\begin{example}\rm Consider the matrices in Example \ref{ex_saad}. By Proposition \ref{sametransferinf}, the linear polynomial system matrix $P(\la)$ is a linearization of $G(\la)$ at $\infty$ of grade $1.$
\end{example}

\begin{example}\label{ex_subaiinfinity}\rm Consider the matrices in Example \ref{ex_subai}. Let us view $L(\la)$ as a polynomial system matrix with state matrix $A(\la)$ in \eqref{ex_statematrix}. With such partition, $G(\la)$ is the transfer function matrix of $L(\la).$ Then, by Proposition \ref{sametransferinf}, $L(\la)$ is a linearization of $G(\la)$ at $\infty$ of grade $1$ if $D_q$ has full column rank. However, the condition $\rank D_q =m$ is very restrictive, since it implies also $\rank D(\la)=m.$ Moreover, the structure of $G(\la)$ at $\infty$ is, in such a case, trivial because it is very easy to see that the $m$ invariant orders at $\infty$ of $G(\la)$ are all equal to $-q.$ This is consistent with Proposition \ref{prop:invariantorders}, because if $\rank D_q = m$ then $\rev L(0)$ has full column rank and, thus, $\rev L(\la)$ does not have partial multiplicities at zero. Moreover, as $A(\la)$ is the pencil in \eqref{ex_statematrix}, then it is easy to see that $\rev A(\la)$ has $m $ partial multiplicities at zero all equal to $q-1.$\\
Observe that, if we consider $A(\la)$ in \eqref{ex_statematrix} as state matrix of $L(\la),$ $\rev L(\la)$ is minimal at $0$ if and only if $\rank D_q =m.$ Thus, this hypothesis can not be avoided under such choice of state matrix. However, it is important to emphasize that if $L(\la)$ is viewed as a polynomial system matrix with empty state matrix then $L(\la)$ is a linearization of $G(\la)$ at $\infty $ of grade $q,$ without imposing any hypothesis. We postpone the proof of this result to Example \ref{ex:subairevisited}.
\end{example}

A linear polynomial system matrix that satisfies Definition \ref{def:linsubset} in $\F$ and Definition \ref{def:infinity}, for a certain grade $g,$ allows us to recover the complete information about the poles and zeros of the corresponding rational matrix, finite and at infinity. This is due to Theorem \ref{theo:spectral} and Proposition \ref{prop:invariantorders}. This important case leads us to introduce the following definition.

\begin{deff}[$g$-strong linearization]\label{def:strongrade} Let $G(\la) \in\F(\la)^{p\times m}$ and let $g$ be an integer. A linear polynomial system matrix $\mathcal{L}(\la)$ is said to be a strong linearization of grade $g$, or a $g$-strong linearization, of $G(\la)$ if $\mathcal{L}(\la)$ is a linearization of $G(\la)$ in $\F$ and also at $\infty$ of grade $g.$
\end{deff}

\begin{example} \rm Consider again the matrices in Example \ref{ex_saad}. Then the linear polynomial system matrix $P(\la)$ is a $1$-strong linearization of $G(\la)$ if and only if all the matrices $B_1, \ldots, B_s$ are nonsingular.
\end{example}

\subsection{Comparison with another definition of strong linearization}\label{comparison}

Recently, another definition of ``strong linearization'' of a rational matrix $G(\la)$ has been presented in \cite[Definition 3.4]{strong}. In contrast to Definition \ref{def:strongrade}, that definition does not make any reference to a ``grade $g$'', but the linear polynomial system matrices satisfying \cite[Definition 3.4]{strong} also allow us to recover the information about poles and zeros of $G(\la),$ including those at infinity. Therefore, it is convenient to establish a relation between \cite[Definition 3.4]{strong} and Definition \ref{def:strongrade}. This is the purpose of Proposition \ref{comparison}. Before stating and proving that proposition, we introduce some comments. Let us consider a linear polynomial system matrix \begin{equation*}
\mathcal{L}(\la)=\left[\begin{array}{cc}
A_1 \la +A_0 &B_1 \la +B_0\\-(C_1 \la +C_0)&D_1 \la +D_0
\end{array}\right]\in\F[\la]^{(n+q)\times (n+r)},
\end{equation*} with transfer function matrix $\widehat{G}(\la),$ and let $G(\la)\in \F(\la)^{p\times m}$ be a rational matrix written as in \eqref{eq.polspdec}. We recall that, according to \cite[Remark 3.5]{strong}, $\mathcal{L}(\la)$ is a strong linearization of $G(\la)$ in the sense of \cite[Definition 3.4]{strong} if the following statements hold:
\begin{itemize}
	\item[(a)] $\mathcal{L}(\la)$ is a linearization of $G(\la)$ in $\efe$ and,
	\item[(b)] $A_{1}$ is invertible if $n>0,$ and there exist integers $s_1,$ $s_2\geq 0$ and rational matrices $Q_1(\la)\in\F(\la)^{(p+s_1)\times (p+s_1)}$ and $Q_2(\la)\in\F(\la)^{(m+s_1)\times (m+s_1)}$ regular at $0$ such that \begin{equation}\label{reversalsdegree}
	Q_1(\la)\diag(\rev G(\la),I_{s_1})Q_2(\la)=\diag(\rev \wh{G}(\la),I_{s_2}).\end{equation}
\end{itemize}
As we stated in Remark \ref{equiv_finitepoins}, condition $(a)$ is equivalent to $\mathcal{L}(\la)$ being a linearization of $G(\la)$ in the sense of \cite[Definition 3.2]{strong}. For condition $(b)$ notice that, if $n>0$ and $\deg (\mathcal{L} (\la)) = 1$, in Definition \ref{def:strongrade} we do not require $A_{1}$ to be invertible but $\rank\begin{bmatrix} A_1 \\ C_1 \end{bmatrix}=\rank\begin{bmatrix} A_1 & B_1 \end{bmatrix}=n,$ according to Remark \ref{cond_a}. Observe, in addition, that these rank conditions are satisfied if $A_{1}$ is invertible.  Moreover, in contrast to \eqref{reversalsdegree}, in \eqref{eqinfty} we consider $\rev_{\ell} \wh{G}(\la)$ instead of $\rev \wh{G}(\la),$ where $\ell=\deg (\mathcal{L}(\la))$, and $\rev_g G(\la)$ instead of $\rev G(\la),$ for an integer $g.$ In this way, Definition \ref{def:infinity} looks for $\rev\mathcal{L}(\la)$ to be a linearization at $0$ of the $g$-reversal of $G(\la)$, because the transfer function matrix of $\rev\mathcal{L}(\la)$ is $\rev_{\ell} \wh{G}(\la).$ Note that, $\rev \wh{G}(\la)$ is the transfer function matrix of $\rev\mathcal{L}(\la)$ if and only if $\wh{G} (\la)$ is not strictly proper and the degree of the polynomial part of $\wh{G}(\la)$ is equal to the degree of $\mathcal{L}(\la)$. Thus, condition \eqref{reversalsdegree} is different from \eqref{eqinfty} in some cases. Nevertheless, as we will see in Proposition \ref{comparison}, in most cases strong linearizations of $G(\la)$ in the sense of \cite[Definition 3.4]{strong} are $g$-strong linearizations of $G(\la)$ of a certain grade $g.$

\begin{prop} \label{comparison}Let $G(\la) \in\F(\la)^{p\times m},$ and let \begin{equation*}
	\mathcal{L}(\la)=\left[\begin{array}{cc}
	A_1 \la +A_0 &B_1 \la +B_0\\-(C_1 \la +C_0)&D_1 \la +D_0
	\end{array}\right]\in\F[\la]^{(n+q)\times (n+r)}
	\end{equation*} be a strong linearization of $G(\la)$ according to \cite[Definition 3.4]{strong}. Let $G(\la)$ be expressed uniquely as in \eqref{eq.polspdec}, and let $g_G:=\deg(Q(\la))$ if $G(\la)$ is not strictly proper and $g_G:=0$ otherwise. Then the following statements hold:
	\begin{itemize}
		\item[\rm(a)] If $n=0$ or $D_1 + C_1 A_1^{-1}B_1\neq 0,$ then $\mathcal{L}(\la)$ is a $g_G$-strong linearization of $G(\la)$.
		\item[\rm(b)] If $D_1 + C_1 A_1^{-1}B_1= 0,$ $q=p,$ and $r=m,$ then $\mathcal{L}(\la)$ is a $(g_G+1)$-strong linearization of $G(\la)$.
		\item[\rm(c)] If $D_1 + C_1 A_1^{-1}B_1= 0,$ and $q\neq p$ or $r\neq m,$ then $\mathcal{L}(\la)$ is not a $g$-strong linearization of $G(\la)$ for any integer $g.$
	\end{itemize}
	
\end{prop}

\begin{proof} We remark that this proof is somewhat technical and that can be skipped without affecting the understanding of the rest of the paper. We will use throughout the proof that $\rev G(\la) = \rev_{g_G} G(\la)$ without mentioning it explicitly. Let $\mathcal{L}(\la)$ be a strong linearization of $G(\la)$ in the sense of \cite[Definition 3.4]{strong} and let $\wh{G} (\la)$ be the transfer function matrix of $\mathcal{L}(\la)$. Then $\mathcal{L}(\la)$ is a linearization of $G(\la)$ in $\F.$ Moreover, if $n>0,$ $A_{1}$ is invertible, which implies $\rank\begin{bmatrix} A_1 \\ C_1 \end{bmatrix}=\rank\begin{bmatrix} A_1 & B_1 \end{bmatrix}=n.$ Then, it only remains to study the different cases that may occur in condition \eqref{reversalsdegree} in order $\mathcal{L}(\la)$ to satisfy \eqref{eqinfty}, that is, in order to be a $g$-strong linearization of $G(\la)$ for some integer $g.$
	
We consider first the trivial case $n=0.$ In this case, $G(\la)$ is a polynomial matrix and $\widehat{G}(\la)=\mathcal{L}(\la)=D_1\la + D_0.$ Therefore, $\rev \widehat{G}(\la)=\rev_{\ell} \widehat{G}(\la),$ where $\ell=\deg (\mathcal{L}(\la)).$ Thus, $\mathcal{L}(\la)$ satisfies \eqref{eqinfty} with $g = g_G$, and it is a $g_G$-strong linearization of $G(\la).$
		
	In the rest of the proof, we assume $n>0$, which implies $\ell = \deg (\mathcal{L} (\la)) = 1$. In this case, $\widehat{G}(\la)$ can be written as $\widehat{G}(\la)=\la(D_1 + C_1 A_1^{-1}B_1)+\widehat{G}_{pr}(\la),$ where $\widehat{G}_{pr}(\la)$ is a proper rational matrix. Therefore, $\rev \widehat{G}(\la)=\rev_{\widehat{g}} \widehat{G}(\la),$ where $\widehat{g}=1$ if $D_1 + C_1 A_1^{-1}B_1\neq 0$ and $\widehat{g}=0$ otherwise. Then, we have two different cases. If $D_1 + C_1 A_1^{-1}B_1\neq 0$ then $\widehat{g}=\ell =1,$ and, therefore, $\mathcal{L}(\la)$ is a $g_G$-strong linearization of $G(\la).$ If $D_1 + C_1 A_1^{-1}B_1= 0$ then $\widehat{g}=0,$ and there are two different sub-cases:
			\begin{itemize}
				\item $q=p$ and $r=m,$ that is, $G(\la)$ and $\widehat{G}(\la)$ have the same size. So, in \eqref{reversalsdegree}, we have $s_1=s_2=:s.$ Then the invariant orders at $0$ of $\diag(\rev_{g_{G}} G(\la),I_{s})$ are equal to those of $\diag(\rev_0 \wh{G}(\la),I_{s}),$ which implies that the invariant orders at $0$ of $\rev_{g_{G}} G(\la)$ are also equal to those of $\rev_0 \wh{G}(\la).$ Multiplication by $\la$ of $\rev_{g_{G}} G(\la)$ and $\rev_0 \wh{G}(\la)$ yields that $\rev_{g_{G}+1} G(\la)$ and $\rev_1 \wh{G}(\la)$ have the same invariant orders at $0.$ The same happens with $\diag(\rev_{g_{G} + 1} G(\la),I_{s})$ and $\diag(\rev_1 \wh{G}(\la),I_{s}).$ Thus, there exist $\widetilde{Q}_1(\la)$ and $\widetilde{Q}_2(\la)$ rational matrices regular at $0$ such that $	\widetilde{Q}_1(\la)\diag(\rev_{g_G + 1} G(\la),I_{s})\widetilde{Q}_2(\la)=\diag(\rev_1 \wh{G}(\la),\allowbreak I_{s}),$ which proves according to \eqref{eqinfty} that $\mathcal{L}(\la)$ is a $(g_G + 1)$-strong linearization of $G(\la).$
				\item $q\neq p$ or $r\neq m,$ that is, $G(\la)$ and $\widehat{G}(\la)$ have different sizes and $s_1\neq s_2.$ In this case, there does not exist any integer $g$ such that the invariant orders at $0$ of $\diag(\rev_{g} G(\la),I_{s_1})$ are equal to the invariant orders at $0$ of $\diag(\rev_1 \wh{G}(\la),\allowbreak I_{s_2}).$ As a consequence, $\mathcal{L}(\la)$ is not a $g$-strong linearization of $G(\la)$ for any grade $g,$ since \eqref{eqinfty} is never satisfied. In order to prove this, note that by \eqref{reversalsdegree}, $\rank G(\la)\neq \rank\wh G(\la),$ and the invariant orders at zero of $\diag(\rev_{g_{G}} G(\la),I_{s_1})$ are equal to those of $\diag(\rev_0 \wh{G}(\la),I_{s_2}).$ Moreover, $\wh{G}(\la)$ is proper since $D_1 + C_1 A_1^{-1}B_1= 0.$ Therefore, all the invariant orders at $0$ of $\rev_0 \wh{G}(\la)=\wh{G}(1/\la)$ are nonnegative. So, the same happens to $\rev_{g_G}G(\la).$ Then, $\diag(\rev_1 \wh{G}(\la),\allowbreak I_{s_2})$ has $s_2$ invariant orders at $0$ equal to zero, and its remaining invariant orders at $0$ are $\rank \wh{G}(\la)$ positive numbers. In contrast, if $g>g_G,$ then $\diag(\rev_{g} G(\la),I_{s_1})$ has $s_1$ invariant orders at $0$  equal to zero, and its remaining invariant orders at $0$ are $\rank G(\la)$ positive numbers. If $g\leq g_G,$ notice that the largest invariant order at $0$ of $\diag(\rev_{g} G(\la),I_{s_1})$ is less than the largest of $\diag(\rev_{1} \wh{G}(\la),I_{s_2}).$
			\end{itemize}\vspace{-0.6cm}\end{proof}
We emphasize that, as far as we know, no explicit examples of strong linearizations in the sense of \cite{strong} with $n>0,$ $q\neq p$ or $r\neq m,$ and $D_1 + C_1 A_1^{-1}B_1= 0$ have been constructed so far in the literature. Thus, in plain words, Proposition \ref{comparison} states that strong linearizations according to \cite{strong} are particular cases of $g$-strong linearizations according to Definition \ref{def:strongrade}, except for a very particular instance.

In the rest of this section, we first revise important examples of strong linearizations in \cite{strong} from the perspective of Definition \ref{def:strongrade}. Then, in Example \ref{examplegrade}, we provide a $g$-strong linearization that is not a strong linearization in the sense of \cite{strong}. This example illustrates that the local approach followed in this paper yields, apart from the flexibility of constructing local linearizations, a concept of ``global'' strong linearization, more general than that of \cite{strong}.

\begin{example}\label{ex_1} \rm
Let $G(\la)$ be a rational matrix written as in \eqref{eq.polspdec}, i.e., $G(\la)=Q(\la)+G_{sp}(\la),$ with $\deg(Q(\la))>1.$ Then, the strong block minimal bases linearizations constructed in \cite[Theorem 5.11]{strong} are $\deg(Q(\la))$-strong linearizations of $G(\la),$ according to Definition \ref{def:strongrade}. Note that, with the notation in Proposition \ref{comparison}, these linearizations satisfy $D_1 + C_1 A_1^{-1}B_1\neq 0,$ since $D_1\neq 0,$ and $C_1=B_1=0.$
\end{example}

\begin{example}\label{ex_2} \rm Let us now consider a rational matrix $G(\la)$ written as in \eqref{eq.polspdec} with $\deg(Q(\la))\leq 1$ or $Q(\la)=0,$ and let $G_{sp}(\la)=C(\la I_n -A)^{-1}B$ be a minimal state-space realization of $G_{sp}(\la).$ Then, the following strong linearization
\begin{equation}
L(\la)=\left[\begin{array}{cc}
\la I_n - A & B \\
-C & Q(\la)
\end{array}\right]
\end{equation}
is considered in \cite{strong} (paragraph below equation $(28)$). In this case, with the notation in Proposition \ref{comparison}, we have $n>0,$ $q=p,$ $r=m,$ and $C_1A_1^{-1}B_1 =0.$ Then $D_1 +C_1A_1^{-1}B_1 =0$ if $g_G =0,$ or $D_1 +C_1A_1^{-1}B_1 \neq 0$ if $g_G =1.$ In any case, $L(\la)$ is a $1$-strong linearization by Proposition \ref{comparison}. Observe that in this example $G(\la)$ is the transfer function of $L(\la).$ Thus, the fact that $L(\la)$ is a $1$-strong linearization can also be obtained directly from Propositions \ref{sametransfer} and \ref{sametransferinf}.
	
\end{example}
Finally, we discuss the announced example of a linear polynomial system matrix that is a strong linearization in the sense of Definition \ref{def:strongrade} but not in the sense of \cite[Definition 3.4]{strong}.
\begin{example}\label{examplegrade}\rm Let us consider the rational matrix $$G(\la)=\left[\begin{array}{cc}
	\dfrac{\la^2+\la-1}{\la} & -\dfrac{1}{\la}\\
	-1 & -\la^2+\la-2
	\end{array}\right].$$ It can be easily proved that \begin{equation*}
	\mathcal{L}(\la) = \left[
	\begin{array}{cc|cc}
	\la & 0 & 1 & 1 \\
	0 & 1 & 0 & \la \\
	\hline \phantom{\Big|}
	1 & 0 &\la+1 & 0 \\
	\la & \la & 0 & \la-1
	\end{array}
	\right]:=\left[\begin{array}{c|c}
	A_1 \la +A_0 &B_1 \la +B_0\\\hline \phantom{\Big|} -(C_1 \la +C_0)&D_1 \la +D_0
	\end{array}\right]
	\end{equation*}
	is a linear polynomial system matrix of $G(\la).$ Moreover, note that $\mathcal{L}(\la)$ is minimal for all $\la_0\in\F.$ Therefore, by Proposition \ref{sametransfer}, $\mathcal{L}(\la)$ is a linearization of $G(\la)$ in $\F.$ By Proposition \ref{sametransferinf}, $\mathcal{L}(\la)$ is also a linearization of $G(\la)$ at $\infty$ of grade $1$ since $\rank\begin{bmatrix} A_1 \\ C_1 \end{bmatrix}=\rank\begin{bmatrix} A_1 & B_1 \end{bmatrix}=2.$ Thus, $\mathcal{L}(\la)$ is a $1$-strong linearization of $G(\la),$ according to Definition \ref{def:strongrade}. However, $\mathcal{L}(\la)$ is not a strong linearization according to \cite[Definition 3.4]{strong} since $A_1$ is singular. Nevertheless, we can recover easily the invariant orders at $\infty$ from $\mathcal{L}(\la)$ by applying Proposition \ref{prop:invariantorders} with $g=1.$ For this purpose, note that $\rev\mathcal{L}(\la)$ does not have elementary divisors at $0,$ since $\rev \mathcal{L}(\la)$ is regular at $0.$ Moreover, the only elementary divisor at $0$ of $A_1+A_0 \la$ is $\la.$ Therefore, the invariant orders at infinity of $G(\la)$ are $-2$ and $-1$ by Proposition \ref{prop:invariantorders}. The invariant orders of $G(\la)$ at any finite point can be recovered from $\mathcal{L}(\la)$ by using Theorem \ref{theo:spectral}. It is worthwhile to emphasize that the grade of $\mathcal{L}(\la)$ as a strong linearization of $G(\la)$ is different from the degree of the polynomial part of $G(\la).$ Observe that this also happens in Example \ref{ex_2} when $Q(\la)$ is a constant matrix.
	
\end{example}

\section{Block full rank pencils} \label{sec-full-rank-pencils}

In this section, we introduce a wide family of pencils that give us the information about the zeros of rational matrices locally. More precisely, they are linearizations with empty state matrix of rational matrices in some subsets of $\F$, as well as at $\infty$ under some conditions. These pencils will be called block full rank pencils, since they generalize the block minimal bases pencils introduced in \cite[Definition 3.1]{BKL}. The definition of block full rank pencils is motivated by the fact that most of the linearizations for rational approximations of NLEPs that have been constructed so far are pencils of this type. The key results in this section are Theorems \ref{th:1blockfullrank} and \ref{th:2blockfullrank}, which will be applied in the following section to establish rigorously and very easily the properties of the linearizations used in \cite{nlep}. Note that, according to Theorem \ref{theo:spectral}, the results in this section are not useful for studying, or computing, the finite poles of rational matrices because the considered linearizations have empty state matrix. This may be a drawback in certain situations, but we emphasize again that it is not in the development of algorithms for solving large-scale NLEPs via rational approximations \cite{guttel-tisseur-2017,nlep, automatic,Saad}. This is due to the fact that, in those cases, the poles of the rational matrix are known, since they are chosen for constructing the approximation, and/or are located outside the target set.

The theory we develop for block full rank pencils is based on the results for block minimal bases pencils from \cite{BKL}. It is also possible to develop directly such theory at the cost of proving
some preliminary lemmas. However, we think that our approach has the advantages of establishing a connection between both families of pencils and of emphasizing the relevance of these families in the study of rational and polynomial matrices.

Next, we present a few definitions and results from \cite{BKL} for making easier the reading of this section. The notion of (strong) block minimal bases pencil is recalled in Definition \ref{def:minlinearizations}. It relies on the concept of  minimal bases of rational subspaces, which are certain polynomial bases of such subspaces defined in \cite{forney,Kailath}. As in \cite{BKL}, we will say for brevity that a polynomial matrix $K(\la)\in\F[\la]^{p\times m}$ (with $p<m$) is a minimal basis if its rows form a minimal basis of the rational subspace they span. One of the most useful characterizations of minimal bases (see \cite[Main Theorem]{forney} or \cite[Theorem 2.2]{BKL}) is that $K(\la)\in\F[\la]^{p\times m}$ is a minimal basis if and only if $K(\la_{0})$ has full row rank for all $\la_{0}\in\F$ and $K(\la)$ is row reduced, i.e., its highest row degree coefficient matrix has full row rank (see \cite[Definition 2.1]{BKL}). Moreover, a minimal basis $N(\la)\in\F[\la]^{q\times m}$ is said to be dual to $K(\la)$ if $p+q=m$ and $K(\la)N(\la)^{T}=0$ \cite[Definition 2.5]{BKL}.

\begin{deff} {\rm \cite[Definition 3.1]{BKL}} \emph{((Strong) block minimal bases pencil)}. \label{def:minlinearizations} A block minimal bases pencil is a linear polynomial matrix over $\mathbb{F}$ with the following structure
	\begin{equation}
	\label{eq:minbaspencil}
	L(\la) =
	\left[
	\begin{array}{cc}
	M(\la) & K_2 (\la)^T \\
	K_1 (\lambda) &0
	\end{array}
	\right]
	\end{equation}
	where $K_1(\la)$ and $K_2(\la) $ are both minimal bases. In addition, if $K_1(\la) $ (respectively $K_2(\la)$) is a minimal basis with all its row degrees equal to $1$ and with the row degrees of a minimal basis $N_1(\la)$ (respectively $N_2(\la) $) dual to $K_1(\la)$ (respectively $K_2(\la)$) all equal, then $L(\la)$ is called a strong block minimal bases pencil. Moreover, given a polynomial matrix $P(\la),$ it is said that $L(\la)$ is associated with $P(\la)$ if $N_{2}(\la) M(\la) N_{1}(\la)^{T}=P(\la).$
\end{deff}

Theorem 3.3 in \cite{BKL} uses the standard definitions of linearizations and strong linearizations of polynomial matrices (see, for instance, \cite{spectral}) to prove the most important property of a (strong) block minimal bases pencil $L(\la)$ as in \eqref{eq:minbaspencil}, namely, $L(\la)$ is a (strong) linearization of the polynomial matrix $P(\la) = N_{2}(\la) M(\la) N_{1}(\la)^{T}$ for any $N_1(\la)$ and $N_2(\la)$ minimal bases dual to $K_1(\la)$ and $K_2(\la)$, respectively. In the strong case, this result considers $P(\la)$ as a polynomial matrix with grade $g_P := 1+\deg ( N_{1}(\la)) + \deg ( N_{2}(\la) )$. We can state \cite[Theorem 3.3]{BKL} in the language of this paper through Definitions \ref{def:linsubset} and \ref{def:strongrade} as ``a block minimal bases pencil $L(\la)$ is a linearization of $P(\la)$ in $\mathbb{F}$ with empty state matrix and a strong block minimal bases pencil $L(\la)$ is a $g_P$-strong linearization of $P(\la)$ with empty state matrix''. In order to see this, recall that the ``empty state matrix'' condition implies that the minimality condition is automatically satisfied (see Remarks \ref{rem:convention} and \ref{extremecases}) and that $\wh{G} (\la) = L(\la)$ in the definitions cited above.

Next, we relax to a minimum the conditions on $K_1(\la)$ and $K_2(\la)$ in \eqref{eq:minbaspencil} for defining a wider family of pencils that includes block  minimal bases pencils as a particular case.

\begin{deff} \label{def:blockfullrank} \emph{(Block full rank pencil)} A block full rank pencil is a linear polynomial matrix over $\mathbb{F}$ with the following structure
	\begin{equation}\label{almost}
	L(\la) =
	\left[
	\begin{array}{cc}
	M(\la) & K_2 (\la)^T \\
	K_1 (\lambda) &0
	\end{array}
	\right]
	\end{equation}
	where $K_1(\la)$ and $K_2(\la)$ are pencils with full row normal rank.
\end{deff}
Note that Definition \ref{def:blockfullrank} includes the cases when $K_1 (\la)$ or $K_2 (\la)$ are empty matrices, that is, when $L(\la)$ has only one block row or only one block column, respectively.

We introduce some auxiliary concepts and results before establishing the most important properties of block full rank pencils in Theorems \ref{th:1blockfullrank} and \ref{th:2blockfullrank}. We will say that a rational matrix $R(\la) \in \mathbb{F}(\la)^{p \times m}$ has full row rank in  $\Sigma \subseteq \mathbb{F}$ if, for all $\la_0\in \Sigma$, $R(\la_0)\in \mathbb{F}^{p \times m}$, i.e., $R(\la)$ is defined or bounded at $\la_0$, and $\rank R(\la_0) = p$. Observe that this implies that $R(\la)$ has no poles in $\Sigma$. The following lemma connects rational matrices with full row rank in $\Sigma$ with minimal bases, and establishes other properties that will be used later.

\begin{lem} \label{lem:fulltomin} Let $R(\la) \in \mathbb{F}(\la)^{p \times m}$ be a rational matrix with full row normal rank and let $T(\la) \in \mathbb{F}[\la]^{p \times m}$ be a minimal basis of the row space of $R(\la)$. Then the following statements hold:
\begin{enumerate}
\item[\rm (a)] There exists a unique regular rational matrix $S(\la) \in \mathbb{F}(\la)^{p \times p}$ such that $R(\la) = S(\la) T(\la)$.
\item[\rm (b)] $R(\la)$ has full row rank in  $\Sigma \subseteq \mathbb{F}$ if and only if $S(\la)$ in {\rm (a)} is regular in $\Sigma$.
\item[\rm (c)] $R(\la)$ is a polynomial matrix if and only if $S(\la)$ in {\rm (a)} is a polynomial matrix.
\item[\rm (d)] If $R(\la)$ is a matrix pencil, then $S(\la)$ in {\rm (a)} and $T(\la)$ are both matrix pencils.
\end{enumerate}
\end{lem}

\begin{proof}
Part (a). Each row of $S(\la)$ is uniquely defined because its entries are the unique rational coefficients that allow us to express the corresponding row of $R(\la)$ as a unique linear combination of the rows of $T(\la)$. Moreover, $S(\la)$ must be regular since, otherwise, there would exist a nonzero vector $y(\la) \in \F (\la)^{p\times 1}$ such that $y(\la)^T S(\la) =0$. So, $y(\la)^T R(\la) =0$, which contradicts that $\rank R(\la) = p$.

Part (b). It is obvious that if $S(\la)$ is regular in $\Sigma$, then $R(\la)$ has full row rank in $\Sigma$, because $T(\la)$ is defined in $\Sigma$, as $T(\la)$ is a polynomial matrix, and $T(\la)$ has full row rank in $\Sigma$, since $T(\la)$ is a minimal basis. The proof of the converse implication starts by proving that if $R(\la)$ has full row rank in  $\Sigma$, then $S(\la)$ is defined in $\Sigma$. To see this, note that the Smith form of $T(\la)$ is $[I_p \; \; 0]$, because $T(\la)$ is a minimal basis and, therefore, does not have finite zeros. Thus, there exist unimodular matrices $U(\la)$ and $V(\la)$ such that $T(\la) = U(\la) \, [I_p \; \; 0] \, V(\la)$, and
$R(\la) V(\la)^{-1} = [S(\la)U(\la) \; \; 0]$. This shows that $C(\la) := S(\la)U(\la)$ is defined in $\Sigma$, because $R(\la)$ and $V(\la)^{-1}$ are both defined in $\Sigma$ ($R(\la)$ by hypothesis and $V(\la)^{-1}$ because is unimodular and so a polynomial matrix). Therefore, $S(\la) = C(\la) U(\la)^{-1}$ is defined in $\Sigma$. This implies that we can write $R(\la_0) = S(\la_0) T(\la_0)$ for each $\lambda_0 \in \Sigma$, which in turns implies that $S(\la_0)$ is invertible because $R(\la_0)$ has full row rank.

Part (c). It follows directly from \cite[Main Theorem, part 4]{forney}.

Part (d). From \cite[Main Theorem, part 4]{forney}, we have that
\begin{equation} \label{eq:gradosR}
\deg (\mbox{row}_i \, (R(\la))) = \max_{1\leq j \leq p} \, (\deg (s_{ij} (\la)) + \deg (\mbox{row}_j \, (T(\la)))) \leq 1, \qquad \mbox{for $1 \leq i \leq p$,}
\end{equation}
where $\mbox{row}_i \, (R(\la))$ denotes the $i$th row of $R(\la)$ and the maximum is taken over the nonzero entries $s_{ij} (\la)$ of $S(\la)$. Since all the rows of $T(\la)$ are different from zero, \eqref{eq:gradosR} implies that $\deg (s_{ij} (\la)) \leq 1$ for each nonzero entry of $S(\la)$. Moreover, each column of $S(\la)$ has at least one nonzero entry, because $S(\la)$ is regular, which, combined with \eqref{eq:gradosR}, implies that $\deg (\mbox{row}_j \, (T(\la))) \leq 1$, for each $j = 1,\ldots , p$.
\end{proof}

The last concepts we need before stating and proving the main Theorem \ref{th:1blockfullrank} are those of rational basis and dual rational bases. A rational matrix $G(\la)\in \F (\la)^{p\times m}$ (with $p<m$) is said to be a rational basis if it is a basis of the rational subspace spanned by its rows, i.e., if it has full row normal rank. Two rational bases $G(\la) \in \F (\la)^{p\times m}$ and $H(\la) \in \F (\la)^{q\times m}$ are said to be dual if $p+q = m$, and $G(\la) \, H(\la)^T =0$. We are finally ready for presenting the main result of this section.

\begin{theo} \label{th:1blockfullrank} Let $L(\la)$ be a block full rank pencil as in \eqref{almost} and let $N_1 (\la)$ and $N_2 (\la)$ be any rational bases dual to $K_1 (\la)$ and $K_2 (\la)$, respectively. Let $\Omega\subseteq\efe$ be nonempty. If $K_i (\la)$ and $N_i (\la)$ have full row rank in $\Omega$, for $i =1,2,$ then
$L(\la)$ is a linearization with empty state matrix of the rational matrix $G(\la) = N_{2}(\la) M(\la) N_{1}(\la)^{T}$ in $ \Omega$.
\end{theo}

\begin{proof} In order to simplify the notation, throughout this proof we do not specify the sizes of different identity matrices and all of them are denoted by $I$. Let $\widetilde{K}_1 (\la), \widetilde{K}_2 (\la), \widetilde{N}_1 (\la)$ and $\widetilde{N}_2 (\la)$ be minimal bases of the row spaces of $K_1 (\la)$, $K_2 (\la)$, $N_1 (\la)$ and $N_2 (\la)$, respectively. Then, Lemma \ref{lem:fulltomin} implies that there exist regular rational matrices $S_1 (\la)$, $S_2 (\la)$, $W_1 (\la)$ and $W_2 (\la)$ such that
\begin{eqnarray*}
	& K_i (\la) = S_i (\la) \widetilde{K}_i (\la), &  \mbox{and $S_i (\la)$ is regular in $\Omega$, for $i=1,2.$} \\
	& N_i (\la) = W_i (\la) \widetilde{N}_i (\la), &  \mbox{and $W_i (\la)$ is regular in $\Omega$, for $i=1,2.$}
\end{eqnarray*}
Moreover, $\widetilde{K}_1 (\la), \widetilde{K}_2 (\la), S_1(\la)$ and $S_2(\la)$ are all matrix pencils. Then, $L(\la)$ can be factorized as follows,
\begin{equation} \label{eq:fact1fullrank}
	L(\la) =
    \left[
	\begin{array}{cc}
	I & 0\\
	0 & S_1 (\lambda)
	\end{array}
	\right]
	\left[
	\begin{array}{cc}
	M(\la) & \widetilde{K}_2 (\la)^T \\
	\widetilde{K}_1 (\lambda) &0
	\end{array}
	\right]
    \left[
	\begin{array}{cc}
	I & 0\\
	0 & S_2 (\lambda)^T
	\end{array}
	\right],
\end{equation}
where the first and third factors are regular in $\Omega$. Note that the factor in the middle is a block minimal bases pencil associated with the polynomial matrix $\widetilde{N}_{2}(\la) M(\la) \widetilde{N}_{1}(\la)^{T}$, since the regularity of $S_i (\la)$ and $W_i (\la)$ implies that $\widetilde{K}_{i}(\la)$ and $\widetilde{N}_{i}(\la)$ are dual minimal bases for $i=1,2$. Then, there exist unimodular matrices $U(\la)$ and $V(\la)$ such that
\begin{align} \nonumber
\left[
	\begin{array}{cc}
	M(\la) & \widetilde{K}_2 (\la)^T \\
	\widetilde{K}_1 (\lambda) &0
	\end{array}
	\right]
& = U(\la)
\left[
	\begin{array}{cc}
	\widetilde{N}_{2}(\la) M(\la) \widetilde{N}_{1}(\la)^{T} & 0\\
	0 &  I
	\end{array}
	\right]
V(\la) \\ \label{eq:fact2fullrank}
& = U(\la)
\left[
	\begin{array}{cc}
	W_2 (\la)^{-1} & 0\\
	0 & I
	\end{array}
	\right]
\left[
	\begin{array}{cc}
	G(\la) & 0\\
	0 &  I
	\end{array}
	\right]
\left[
	\begin{array}{cc}
	W_1 (\la)^{-T}  & 0\\
	0 & I
	\end{array}
	\right]
V(\la),
\end{align}
where $U(\la) \diag (W_2 (\la)^{-1} , I)$ and $\diag (W_1 (\la)^{-T} , I) V(\la)$ are regular in $\Omega$. From combining \eqref{eq:fact1fullrank} and \eqref{eq:fact2fullrank}, we obtain that $L(\la)$ and $\diag (G (\la) , I)$ are equivalent in $\Omega$. This proves that $L(\la)$ is a linearization with empty state matrix of $G(\la)$ in $\Omega$ according to Definitions \ref{def_pointstronglin} and \ref{def:linsubset}, since the minimality condition is automatically satisfied if the state matrix is empty.
\end{proof}

\begin{rem} {\rm
In the scenario of Theorem \ref{th:1blockfullrank}, Theorem \ref{theo:spectral} guarantees that the elementary divisors of $L(\la)$ in $\Omega$  coincide exactly with the zero elementary divisors of $G(\la)$ in $\Omega$. Moreover, it is clear from the expression $G(\la) = N_{2}(\la) M(\la) N_{1}(\la)^{T}$ that $G(\la)$ does not have poles in $\Omega,$ since the matrices $N_{i}(\la)$ must be defined in $\Omega $ but they are not defined at the poles of $G(\la)$. Thus, $G(\la)$ has only eigenvalues in $\Omega$, and all the information about them, i.e., geometric, algebraic and partial multiplicities, is contained in $L(\la)$.}
\end{rem}

\begin{rem}\rm
	If in Theorem \ref{th:1blockfullrank}, $K_1(\la)$ (resp. $K_2(\la) $)
	is an empty matrix, we can take any rational matrix $N_1(\la)\in \F (\la)^{s_1 \times s_1}$ (resp. $N_2(\la) \in \F (\la)^{s_2 \times s_2}$) regular in $\Omega$, where $s_1$ (resp. $s_2$) is the number of colums (resp. rows) of $M(\la)$. The standard choices are $N_1(\la)= I_{s_1}$ and $N_2(\la)= I_{s_2}$.
\end{rem}

\begin{rem}\label{rem.blockfullassoc} \rm Under the conditions of Theorem \ref{th:1blockfullrank}, we will say for brevity that ``$L(\la)$ is a block full rank pencil associated with $G(\la)$ in $\Omega$''. We emphasize that this ``association'' is not one-to-one because there are infinitely many rational bases $N_1 (\la)$ and $N_2 (\la)$ dual to $K_1 (\la)$ and $K_2 (\la)$.
\end{rem}

Next, we present sufficient conditions for a block full rank pencil to be a linearization of $G(\la) = N_{2}(\la) M(\la) N_{1}(\la)^{T}$ at $\infty$ of a certain grade $g$. In order to avoid cases with limited interest in applications that complicate the statement, in Theorem \ref{th:2blockfullrank} we assume $\deg (L(\la)) =1$.

\begin{theo} \label{th:2blockfullrank} Let $L(\la)$ be a block full rank pencil as in \eqref{almost} with $\deg (L(\la)) =1$ and let $N_1 (\la)$ and $N_2 (\la)$ be rational bases dual to $K_1 (\la)$ and $K_2 (\la)$, respectively. If, for $i=1,2$, $\rev_1 K_i (\la)$ has full row rank at zero, and there exists an integer number $t_i$ such that $\rev_{t_i} N_i (\la)$ has full row rank at zero, then $L(\la)$ is a linearization with empty state matrix of the rational matrix $G(\la) = N_{2}(\la) M(\la) N_{1}(\la)^{T}$ at $\infty$ of grade $1 + t_1 + t_2$.
\end{theo}

\begin{proof} Note that
\[
\rev  L(\la) =
	\left[
	\begin{array}{cc}
	\rev_1 M(\la) & \rev_1 K_2 (\la)^T \\
	\rev_1 K_1 (\lambda) &0
	\end{array}
	\right]
\]
is a block full rank pencil. Moreover, for $i=1,2$, $\rev_{t_i} N_i(\la)$ has full row normal rank, and $K_i (\la)\, N_i(\la)^T =0$ implies $(\rev_1 K_i (\la))\,  (\rev_{t_i} N_i(\la))^T =0$. Therefore, $\rev_{t_i} N_i(\la)$ is a rational basis dual to $\rev_1 K_i (\la)$. Then, Theorem \ref{th:1blockfullrank} applied to $\rev L(\la)$ proves that $\rev L(\la)$ is a linearization with empty state matrix at zero of
\[
(\rev_{t_2} N_2(\la)) \, (\rev_1 M(\la)) \, (\rev_{t_1} N_1(\la)^T) = \rev_{1 + t_1 + t_2} G(\la),
\]
which combined with Proposition \ref{prop_lininf} proves the result.
\end{proof}
As a consequence of Theorems \ref{th:1blockfullrank} and \ref{th:2blockfullrank}, we obtain Corollary \ref{cor_locallin}. Although it follows immediately from them, we state it since it generalizes the structure of most of the linearizations of rational approximations of NLEPs that appear in the literature. Moreover, it is very useful in order to characterize easily some pencils as linearizations of rational matrices locally and to obtain the information about the zeros of such rational matrices in subsets not containing poles.
\begin{cor}\label{cor_locallin}
	Let $$R(\la)=(A_0-\la B_0)R_0(\la)+ (A_1-\la B_1)R_1(\la) + \cdots + (A_N-\la B_N)R_N(\la)$$ be a $p\times m$ rational matrix written in terms of some matrix pencils $A_i-\la B_i \in \F[\la]^{p \times n_i}$ and rational matrices $R_{i}(\la) \in \F(\la)^{n_i \times m}$. Define
	\begin{align*}
	M(\la)&:=\left[(A_0-\la B_0)\quad (A_1-\la B_1)\quad \cdots \quad (A_N-\la B_N) \right] \; \mbox{and} \\
	N_1 (\la) &:=\left[R_0(\la)^{T}\quad R_1(\la)^{T}\quad \cdots \quad R_N(\la)^{T} \right],
	\end{align*}
	and assume that $N_1(\la)$ has full row normal rank. Let $L(\la)=\left[\begin{array}{c}
	M(\la)\\
	\hdashline[2pt/2pt]
	K_1(\la)
	\end{array}\right]$ be a block full rank pencil of degree $1$ with only one block column and such that $K_1(\la)$ and $N_1 (\la)$ are dual rational bases. Let $\Omega\subseteq\efe$ be nonempty. Then the following statements hold:
	\begin{itemize}
		\item[\rm(a)] If $K_1(\la)$ and $N_1 (\la)$ have full row rank in $\Omega$ then $L(\la)$ is a linearization with empty state matrix of $R(\la)$ in $ \Omega . $
		\item[\rm(b)] If $\rev_1 K_1(\la)$ has full row rank at $0,$ and there exists an integer $t$ such that $\rev_t N_1(\la)$ has full row rank at $0$, then $L(\la)$ is a linearization with empty state matrix of $R(\la)$ at $\infty$ of grade $1+t.$
		
	\end{itemize}
\end{cor}
	
	\begin{rem} \rm
		We emphasize that in some relevant applications the rational matrices $R_i(\la)$ of Corollary \ref{cor_locallin} are just of the form $R_i(\la)=r_i(\la)I_m,$ where $r_i(\la)$ are scalar rational functions, and/or most of the pencils $A_i -\la B_i$ are constant matrices or a linear scalar function times a constant matrix. Moreover, in some other applications a low rank structure is present in $R(\la)$, that is, some of the terms in $R(\la)$ have a rank much smaller than $\min \{ p,m\}$, and the corresponding rational matrices are written in the form $R_i(\la)=r_{i}(\la)R_i$, where $R_i \in \F^{n_i \times m}$ is a constant matrix with $n_i \ll m.$
	\end{rem}

In the next two examples, we revisit the pencils introduced in Examples \ref{ex_saad} and \ref{ex_subai} from the perspective of the block full rank pencils. These examples illustrate how the theory of block full rank pencils may simplify the analysis of the properties of important linearizations of rational matrices when one is not interested on the information about the poles.

\begin{example}\rm\label{ex:saadrevisited} Let us consider the rational matrix $G(\la)$ and the pencil $P(\la)$ in Example \ref{ex_saad}. We partition $P(\la)$ as follows:
	$$ P(\la) =
		\left[
		\begin{array}{ccccc}
		(\la - \sigma_1) I & & & & I \\
		&(\la - \sigma_2) I & & & I \\
		& & \ddots & & \vdots \\
		&& &(\la - \sigma_s) I & I \\ \hdashline[2pt/2pt]
		-B_1 & -B_2 &\cdots & -B_s & \la A_0 - B_0\\
		\end{array}
		\right] =: \left[
		\begin{array}{c}
		K_1 (\lambda) \\  \hdashline[2pt/2pt]
		M(\la)
		\end{array}
		\right].$$
	Observe that, in the above partition, we are considering a permuted version of the structure of the pencil $L(\la)$ in Corollary \ref{cor_locallin}. Note now that $K_1 (\la)$ has full row rank in $\mathbb{C}$, and
		$$N_1 (\la) := \left[
		\begin{array}{ccccc}
		\frac{1}{\sigma_1 - \la} I & \frac{1}{\sigma_2 - \la} I & \ldots & \frac{1}{\sigma_s - \la} I & I
		\end{array}
		\right]$$ is a rational basis dual to $K_1 (\la)$ with full row rank in $\Sigma:=\mathbb{C} \setminus \{\sigma_1 , \ldots , \sigma_s \}$.
Then, by Corollary \ref{cor_locallin}$\rm (a)$, $P(\la)$ is a linearization with empty state matrix of $G(\la)$ in $\Sigma.$ Moreover, note that $\rev_1 K_1(\la)$ and $\rev_0 N_1(\la)= \left[
\begin{array}{ccccc}
\frac{\la}{\la\sigma_1 - 1} I & \frac{\la}{\la\sigma_2 - 1} I & \ldots & \frac{\la}{\la\sigma_s - 1} I & I
\end{array}
\right]$ both have full row rank at $0$. Thus, by Corollary \ref{cor_locallin}$\rm (b)$, $P(\la)$ is a linearization with empty state matrix of $G(\la)$ at $\infty $ of grade $1.$
\end{example}

\begin{example}\rm\label{ex:subairevisited} Let us consider the rational matrix $G(\la)$ and the pencil $L(\la)$ in Example \ref{ex_subai}. We now consider the following partition of $L(\la)$:
	\begin{equation} \label{lin_subai2}
	L(\la) =
	\left[
	\begin{array}{ccccc}
	\la D_q + D_{q-1} & D_{q-2} & \cdots & D_0 & -C \\ 	\hdashline[2pt/2pt]
	-I_m & \la I_m & 0 & \cdots & 0 \\
	& \ddots & \ddots & \vdots & \vdots \\
	&  & -I_m & \la I_m & 0 \\
	&  &  & B & \la I_n - A
	\end{array}
	\right]:=\left[\begin{array}{c}
	M(\la)\\
	\hdashline[2pt/2pt]
	K_1(\la)
	\end{array}\right] .
	\end{equation}
Since $K_1(\la)$ has full row normal rank, $L(\la)$ has the structure of the block full rank pencil in Corollary \ref{cor_locallin}. Observe that
\begin{equation}\label{dual}N_1(\la):= [\la^{q-1}I_m\quad \la^{q-2} I_m \quad \cdots \quad I_m \quad -B^T (\la I_n - A)^{-T} ]
\end{equation}
is a rational basis dual to $K_1(\la)$ and that $K_1(\la)$ and $N_1(\la)$ have both full row rank in $\Omega:=\{\la\in\F : \la \text{ is not eigenvalue of }A \}.$ Thus, Corollary \ref{cor_locallin}$\rm(a)$ implies that $L(\la)$ is a linearization with empty state matrix of $G(\la)$ in $\Omega.$ This example, together with Example \ref{ex_subai}, illustrates a very important fact that we have already mentioned: the same pencil can be viewed as a linearization with different state matrices. Moreover, different views may require different conditions, may lead to different sets where the pencil is a linearization, and may differ in the difficulty to get the results. For instant, when the developments in this example are compared with the direct application of the definition of linearization presented in the second approach in Example \ref{ex_subai} through the matrices $V(\la)$ and $U(\la)$ in \eqref{eq_V} and \eqref{eq_U}, respectively, we can conclude that the ``block full rank pencil'' view leads to the same results in a much simpler way. We have experimented the simplicity of the ``block full rank pencil'' approach in many other examples.

Finally, note that the pencil in \eqref{lin_subai2} satisfies that $\rev_1 K_1(\la)$ has full row rank at $0$ and that $N_1(\la)$ in \eqref{dual} satisfies that $\rev_{q-1} N_1(\la)$ has also full row rank at $0$. Thus, Corollary \ref{cor_locallin}$\rm (b)$ implies that $L(\la)$ is a linearization with empty state matrix of $G(\la)$ at $\infty$ of grade $q$. By comparing this result with the result in Example \ref{ex_subaiinfinity}, we see that considering $L(\la)$ as a block full rank pencil leads to much stronger results on the structure at infinity than considering $L(\la)$ as a polynomial system matrix with state matrix $A(\la)$ in \eqref{ex_statematrix}. In the former case, we do not need any extra hypothesis in order $L(\la)$ to be a linearization at infinity, while in the latter the condition $\rank D_q =m$ is needed.
\end{example}

As previously announced, the results in this section will be used in Section \ref{guttel}. In addition, in a future work \cite{local2}, we will extend them. More precisely, we will define block full rank linearizations of rational matrices with non empty state matrix that, therefore, will contain information about the poles. Moreover, we will apply these results to establish rigorously and very easily the properties of the linearizations introduced in \cite{automatic}.
	
\section{Application of the local linearization theory to NLEIGS pencils}\label{guttel}
In this section we study in depth the pencils introduced in the influential reference \cite{nlep}. This reference presents one of the first systematic approaches for solving large scale NLEPs. The approach in \cite{nlep} consists essentially of three steps: (1) the matrix defining the NLEP is approximated by a rational matrix $Q_N(\la)$ via Hermite's interpolation in a certain compact target set $\Sigma \subset \mathbb{C}$ where the eigenvalues of interest are located; (2) the obtained rational matrix is linearized by using a certain pencil $L_N(\la)$; and (3) a highly structured rational Krylov method is applied to the pencil to compute the eigenvalues of $Q_N (\la)$ in $\Sigma$. For brevity of exposition, and also for recognizing the key contribution of \cite{nlep}, we will call {\em NLEIGS pencils} to the pencils introduced in this reference. The main goal of this section is to replace the vague usage of the word ``linearization'' in \cite{nlep} by a number of rigorous results on NLEIGS pencils which, combined with the results in Sections \ref{sec_local} and \ref{sec-full-rank-pencils}, establish the precise properties enjoyed with respect to eigenvalues (and poles) of the NLEIGS pencils. We remark that NLEIGS pencils $L_N (\la)$ were the initial motivation for developing the results of this paper, since $L_N (\la)$ is not a linearization of the rational matrix $Q_N(\la)$, according to the definitions of linearization and strong linearization presented in \cite{strong} or \cite{AlBe16}.

Since we are interested in rational matrices and their linearizations, all the delicate details about how the rational matrices $Q_N(\la)$ are constructed as approximations of the original NLEPs are omitted. Such details can be found in \cite{nlep}. As in the rest of the paper, the results in this section are valid and are stated in any algebraically closed field $\mathbb{F}$ that does not include infinity. Note, however, that reference \cite{nlep} considers only the complex field and that this restriction is important in the approximation phase of the NLEP. Moreover, although \cite{nlep} deals with regular rational matrices $Q_N (\la)$, we will not impose such condition initially in our developments.

Reference \cite{nlep} uses two families of rational matrices, and corresponding pencils, depending on whether or not a certain low rank structure is present in the original NLEP. We will refer to them as the {\em NLEIGS basic problem} and the {\em NLEIGS low rank structured problem}, respectively. The NLEIGS pencils corresponding to each of these two cases will be studied from two perspectives giving rise to the four subsections included in this section. These two perspectives are considering NLEIGS pencils as block full rank pencils and, thus, as linearizations with empty state matrices, and considering them as polynomial system matrices with transfer function matrices equivalent to $Q_N (\la)$ everywhere except at a point $\xi_N$. Both perspectives allow us to state in a rigorous sense that NLEIGS pencils are linearizations of $Q_N(\la)$, but the one based on block full rank pencils is much simpler, does not require any hypothesis and covers fully the applications of interest in \cite{nlep}. In contrast, the polynomial system matrix perspective provides more information on $Q_N(\la)$ but at the cost of extra hypotheses which are not imposed in \cite{nlep} and that require considerable effort to check.

\subsection{The NLEIGS basic problem from the point of view of block full rank pencils} \label{sec.NLEIGSbasic1} The families of rational matrices considered in \cite{nlep} are defined in terms of the following parameters: a list of nodes $(\sigma_{0},\sigma_{1},\ldots,\sigma_{N-1})$ in $ \FF$, a list of nonzero poles $(\xi_{1},\xi_{2},\ldots,\xi_{N})$ in $ \FF  \, \cup \, \{\infty\}$, and a list of nonzero scaling parameters $(\beta_{0},\beta_{1},\ldots,\beta_N) $ in $\FF $. It is important to bear in mind that \cite{nlep} assumes that the poles are all distinct from the nodes. However, we do not assume such property, except in a few results where it will be explicitly stated. With these parameters, the following sequence of rational scalar functions is defined:
\begin{equation}\label{eq_interpolant}
b_0(\la)=\frac{1}{\beta_0}, \quad
b_{i}(\lambda)=\dfrac{1}{\beta_{0}}\displaystyle\prod_{k=1}^{i}\dfrac{\la - \sigma_{k-1}}{\beta_{k}(1- \la /\xi_{k})},\quad i= 1,\ldots , N .
\end{equation}
Let us now define the linear scalar functions
\begin{equation}\label{scalar_functions}
g_{i}(\la):=\beta_{i}\left(1-\la /\xi_{i}\right),\quad \text{  and  }\quad h_{j}(\la):= \la-\sigma_j,
\end{equation}
for $i=1,\ldots , N$, and $j=0,\ldots , N-1.$ Then, the rational functions $b_{i}(\la)$ satisfy the simple recursion
\begin{equation*}\label{eq_recur}
g_{j+1}( \la)\,b_{j+1}(\lambda) =h_j(\la)\,b_{j}(\lambda),\quad j=0,1,\ldots, N-1 \, ,
\end{equation*}
which will be useful in the sequel. Note that the rational functions $b_{i}(\la)$ could not be proper, since for any infinite pole $\xi_i = \infty$ the corresponding factor $1- \la /\xi_{i}$ is just equal to $1$, and, therefore, $b_{i}(\la)$ has a nonconstant polynomial part.

With all this information, we are in the position of introducing the first family of rational matrices considered in \cite{nlep}, whose elements are defined as
\begin{equation} \label{eq.QN}
Q_{N}(\lambda)=b_{0}(\lambda)D_{0}+b_{1}(\lambda)D_{1}+\cdots + b_{N}(\lambda)D_{N} \in \FF (\la)^{m\times m},
\end{equation}
where $D_0, \ldots , D_N \in \FF^{m \times m}$ are constant matrices.

In this section, the nodes $(\sigma_{0},\ldots, \allowbreak \sigma_{N-1})$, the poles $(\xi_{1},\ldots,\xi_{N})$, the scaling parameters $(\beta_{0},\ldots,\beta_N)$ and the matrices $D_0, \allowbreak \ldots , D_N$ are arbitrary parameters that allow us to define the considered family of rational matrices. However, in \cite{nlep} these parameters are carefully chosen in such a way that $Q_N (\la)$ approximates satisfactorily the matrix defining the NLEP to be solved in the target set $\Sigma \subset \FF$ containing the desired eigenvalues of the NLEP. In this scenario, it is important to stress that the poles $(\xi_{1},\ldots,\xi_{N})$ are always chosen outside the region of interest $\Sigma$ \cite[p. A2852]{nlep}, which implies that all the zeros of $Q_N (\la)$ located in $\Sigma$ are eigenvalues of $Q_N (\la)$. Thus, the REP associated with $Q_N (\la)$ is an explicit example of a problem with a property that has been mentioned before in this paper, i.e., the poles are known and located outside the region of interest and, then, it is not needed to compute them. Note, however, the following subtlety: though it is clear that the finite poles of $Q_N (\la)$ are included in the list $(\xi_{1},\ldots,\xi_{N})$, it is easy to construct examples of matrices as in \eqref{eq.QN} for which some of the finite numbers in $(\xi_{1},\ldots,\xi_{N})$ are not poles due to some cancellations. Thus, all the finite numbers in $(\xi_{1},\ldots,\xi_{N})$ are not necessarily finite poles of $Q_N (\la)$ and, even more, the partial multiplicities of such poles are not immediately visible from \eqref{eq.QN}. Despite these comments, we will call the numbers $(\xi_{1},\ldots,\xi_{N})$ poles, following the usage in \cite{nlep}.

In order to solve the REP $Q_{N}(\lambda) \, y=0$, the authors of \cite{nlep} solve the generalized eigenvalue problem corresponding to the pencil
\begin{equation} \label{eq.LN}
L_{N}(\lambda)=\left[\begin{array}{c}
M_{N}(\lambda)\\
K_{N}(\lambda)
\end{array}\right],
\end{equation}
where
$$
M_{N}(\lambda) :=\left[\begin{array}{ccccc}
\dfrac{g_{N}(\la)}{\beta_{N}}	D_{0}  & \dfrac{g_{N}(\la)}{\beta_{N}} D_{1} & \cdots &\dfrac{g_{N}(\la)}{\beta_{N}}D_{N-2} & \dfrac{g_{N}(\la)}{\beta_{N}}D_{N-1}+\dfrac{h_{N-1}(\la)}{\beta_{N}}D_{N}
\end{array}\right],$$ and
$$ K_{N}(\la) :=\left[\begin{array}{cccccc}
-h_{0}(\lambda) & g_{1}(\la)  \\
& -h_{1}(\lambda) & g_{2}(\la)    \\
& & \ddots  & \ddots  \\
&  & &-h_{N-2}(\lambda) & g_{N-1}(\la)
\end{array}\right] \otimes I_{m}.$$
In \cite{nlep} the use of $L_N (\la)$ for solving the REP associated to $Q_N (\la)$ is supported by \cite[Theorem 3.2]{nlep}, which states that $L_{N}(\lambda)$ is a strong linearization of the rational matrix $Q_{N}(\lambda)$ without specifying the exact meaning of ``strong linearization'' in this rational context. Moreover, the proof of \cite[Theorem 3.2]{nlep} consists of a reference to \cite[Theorem 3.1]{amiretal-09}, which is a paper dealing with strong linearizations of {\em polynomial} matrices in the classical sense of \cite{GoLaRo82}. However, as a consequence of the results in Section \ref{sec-full-rank-pencils}, it is very easy to prove that $L_N(\la)$ is always a linearization of $Q_N (\la)$ in a set including the region of interest in \cite{nlep}, as well as at infinity. This is proved in Theorem \ref{th.NLEIGSbasic1}, where the nomenclature introduced in Remark \ref{rem.blockfullassoc} is used.

\begin{theo} \label{th.NLEIGSbasic1}
	Let $Q_N (\la)$ be the rational matrix in \eqref{eq.QN} and $L_N(\la)$ be the pencil in \eqref{eq.LN}. Let $\mathcal{P}_N$ and $i_N$ be, respectively, the set of finite poles and the number of infinite poles in the list $(\xi_1, \xi_2, \ldots , \xi_N)$. Then, the following statements hold:
	\begin{enumerate}
		\item[\rm (a)] $L_N (\la)$ partitioned as in \eqref{eq.LN} is a block full rank pencil with only one block column associated with $Q_N(\la)$ in $\FF \setminus \mathcal{P}_N$.
		\item[\rm (b)] $L_N (\la)$ is a linearization with empty state matrix of $Q_N(\la)$ in $\FF \setminus \mathcal{P}_N$.
		\item[\rm (c)] $L_N (\la)$ is a linearization with empty state matrix of $Q_N(\la)$ at $\infty$ of grade $i_N$.
	\end{enumerate}
\end{theo}

\begin{proof}
	It is immediate to check that
	\begin{equation}\label{eq.dualrational1}
	N_N (\la) := \frac{1}{1-\frac{\la}{\xi_N}} \, \left[
	\begin{array}{cccc}
	b_0(\la) & b_1(\la) & \cdots & b_{N-1} (\la)
	\end{array}
	\right] \otimes I_m
	\end{equation}
	is a rational basis dual to $K_N (\la)$. Note also that $K_N (\la)$ and $N_N (\la)$ have both full row rank in $\FF\setminus \mathcal{P}_N$. In addition, an easy direct computation proves $M_N (\la) N_N(\la)^T = Q_N (\la)$. Thus, parts (a) and (b) follow from Theorem \ref{th:1blockfullrank}. Observe that parts (a) and (b) can also be proved from Corollary \ref{cor_locallin}, since the structures of $Q_N (\la)$, $L_N (\la)$ and $N_N (\la)$ are particular cases of those described in that corollary.
	
	In order to prove part (c), note first that $\mbox{rev}_1 \, K_N (\la)$ has full row rank at zero. We now consider the rational matrix $\mbox{rev}_{i_N - 1} \, N_N (\la)=\la^{i_N - 1}  N_N\left( \frac{1}{\la} \right) $, which is of the form
	$$\la^{i_{N }-1}  N_N \left( \frac{1}{\la} \right) =  \left[
	\begin{array}{ccccccc} \displaystyle
	* &  \cdots &
	* &   \frac{\la}{\la-1/\xi_N}\la^{i_{N}-1}b_{N-1}\left( \frac{1}{\la} \right) I_m
	\end{array}
	\right],$$ where the entries $*$ are defined at $0.$ Denote by $i_{N-1}$ the number of infinite poles in the list $(\xi_1, \xi_2, \ldots , \xi_{N-1})$. Then, $b_{N-1}\left( \frac{1}{\la} \right)=\frac{1}{\la^{i_{N-1}}}c(\la),$ for a certain rational function $c(\la)$ with $c(0)\neq 0.$ Thus, we obtain that $\mbox{rev}_{i_N - 1} \, N_N (\la)$ has full row rank at $0,$ taking into account that $i_{N-1}=i_{N}$ if $\xi_{N}\neq \infty,$ and $i_{N-1}=i_{N}-1$ if $\xi_{N}=\infty.$ Then, part (c) follows from Theorem \ref{th:2blockfullrank}.
\end{proof}

Combining Theorems \ref{th.NLEIGSbasic1} and \ref{theo:spectral}, we get that $L_N(\la)$ contains all the information about the finite eigenvalues of $Q_N (\la)$ in $\FF \setminus \mathcal{P}_N$, including all type of multiplicities (algebraic, geometric and partial). Moreover, Proposition \ref{prop:invariantorders} allows us to recover the complete pole-zero structure of $Q_N (\la)$ at $\infty$ from the eigenvalue structure at $0$ of $\rev L_N (\la)$, just by noting that, in this case, $t=0$ in Proposition \ref{prop:invariantorders} since we are taking an empty state matrix. We stress that all these results hold for {\em any} rational matrix $Q_N (\la)$ either regular or singular. However, no information is provided on the finite poles of $Q_N (\la),$ and some of them could also be zeros.  As explained above, this is not an issue in \cite{nlep}, since $\mathcal{P}_N$ is outside the target set $\Sigma$. Nevertheless, at the cost of imposing extra hypotheses, we will solve this problem in Section \ref{sec.NLEIGSbasic2} for completeness and also because it is of interest for the theory of REPs.

\begin{rem} {\rm Let $d_N(\la)$ and $d_{N-1}(\la)$ be the denominators of $b_N(\la)$ and $b_{N-1}(\la)$ in \eqref{eq_interpolant}, respectively. Then, under the hypothesis $\xi_i \ne \sigma_j$, $1 \leq i \le N$, $0 \leq j \le N-1$, $L_N (\la)$ is a strong block minimal bases pencil (recall Definition \ref{def:minlinearizations}) associated with the {\em polynomial matrix} $d_N (\la) Q_N(\la)$. This follows easily from the facts that $K_N (\la)$ in \eqref{eq.LN} is a minimal basis with all its row degrees equal to one, that $$\widehat{N}_N (\la) := \beta_N \, d_{N-1} (\la) \, \left[
		\begin{array}{cccc}
		b_0(\la) & b_1(\la) & \cdots & b_{N-1} (\la)
		\end{array}
		\right] \otimes I_m$$ is a minimal basis dual to $K_N (\la)$ with all its row degrees equal to $N-1$, and that $M_N (\la) \widehat{N}_N (\la)^T = d_N (\la) Q_N(\la)$. Thus, using the results stated in the paragraph after Definition \ref{def:minlinearizations}, we get that $L_N(\la)$ is a $N$-strong linearization of the polynomial matrix $d_N (\la) Q_N(\la)$ with empty state matrix. Since $d_N (\la) Q_N(\la)$ and $Q_N(\la)$ are equivalent in $\FF \setminus \mathcal{P}_N$, we obtain again the result in Theorem \ref{th.NLEIGSbasic1}(b) through a different path which requires to use extra hypotheses.
	}
\end{rem}

\subsection{The NLEIGS low rank problem from the point of view of block full rank pencils} \label{sec.NLEIGSlowrank1} The second family of rational matrices considered in \cite{nlep} comes from approximating NLEPs, $A(\la) x = 0$, such that the associated matrix $A(\la)$ is the sum of a polynomial matrix plus a matrix of the form $\sum_{i=1}^{n}C_{i}f_{i}(\la)$, where the constant matrices $C_i$ have much smaller rank than the size of $A(\la)$ and $f_i (\la)$ are scalar nonlinear functions of $\la$. This type of NLEPs arise in several applications \cite{guttel-tisseur-2017} and are approximated in \cite[eq. (6.2)]{nlep} by a family of rational matrices of the form
\begin{equation}\label{eq.QNlowrank}
\widetilde{Q}_{N}(\la) =\displaystyle\sum_{i=0}^{p} b_{i}(\la) \, \widetilde{D}_{i} + \displaystyle\sum_{i=p+1}^{N} b_{i}(\la) \, \widetilde{L}_{i} \, \widetilde{U}^T \in \FF (\la)^{m\times m},
\end{equation}
where $b_0 (\la) , \ldots, b_N (\la)$ are the scalar rational functions in \eqref{eq_interpolant}, $\widetilde{D}_{0}, \ldots , \widetilde{D}_{p} \in \FF^{m \times m}$, $\widetilde{L}_{p+1} , \ldots , \widetilde{L}_{N} \in \FF^{m \times r}$ and $\widetilde{U} \in \FF^{m \times r}$ are constant matrices, and $r \ll m$. For the functions in \eqref{scalar_functions}, let us consider the simpler notation $h_{i}:=h_{i}(\la)$ and $g_{i}:=g_{i}(\la)$. Then, in order to solve the REP $\widetilde{Q}_{N}(\la) y = 0$ efficiently by taking advantage of the low rank structure of $Q_N (\la),$ the following pencil is introduced in \cite[Sec. 6.4]{nlep}:
\begin{equation} \label{eq.LNlowrank}
\widetilde{L}_{N}(\lambda)=\left[\begin{array}{c}
\widetilde{M}_{N}(\lambda)\\
\widetilde{K}_{N}(\lambda)
\end{array}\right],
\end{equation}
where
\begin{align*}
\widetilde{M}_{N}(\lambda) & =  \left[\begin{array}{cccccccc} \frac{g_N}{\beta_N} \widetilde{D}_{0}   & \frac{g_N}{\beta_N} \widetilde{D}_{1}   & \cdots &\frac{g_N}{\beta_N}\widetilde{D}_{p}  & \frac{g_N}{\beta_N} \widetilde{L}_{p+1} & \cdots & \frac{g_N}{\beta_N}\widetilde{L}_{N-2} & \frac{g_N}{\beta_N}\widetilde{L}_{N-1}+\frac{h_{N-1}}{\beta_{N}}\widetilde{L}_{N} \end{array}\right] \, ,
\end{align*}
and
\begin{equation*}
\widetilde{K}_{N}(\lambda)  = \left[\begin{array}{ccccccccc}
-h_0I_{m} & g_{1}I_{m}  \\
	&  \ddots  & \ddots  \\
	& &-h_{p-1}I_{m} & g_{p}I_{m}
	\\
 & & & -h_{p}\widetilde{U}^{T} &
	g_{p+1}I_{r}  \\
 & & & &	-h_{p+1}I_{r} & g_{p+2}I_{r} \\
 & & & & &	\ddots  & \ddots  \\
	& & & & & &-h_{N-2}I_{r} & g_{N-1}I_{r} \end{array}\right].
\end{equation*}	
A result analogous to Theorem \ref{th.NLEIGSbasic1} can be proved for the pencil $\widetilde{L}_N(\la)$ and the matrix $\widetilde{Q}_N (\la)$. This is accomplished in Theorem \ref{th.NLEIGSlowrank1}. We remark, nevertheless, that the result concerning the linearizations at $\infty$ is weaker in Theorem \ref{th.NLEIGSlowrank1} than in Theorem \ref{th.NLEIGSbasic1}. This is an unavoidable consequence of the used approach and the low rank structure of $\widetilde{Q}_N (\la)$.

\begin{theo} \label{th.NLEIGSlowrank1}
	Let $\widetilde{Q}_N (\la)$ be the rational matrix in \eqref{eq.QNlowrank} and $\widetilde{L}_N(\la)$ be the pencil in \eqref{eq.LNlowrank}. Let $\mathcal{P}_N$ and $i_N$ be, respectively, the set of finite poles and the number of infinite poles in the list $(\xi_1, \xi_2, \ldots , \xi_N)$. Then, the following statements hold:
	\begin{enumerate}
		\item[\rm (a)] $\widetilde{L}_N (\la)$ partitioned as in \eqref{eq.LNlowrank} is a block full rank pencil with only one block column associated with $\widetilde{Q}_N(\la)$ in $\FF \setminus \mathcal{P}_N$.
		\item[\rm (b)] $\widetilde{L}_N (\la)$ is a linearization with empty state matrix of $\widetilde{Q}_N(\la)$ in $\FF \setminus \mathcal{P}_N$.
		\item[\rm (c)] If, in addition, the poles $\xi_{p+1}, \xi_{p+2}, \ldots , \xi_{N-1}$ are all finite, then $\widetilde{L}_N (\la)$ is a linearization with empty state matrix of $\widetilde{Q}_N(\la)$ at $\infty$ of grade $i_N$.
	\end{enumerate}
\end{theo}
\begin{proof} The proof is similar to that of Theorem \ref{th.NLEIGSbasic1} with some differences coming from the presence of the low rank term in $\widetilde{Q}_N (\la)$. It is immediate to check that
	\begin{equation}\label{eq.dualrational2}
	\widetilde{N}_N (\la) = \frac{1}{1-\frac{\la}{\xi_N}} \, \left[
	\begin{array}{cccccc}
	b_0(\la) I_m & \cdots & b_{p} (\la) I_m&  b_{p+1} (\la) \widetilde{U} & \cdots &  b_{N-1} (\la)\widetilde{U}
	\end{array}
	\right]
	\end{equation}
	is a rational basis dual to $\widetilde{K}_N (\la)$, that $\widetilde{K}_N (\la)$ and $\widetilde{N}_N (\la)$ have both full row rank in $\FF\setminus \mathcal{P}_N$ and that $\widetilde{M}_N (\la) \widetilde{N}_N(\la)^T = \widetilde{Q}_N (\la)$. Thus, parts (a) and (b) follow from Theorem \ref{th:1blockfullrank}.
	
	In order to prove part (c), note first that $\mbox{rev}_1 \, \widetilde{K}_N (\la)$ has full row rank at zero as a consequence of the fact that the poles $\xi_{p+1}, \xi_{p+2}, \ldots , \xi_{N-1}$ are all finite. We now consider the rational matrix $\mbox{rev}_{i_N - 1} \, \widetilde{N}_N (\la)=\la^{i_N - 1}  \widetilde{N}_N\left( \frac{1}{\la} \right) $, which is of the form
	$$\la^{i_{N }-1}  \widetilde{N}_N \left( \frac{1}{\la} \right) =  \left[
	\begin{array}{ccccccc} \displaystyle
	* &  \cdots &
	* &   \frac{\la}{\la-1/\xi_N}\la^{i_{N}-1}b_{p}\left( \frac{1}{\la} \right) I_m &	* &  \cdots &
	*
	\end{array}
	\right],$$
	 where the entries $*$ are defined at $0.$ Denote by $i_p$ the number of infinite poles in the list $(\xi_1, \xi_2, \ldots , \xi_p)$. Then, $b_{p}\left( \frac{1}{\la} \right)=\frac{1}{\la^{i_{p}}}\tilde c(\la)$ for a certain rational function $\tilde c(\la)$ with $\tilde c(0)\neq 0.$ Taking into account that the poles $\xi_{p+1}, \xi_{p+2}, \ldots , \xi_{N-1}$ are all finite, we have that $i_{p}=i_{N}$ if $\xi_{N}\neq \infty,$ and $i_{p}=i_{N}-1$ if $\xi_{N}=\infty.$ Therefore, $\mbox{rev}_{i_N - 1} \, \widetilde{N}_N (\la)$ has full row rank at $0$ because $\tilde c(0) \ne 0$. Thus, part (c) follows from Theorem \ref{th:2blockfullrank}. \end{proof}

A discussion similar to the one in the last paragraph of Section \ref{sec.NLEIGSbasic1} can be developed on the basis of Theorem \ref{th.NLEIGSlowrank1}. The details are omitted for brevity. The open problem corresponding to the information of the finite poles will be solved in Section \ref{sec.NLEIGSlowrank2}.

\subsection{The NLEIGS basic problem from the point of view of polynomial system matrices} \label{sec.NLEIGSbasic2} As discussed previously, the approach presented in Section \ref{sec.NLEIGSbasic1} to the NLEIGS pencil $L_N (\la)$ in \eqref{eq.LN} considers $L_N (\la)$ as a linearization with empty state matrix and, thus, it does not provide any information on the finite poles of $Q_N (\la)$. In order to get this information, we need to identify a convenient square regular submatrix $A_N (\la)$ of $L_N(\la)$ that may be used as a state matrix. The block structure of $L_N (\la)$ makes it not possible to find such a matrix $A_N (\la)$ in a way that it includes the information of all the potential poles $(\xi_1 ,\ldots , \xi_N)$. This is related with the comment included in \cite[p. A2849]{nlep} on the fact that $\xi_N$ plays a special role and that it is convenient to choose $\xi_N = \infty$. In what follows we will {\em not} assume that $\xi_N = \infty$, though the obtained results are simpler and stronger under such assumption, but we will focus on getting information on the finite poles in $(\xi_1 ,\ldots , \xi_{N-1})$. With this spirit, we consider the following partition of $L_N (\la)$ in \eqref{eq.LN}, where $A_N(\la)$ will play the role of the state matrix,
\begin{equation}\label{eq.LN2part} L_N (\la) =:
\left[\begin{array}{c|c}
D_N(\la) & -C_N (\la) \phantom{\Big|} \\  \hline \phantom{\Big|}
B_N(\la) & A_N (\la)
\end{array}\right], \quad \mbox{where $D_N (\la) = \left(1-\frac{\la}{\xi_{N}}\right) \, D_0$,}
\end{equation}
and the rest of the blocks are easily described from the blocks in \eqref{eq.LN} as follows: $B_N (\la)$ is the first block column of $K_N (\la)$, $-C_N (\la)$ is obtained by removing the first block of $M_N (\la)$ and $A_N (\la)$ is obtained by removing the first block column of $K_N (\la)$.

The next technical lemma reveals which is the transfer function matrix of $L_N (\la)$, with the partition above, and establishes necessary and sufficient conditions for $L_N (\la)$ to be minimal in the whole field $\FF$. Of course, the conditions in Lemma \ref{lemm.NLEIGSpoly1}(b) come from imposing that $\left[ \begin{array}{cc} B_N(\la_0) & A_N (\la_0) \end{array} \right] \in \FF^{m (N-1) \times m N}$ and
$\left[ \begin{array}{cc} -C_N(\la_0)^T & A_N (\la_0)^T \end{array} \right]^T \in \FF^{m N \times m (N-1)}$ have, respectively, full row and column rank for any $\la_0 \in \FF$, but have an important advantage with respect to these direct conditions for minimality. More precisely, the conditions in Lemma \ref{lemm.NLEIGSpoly1}(b) require to evaluate the rational matrix $R_N(\la)$ of size $m \times m$, which for practical problems is much smaller than $m (N-1) \times m N$.

\begin{lem} \label{lemm.NLEIGSpoly1} Let us consider the pencil $L_N (\la)$ in \eqref{eq.LN} as a polynomial system matrix with state matrix $A_N (\la)$, where $A_N (\la)$ is defined through the partition \eqref{eq.LN2part}, and let $Q_N (\la)$ be the rational matrix in \eqref{eq.QN}. Then the following statements hold:
	\begin{enumerate}
		\item[\rm (a)] The transfer function matrix of $L_N (\la)$ is $\beta_0 \left(1-\frac{\la}{\xi_{N}}\right) Q_N (\la).$
		\item[\rm (b)] Let us define the rational matrix $R_N (\la) :=  (Q_N (\la) - b_0 (\la) D_0)/ b_N(\la)$, whose explicit expression is
		\begin{equation}\label{eq.RN}
		R_N (\la) = \sum_{j=1}^{N-1} \, \left( \prod_{k=j+1}^{N} \frac{g_k (\la)}{ h_{k-1}(\la)}\right) \, D_j \, + \, D_N \in \FF (\la)^{m\times m} ,
		\end{equation}
		let $\mathcal{P}_{N-1}$ be the set of finite poles in the list $(\xi_1 , \xi_2 , \ldots , \xi_{N-1})$, and assume $\xi_i \ne \sigma_j$, $1\leq i \leq N$, $0\leq j \leq N-1$.
		Then, $L_N (\la)$ is minimal in $\FF$ if and only if the matrix $R_N (\xi_k)\in \FF^{m\times m}$ is nonsingular for all $\xi_k \in \mathcal{P}_{N-1}$.
	\end{enumerate}
\end{lem}
\begin{proof}
	Part (a). According to \eqref{eq.LN2part}, the transfer function matrix of $L_N(\la)$ is $D_N (\la) + C_N (\la) A_N (\la)^{-1} B_N (\la)$. The computation of this transfer function is very easy because $B_N (\la) = \left[ \begin{array}{cccc}
	-h_0(\la) I_m & 0 & \cdots & 0
	\end{array}\right]^T$, which implies that only the first block column of $A_N (\la)^{-1}$ is needed. It is immediate to check that this first block column is
	\[
	\frac{1}{b_1(\la) g_1 (\la)} \left[ \begin{array}{ccc}
	b_1(\la) & \cdots & b_{N-1} (\la)
	\end{array}\right]^T \otimes I_m \, .
	\]
	The rest of the proof of part (a) is just an elementary and short algebraic manipulation.
	
	Part (b). The proof is elementary but long. Thus, it is postponed to \ref{appendix1}.
\end{proof}

We emphasize that Lemma \ref{lemm.NLEIGSpoly1}(a) holds for {\rm any} rational matrix $Q_N(\la)$ expressed as in \eqref{eq.QN} without imposing any extra condition. Moreover, the constant matrix $A_N (\la_0)$ is invertible for any $\la_0 \in \FF \setminus \mathcal{P}_{N-1}$ and, so, $L_N(\la)$ is minimal in $\FF \setminus \mathcal{P}_{N-1}$. Combining these results with the fact that $Q_N (\la)$ and $\beta_0 \left(1-\frac{\la}{\xi_{N}}\right) Q_N (\la)$ are equivalent in $\FF$ if $\xi_N = \infty$ or in $\FF\setminus \{\xi_N\}$ if $\xi_N$ is finite, we immediately obtain from Definitions \ref{def_pointstronglin} and \ref{def:linsubset} that $L_N (\la)$ is a linearization of $Q_N (\la)$ with state matrix $A_N (\la)$ in $\FF\setminus \mathcal{P}_{N}$, which is a result analogous to Theorem \ref{th.NLEIGSbasic1}(b). This approach, of course, does not give any information on the finite poles of $Q_N (\la)$, because the finite eigenvalues of $A_N (\la)$ coincide with $\mathcal{P}_{N-1}$. Such information is obtained from the next result, which is the main result of this section and is a corollary of Lemma \ref{lemm.NLEIGSpoly1}.

\begin{theo} \label{theo.NLEIGSpoly1} Let $Q_N (\la)$ be the rational matrix in \eqref{eq.QN}, $L_N(\la)$ be the pencil in \eqref{eq.LN}, $A_N (\la)$ be the submatrix of $L_N(\la)$ in \eqref{eq.LN2part}, and $R_N (\la)$ be the rational matrix in \eqref{eq.RN}. Consider $\mathcal{P}_{N-1}$ the set of finite poles in the list $(\xi_1 , \xi_2 , \ldots , \xi_{N-1})$, and assume $\xi_i \ne \sigma_j$, $1\leq i \leq N$, $0\leq j \leq N-1$. If $R_N (\xi_k)\in \FF^{m\times m}$ is nonsingular for every $\xi_k \in \mathcal{P}_{N-1}$, then $L_N (\la)$ is a linearization of $Q_N (\la)$ with state matrix $A_N (\la)$ in $\FF$, if $\xi_N = \infty$, or in $\FF\setminus \{\xi_N\}$, if $\xi_N$ is finite.
\end{theo}
\begin{proof}
	Under the hypotheses of Theorem \ref{theo.NLEIGSpoly1}, $L_N(\la)$ is minimal in $\FF$. Moreover, its transfer function matrix, i.e., $\beta_0 \left(1-\frac{\la}{\xi_{N}}\right) Q_N (\la)$ is equivalent to $Q_N (\la)$ in $\FF$, if $\xi_N = \infty$, or in $\FF\setminus \{\xi_N\}$, if $\xi_N$ is finite. The result follows immediately from Definitions \ref{def_pointstronglin} and \ref{def:linsubset} with $s_1 = s_2 =0$.
\end{proof}

 We emphasize that the hypotheses that the constant matrices $R_N (\xi_k)$ in Theorem \ref{theo.NLEIGSpoly1} are nonsingular are not mentioned at all in \cite{nlep}, but, fortunately, are generic, in the sense that they are satisfied by almost all regular rational matrices $Q_N (\la)$ expressed as in \eqref{eq.QN}.

\begin{rem}\rm
	Under the conditions of Theorem \ref{theo.NLEIGSpoly1}, the pole elementary divisors of $Q_N(\la)$ in $\FF$, if $\xi_N = \infty$, or in $\FF\setminus \{\xi_N\}$, if $\xi_N$ is finite, are the elementary divisors of $A_N (\la)$, as a consequence of Theorem \ref{theo:spectral}. These elementary divisors can be very easily determined as follows: first express $A_N (\la) = \widehat{A}_N (\la) \otimes I_m$; second note that if $\widehat{S}_N (\la)$ is the Smith form of $\widehat{A}_N (\la)$, then $\widehat{S}_N (\la) \otimes I_m$ is the Smith form of $A_N(\la)$; third, use the fact that $\xi_i \ne \sigma_j$, $1\leq i \leq N$, $0\leq j \leq N-1$, to prove that the greatest common divisor of all $(N-2) \times (N-2)$ minors of $\widehat{A}_N (\la)$ is equal to $1$, which implies, according to \cite[Ch. VI]{gant}, that there is only one invariant polynomial of $\widehat{S}_N (\la)$ different from $1$ and that is equal to
	$$
	p(\la) = c \, (1 - \la /\xi_1) \cdots (1 - \la /\xi_{N-1}) ,
	$$
	where $c\in \FF$ is a constant that makes $p(\la)$ monic. Finally, we get that $A_N (\la)$ has $m$ invariant polynomials different from $1$ all equal to $p(\la)$. This allows us to obtain easily the finite elementary divisors of $A_N(\la)$ and, thus, the finite pole elementary divisors of $Q_N (\la)$ (in $\F$ if $\xi_N=\infty$, or in $\FF\setminus \{\xi_N\}$ if $\xi_N$ is finite). In particular, they are of the form $(\la-\xi_i)^{\nu_i}$ and, in order to obtain the partial multiplicities $\nu_i,$ we have to take into account possible repetitions in $(\xi_1, \ldots , \xi_{N-1})$.
	Observe that the infinite $\xi_i$ for $i=1,\ldots,N-1$ do not contribute at all to the finite pole elementary divisors of $Q_N (\la).$ Moreover, if $\xi_N = \infty$, then we can state the compact and simple result that the $m$ denominators of the global Smith--McMillan form of $Q_N (\la)$ are all equal to $p(\la)$. However, with this choice of state matrix, there is no way of obtaining information on the pole structure of $\xi_N$ when it is finite. This is the reason why, even if $L_{N}(\la)$ is minimal in $\F$, $L_N (\la)$ is not a linearization of $Q_N (\la)$ in $\F.$
	
\end{rem}

\subsection{The NLEIGS low rank problem from the point of view of polynomial system matrices} \label{sec.NLEIGSlowrank2}
The results in this section are the counterpart for $\widetilde{Q}_N(\la)$ in \eqref{eq.QNlowrank} and $\widetilde{L}_N (\la)$ in \eqref{eq.LNlowrank} of those presented in Section \ref{sec.NLEIGSbasic2} for $Q_N(\la)$ and $L_N (\la)$. For brevity, we avoid in this section to introduce auxiliary comments similar to the corresponding ones in Section \ref{sec.NLEIGSbasic2} and just some relevant differences are remarked. The motivation of this section is to obtain from $\widetilde{L}_N (\la)$ information about the finite poles of $\widetilde{Q}_N(\la)$. For this purpose, we consider the following partition of $\widetilde{L}_N (\la)$ in \eqref{eq.LNlowrank}, where $\widetilde{A}_N(\la)$ will play the role of the state matrix,
\begin{equation}\label{eq.LNlowrank2part} \widetilde{L}_N (\la) =:
\left[\begin{array}{c|c}
\widetilde{D}_N(\la) & -\widetilde{C}_N (\la) \phantom{\Big|} \\  \hline \phantom{\Big|}
\widetilde{B}_N(\la) & \widetilde{A}_N (\la)
\end{array}\right], \quad \mbox{where $\widetilde{D}_N (\la) = \left(1-\frac{\la}{\xi_{N}}\right) \, \widetilde{D}_0$,}
\end{equation}
and the rest of the blocks are easily described from the blocks in \eqref{eq.LNlowrank} as follows: $\widetilde{B}_N (\la)$ is the first block column of $\widetilde{K}_N (\la)$, $-\widetilde{C}_N (\la)$ is obtained by removing the first block of $\widetilde{M}_N (\la)$, and $\widetilde{A}_N (\la)$ is obtained by removing the first block column of $\widetilde{K}_N (\la)$.

The next lemma is the counterpart of Lemma \ref{lemm.NLEIGSpoly1}. Note that the low rank structure in $\widetilde{Q}_N(\la)$ complicates the minimality conditions in part (b) of Lemma \ref{lemm.NLEIGSpoly2}, which are expressed in terms of matrices of size $(2 m + r) \times (m+r)$.

\begin{lem} \label{lemm.NLEIGSpoly2} Let us consider the pencil $\widetilde{L}_N (\la)$ in \eqref{eq.LNlowrank} as a polynomial system matrix with state matrix $\widetilde{A}_N (\la)$, where $\widetilde{A}_N (\la)$ is defined through the partition \eqref{eq.LNlowrank2part}, and let $\widetilde{Q}_N (\la)$ be the rational matrix in \eqref{eq.QNlowrank}. Then the following statements hold:
	\begin{enumerate}
		\item[\rm (a)] The transfer function matrix of $\widetilde{L}_N (\la)$ is $\beta_0 \left(1-\frac{\la}{\xi_{N}}\right) \widetilde{Q}_N (\la).$
		\item[\rm (b)] Let us define the rational matrices
		\begin{align*}
		\widetilde{R}_N^{(1)} (\la) & = \frac{g_N}{h_{N-1}} \left[\sum_{j=1}^{p-1} \, \left( \prod_{k=j+1}^{p} \dfrac{g_k }{h_{k-1}}\right) \, \widetilde{D}_j \, + \, \widetilde{D}_p \right] \in \FF (\la)^{m\times m} , \\
		\widetilde{R}_N^{(2)} (\la) & = \sum_{j=p+1}^{N-1} \, \left( \prod_{k=j+1}^{N} \dfrac{g_k }{h_{k-1}}\right) \, \widetilde{L}_j \, + \, \widetilde{L}_N \in \FF (\la)^{m\times r} \,
		\end{align*}
		and
		\begin{equation}\label{eq.RNlowrank} \!\!\!\!\!\!\!\!\!\!
		\widetilde{R}_N (\la) =
		\left[
		\begin{array}{c|c}
		\widetilde{R}_N^{(1)} (\la) & \widetilde{R}_N^{(2)} (\la) \\ \hline
		\!\!\! \left( {\displaystyle \prod_{i=1}^{p-1}} \dfrac{g_i }{h_{i}}\right)
		g_p  \, I_m & \phantom{\begin{array}{c}
			\Bigg|\\
			\Bigg|
			\end{array}} 0 \\ \hline \phantom{\begin{array}{c}
			\Bigg|\\
			\Bigg|
			\end{array}}
		-h_p  \, \widetilde{U}^T & \!\!  \left( {\displaystyle \prod_{i=p+1}^{N-2} } \dfrac{g_i}{h_{i}}\right)
		g_{N-1}  \, I_r
		\end{array}
		\right]  .
		\end{equation}
		Let $\mathcal{P}_{N-1}$ be the set of finite poles in the list $(\xi_1 , \xi_2 , \ldots , \xi_{N-1})$, and assume that $\rank \widetilde{U} = r$ and that $\xi_i \ne \sigma_j$, $1\leq i\leq N$, $0\leq j\leq N-1$.
		Then, $\widetilde{L}_N (\la)$ is minimal in $\FF$ if and only if the matrix $\widetilde{R}_N (\xi_k)\in \FF^{(2m+r)\times (m+r)}$ has full column rank for all $\xi_k \in \mathcal{P}_{N-1}$.
	\end{enumerate}
\end{lem}
\begin{proof}
	Part (a). The proof is similar to that of Lemma \ref{lemm.NLEIGSpoly1}(a) with some differences coming from the presence of the low rank term in $\widetilde{Q}_N (\la)$. According to \eqref{eq.LNlowrank2part}, the transfer function matrix of $\widetilde{L}_N(\la)$ is $\widetilde{D}_N (\la) + \widetilde{C}_N (\la) \widetilde{A}_N (\la)^{-1} \widetilde{B}_N (\la)$. The computation of this matrix is very easy because, again, $\widetilde{B}_N (\la) = \left[ \begin{array}{cccc}
	-h_0  I_m & 0 & \cdots & 0
	\end{array}\right]^T$, and only the first block column of $\widetilde{A}_N (\la)^{-1}$ is needed, which, in this case, is equal to
	\[
	\frac{1}{b_1(\la) g_1 } \left[ \begin{array}{cccccc}
	b_1(\la) I_m& \cdots & b_{p} (\la) I_m &  b_{p+1}(\la) \widetilde{U} & \cdots & b_{N-1} (\la) \widetilde{U}
	\end{array}\right]^T.
	\]

	Part (b). The proof is elementary but long. Thus, it is postponed to \ref{appendix2}.
\end{proof}

\begin{rem} \label{rem.simplicondlowrank}
	{\rm If, in addition to $\rank \widetilde{U}=r$ and $\xi_i \ne \sigma_j$, $1\leq i\leq N$, $0\leq j\leq N-1$, we assume that $\xi_1 = \cdots = \xi_p = \infty$, then the necessary and sufficient conditions for minimality in Lemma \ref{lemm.NLEIGSpoly2}(b) can be considerably simplified, since we get as an immediate corollary of Lemma \ref{lemm.NLEIGSpoly2}(b) that ``$\widetilde{L}_N (\la)$ is minimal in $\FF$ if and only if the matrix $\widetilde{R}_N^{(2)} (\xi_k)\in \FF^{m\times r}$ has full column rank for every $\xi_k \in \mathcal{P}_{N-1}$''. Note that the hypothesis $\xi_1 = \cdots = \xi_p = \infty$ implies that the ``no-low rank'' term $\sum_{i=0}^{p} b_i(\la) \widetilde{D}_i$ of $\widetilde{Q}_N (\la)$ is a polynomial matrix, as often happens in NLEPs \cite{nlep}.
		
		Observe also that if $\widehat{R}_N (\la)$ is the $(m+r) \times (m+r)$ matrix obtained from $\widetilde{R}_N (\la)$ in \eqref{eq.RNlowrank} by removing the second block row, then under the assumptions $\rank \widetilde{U}=r$ and $\xi_i \ne \sigma_j$, $1\leq i\leq N$, $0\leq j\leq N-1$, we get, as another immediate corollary of Lemma \ref{lemm.NLEIGSpoly2}(b), the following sufficient condition for minimality: ``if $\widehat{R}_N (\xi_k)\in \FF^{(m+r)\times (m+r)}$ is invertible for every $\xi_k \in \mathcal{P}_{N-1}$, then $\widetilde{L}_N (\la)$ is minimal in $\FF$''.
	}
\end{rem}

Theorem \ref{theo.NLEIGSpoly2} is the main result in this section and is an easy corollary of Lemma \ref{lemm.NLEIGSpoly2}. Its proof is omitted because is very similar to that of Theorem \ref{theo.NLEIGSpoly1}.

\begin{theo} \label{theo.NLEIGSpoly2} Let $\widetilde{Q}_N (\la)$ be the rational matrix in \eqref{eq.QNlowrank}, $\widetilde{L}_N(\la)$ be the pencil in \eqref{eq.LNlowrank}, $\widetilde{A}_N (\la)$ be the submatrix of $\widetilde{L}_N(\la)$ in \eqref{eq.LNlowrank2part}, and $\widetilde{R}_N (\la)$ be the rational matrix in \eqref{eq.RNlowrank}. Consider $\mathcal{P}_{N-1}$ the set of finite poles in the list $(\xi_1 , \xi_2 , \ldots , \xi_{N-1})$. If $\rank \widetilde{U}=r$, $\xi_i \ne \sigma_j$, $1\leq i\leq N$, $0\leq j\leq N-1$, and $\widetilde{R}_N (\xi_k)\in \FF^{(2m+r)\times (m+r)}$ has full column rank for every $\xi_k \in \mathcal{P}_{N-1}$, then $\widetilde{L}_N (\la)$ is a linearization of $\widetilde{Q}_N (\la)$ with state matrix $\widetilde{A}_N (\la)$ in $\FF$, if $\xi_N = \infty$, or in $\FF\setminus \{\xi_N\}$, if $\xi_N$ is finite.
\end{theo}

Finally, note that the conditions in Theorem \ref{theo.NLEIGSpoly2} on the full column rank of the matrices $\widetilde{R}_N (\xi_k)$ can be simplified as in Remark \ref{rem.simplicondlowrank} under extra hypotheses.

\section{Conclusions and future work}\label{sect:con}
A theory of local linearizations of rational matrices has been carefully presented in this paper, by developing as starting point the extension of Rosenbrock's minimal polynomial system matrices to a local scenario. Moreover, this theory has been applied to a number of pencils that have appeared recently in some influential papers on solving numerically NLEPs by combining rational approximations, linearizations of the resulting rational matrices, and efficient numerical algorithms for generalized eigenvalue problems adapted to the structure of such linearizations. It has been emphasized throughout the paper that the theory of local linearizations allows us to view these pencils, and to explain their properties, from rather different perspectives, which depend on the particular choice of the submatrix of the pencil to be considered as state matrix. In particular, we have seen that the choice of an empty state matrix is simple and adequate for those rational matrices and pencils arising in NLEPs, when the poles are already known from the approximation process. This has led us to define and analyze the very general family of block full rank pencils, as a template that covers many of the pencils, available in the literature, that linearize the rational approximations in the corresponding target set. We plan to extend these ideas in \cite{local2}, where other ways to choose the state matrices will be explored. In addition, the results in this paper and also the new ones in \cite{local2} will be applied to the pencils defined in \cite{automatic}, as well as to other pencils. Finally, we also plan to study numerical properties of some of the linearizations analyzed in this work. In particular, given a linearization of the REP in a set, it is important to study the backward stability in terms of the structure of the rational matrix defining the REP when applying a numerical method to compute the eigenvalues of the linearization. In addition, we plan to investigate the conditioning of eigenvalues, that is, the sensitivity to perturbations, both in the original REP and its linearization, of a zero that is not a pole of the rational matrix.

\appendix
\section{Proof of Lemma \ref{lemm.NLEIGSpoly1}(b)} \label{appendix1} Let us consider $L_N(\la)$ partitioned as in \eqref{eq.LN2part} and as a polynomial system matrix with state matrix $A_N(\la)$. Recall throughout the proof that the parameters $\beta_0, \beta_1, \ldots , \beta_N$ are all different from zero.  Observe first that $\xi_i \ne \sigma_j$, $1\leq i \leq N$ and $0\leq j \leq N-1$, implies that $\left[\begin{array}{cc}
B_N (\la_0) & A_N (\la_0)
\end{array} \right]$ has full row rank for any $\la_0 \in \FF$. On the other hand, if we define
\begin{equation}\label{eq.appZN} Z_N (\la) :=
\left[\begin{array}{c}
-C_N (\la) \\ A_N (\la)
\end{array} \right],
\end{equation}
then $Z_N (\la_0)$ has full column rank for every $\la_0 \in \FF \setminus \mathcal{P}_{N-1}$, because $A_N (\la_0)$ is invertible in $\FF \setminus \mathcal{P}_{N-1}$. Therefore, combining the discussion above with Definition \ref{def_minimalpolsysmatsubset}, we obtain that $L_N (\la)$ is minimal in $\FF$ if and only if $Z_N (\xi_k)$ has full column rank for every $\xi_k \in \mathcal{P}_{N-1}$. The rest of the proof proceeds as follows: we will find a rational matrix $S_N (\la)$ such that is equivalent to $Z_N (\la)$ in $\mathcal{P}_{N-1}$ and has a simple structure that allows us to see that $S_N(\xi_k)$ (and, so, $Z_N (\xi_k)$) has full column rank for every $\xi_k \in \mathcal{P}_{N-1}$ if and only if $R_N(\xi_k)$ is invertible for every $\xi_k \in \mathcal{P}_{N-1}$, where $R_N (\la)$ is the rational matrix in \eqref{eq.RN}.

For brevity, we use the notation $g_i:=g_{i}(\la)$ and $h_{i}:=h_{i}(\la)$ for the scalar functions in \eqref{scalar_functions}. In addition, $Z_N (\la)$ in \eqref{eq.appZN} is partitioned as
\begin{equation}\label{eq.appZN2}
Z_N (\la) =:
\left[\begin{array}{cc}
Z_{11} (\la) & Z_{12} (\la) \\ Z_{21} (\la) & Z_{22} (\la)
\end{array} \right],
\end{equation}
where
\begin{align*}
& Z_{11} (\la) =
\left[\begin{array}{ccccc}
\frac{g_N}{\beta_N} D_1 & \frac{g_N}{\beta_N} D_2 & \cdots & \cdots & \frac{g_N}{\beta_N} D_{N-2}  \\  g_1 I_m & 0 & \cdots & \cdots & 0
\end{array} \right],
\quad
Z_{12} (\la) =
\left[\begin{array}{c}
\frac{g_N}{\beta_N} D_{N-1} + \frac{h_{N-1}}{\beta_N} D_N  \\ 0
\end{array} \right],
\\
& Z_{21} (\la) =
\left[\begin{array}{ccccc}
-h_1 & g_2 &  & &  \\
& -h_2 & g_3 &  &  \\
& & \ddots & \ddots & \\
& & & -h_{N-3} & g_{N-2} \\
&&&& -h_{N-2}
\end{array} \right] \otimes I_m, \quad
Z_{22} (\la) =
\left[\begin{array}{c} 0 \\ \vdots \\ \vdots \\ 0 \\
g_{N-1} I_m
\end{array} \right].
\end{align*}
Note that the matrix $Z_{21} (\la)$ is invertible in $\mathcal{P}_{N-1}$ and that the last block column of $Z_{21} (\la)^{-1}$ is
\begin{equation}\label{eq.appY22}
Y_{22} (\la) :=
-\left[\begin{array}{ccccc}
\displaystyle  \!  \! \frac{1}{h_1} \prod_{i=2}^{N-2} \frac{g_i}{h_i} , &
\displaystyle  \! \frac{1}{h_2} \prod_{i=3}^{N-2} \frac{g_i}{h_i} , &  \!  \!
\cdots ,&
\displaystyle  \! \frac{1}{h_{N-3}} \frac{ g_{N-2}}{h_{N-2}} , &
\displaystyle  \! \frac{1}{h_{N-2}}
\end{array} \right]^T \, \otimes \, I_m \, .
\end{equation}

Next, a sequence of equivalence transformations in $\mathcal{P}_{N-1}$ are applied to $Z_N(\la)$. Such transformations are described by using the notation in \eqref{eq.appZN2} and \eqref{eq.appY22}, and the first one is
\[
Y_N (\la) :=\left[\begin{array}{c|c} I_{2m} & 0 \\ \hline
0 & Z_{21} (\la)^{-1} \end{array} \right] Z_N (\la) =
\left[\begin{array}{c|c}
Z_{11} (\la) & Z_{12} (\la) \\ \hline I_{(N-2) m} & g_{N-1} Y_{22} (\la)
\end{array} \right].
\]
The second transformation is designed to turn zero the second block row of $Z_{11} (\la)$ as follows
\begin{align*}
W_N (\la) & := \diag \left( I_m ,
\left[
\begin{array}{cc}
I_m & -g_1 I_m \\
0 & I_m
\end{array}
\right]
, I_{(N-3)m} \right) \, \, Y_N (\la)\\
&=
\left[
\begin{array}{c|c}
\begin{array}{ccc}
\frac{g_N}{\beta_N} D_1 & \cdots & \frac{g_N}{\beta_N} D_{N-2}  \\ 0  & \cdots & 0
\end{array} &
\begin{array}{c}
\frac{g_N}{\beta_N} D_{N-1} + \frac{h_{N-1}}{\beta_N} D_N  \\
\left(\prod_{i=1}^{N-2} \frac{g_i}{h_i}\right) g_{N-1} I_m
\end{array}
\\ \hline
I_{(N-2) m} &  g_{N-1} Y_{22} (\la)
\end{array}
\right] .
\end{align*}
The third transformation turns zero the block $ g_{N-1} Y_{22} (\la)$ of $W_N (\la)$ and performs a convenient scalar multiplication in its first block row. Such transformation is
\begin{align*}
X_N (\la) & :=  \left[
\begin{array}{cc}
\frac{\beta_N}{h_{N-1}} I_m & 0 \\
0 & I_{(N-1) m}
\end{array}
\right]
W_N (\la)
\left[
\begin{array}{c|c}
I_{(N-2)m} & -g_{N-1} Y_{22} (\la) \\ \hline
0 & I_m
\end{array}
\right]\\
& = \left[
\begin{array}{c|c}
\begin{array}{ccc}
\frac{g_N}{h_{N-1}} D_1 & \cdots & \frac{ g_N}{h_{N-1}} D_{N-2}  \\ 0  & \cdots & 0
\end{array} &
\begin{array}{c}
R_{N}(\la)  \\
\left(\prod_{i=1}^{N-2} \frac{g_i}{h_i}\right) g_{N-1} I_m
\end{array}
\\ \hline
I_{(N-2) m} & 0
\end{array}
\right] ,
\end{align*}
where $R_{N}(\la)$ is the rational matrix in \eqref{eq.RN}. The last transformation makes zero the first $N-2$ blocks of size $m\times m$ in the first block row of $X_N (\la)$ and yields the announced matrix $S_N (\la)$ equivalent to $Z_N (\la)$ in $\mathcal{P}_{N-1}$. More precisely,
\begin{align*}
S_N (\la)& :=  \left[
\begin{array}{c|c}
\begin{array}{cc}
I_m & 0 \\
0 & I_m
\end{array}
&
\begin{array}{ccc}
-\frac{ g_N}{h_{N-1}} D_1 & \cdots & -\frac{g_N}{h_{N-1}} D_{N-2}  \\ 0  & \cdots & 0
\end{array}
\\ \hline
0 &  I_{(N-2) m}
\end{array}
\right] \; X_N (\la)
\\
& = \left[
\begin{array}{c|c}
\begin{array}{ccc}
0 & \cdots & 0 \\ 0  & \cdots & 0
\end{array} &
\begin{array}{c}
R_N (\la)  \\
\left(\prod_{i=1}^{N-2} \frac{g_i}{h_i}\right) g_{N-1} I_m
\end{array}
\\ \hline
I_{(N-2) m} & 0
\end{array}
\right].
\end{align*}
The block $H(\la) := \left(\prod_{i=1}^{N-2} \frac{g_i}{h_i}\right) g_{N-1} I_m$ of $S_N (\la)$ satisfies $H(\xi_k) = 0$ for all $\xi_k \in \mathcal{P}_{N-1}$. Therefore, $S_N (\xi_k)$ (and, so, $Z_N (\xi_k)$) has full column rank for every $\xi_k \in \mathcal{P}_{N-1}$ if and only if $R_N (\xi_k)$ is invertible for all $\xi_k \in \mathcal{P}_{N-1}$, and the result is proved.

\section{Proof of Lemma \ref{lemm.NLEIGSpoly2}(b)} \label{appendix2} The first part of the proof is completely analogous to the first part of the proof of Lemma \ref{lemm.NLEIGSpoly1}(b). So, some details are ommited. Let us consider $\widetilde{L}_N(\la)$ partitioned as in \eqref{eq.LNlowrank2part} and as a polynomial system matrix with state matrix $\widetilde{A}_N(\la)$. Then the hypotheses $\rank \widetilde{U} = r$ and $\xi_i \ne \sigma_j$, $1\leq i \leq N$ and $0\leq j \leq N-1$, imply that $\left[\begin{array}{cc}
\widetilde{B}_N (\la_0) & \widetilde{A}_N (\la_0)
\end{array} \right]$ has full row rank for any $\la_0 \in \FF$. Also, if we define
\begin{equation}\label{eq.apptildeZN} \widetilde{Z}_N (\la) :=
\left[\begin{array}{c}
-\widetilde{C}_N (\la) \\ \widetilde{A}_N (\la)
\end{array} \right],
\end{equation}
then $\widetilde{Z}_N (\la_0)$ has full column rank for every $\la_0 \in \FF \setminus \mathcal{P}_{N-1}$, because $\widetilde{A}_N (\la_0)$ is invertible in $\FF \setminus \mathcal{P}_{N-1}$. Therefore, $\widetilde{L}_N (\la)$ is minimal in $\FF$ if and only if $\widetilde{Z}_N (\xi_k)$ has full column rank for every $\xi_k \in \mathcal{P}_{N-1}$. In the rest of the proof we will find a rational matrix $\widetilde{S}_N (\la)$ such that is equivalent to $\widetilde{Z}_N (\la)$ in $\mathcal{P}_{N-1}$ and that allows us to see that $\widetilde{S}_N(\xi_k)$ (and, so, $\widetilde{Z}_N (\xi_k)$) has full column rank for every $\xi_k \in \mathcal{P}_{N-1}$ if and only if $\widetilde{R}_N(\xi_k)$ in \eqref{eq.RNlowrank} has full column rank for every $\xi_k \in \mathcal{P}_{N-1}$. We advance that this second part of the proof is considerably more involved than the corresponding part of the proof of Lemma \ref{lemm.NLEIGSpoly1}(b), as a consequence of the presence in $\widetilde{Z}_N (\la)$ of two kinds of blocks, one kind corresponding to the ``full rank'' part of $\widetilde{Q}_N (\la)$, i.e., the first summation in \eqref{eq.QNlowrank}, and another kind corresponding to the ``low rank'' part of $\widetilde{Q}_N (\la)$. Nevertheless, the equivalence transformations in $\mathcal{P}_{N-1}$ used in the sequel are similar to those in the proof of Lemma \ref{lemm.NLEIGSpoly1}(b), and many details will be omitted for brevity. Recall that we use the notation in \eqref{scalar_functions} omitting the dependence on $\la$ for simplicity, i.e., we write simply $g_i$ and $h_j$.

The first two equivalence transformations in $\mathcal{P}_{N-1}$ that we perform affect only to the last $N-1-p$ block rows of $\widetilde{Z}_N (\la)$, i.e., those containing $I_r$ matrices. Thus, in this part of the proof, it is convenient to partition $\widetilde{Z}_N (\la)$ as
\[
\widetilde{Z}_N (\la) =:
\left[\begin{array}{c}
\widetilde{Z}_N^{(1)} (\la)  \\ \widetilde{Z}_N^{(2)} (\la)
\end{array} \right],
\]
with $\widetilde{Z}_N^{(1)} (\la)$ comprising the first $p+1$ block rows of $\widetilde{Z}_N (\la)$. In order to construct the first equivalence transformation, we pay attention to the following submatrix of $\widetilde{Z}_N^{(2)} (\la)$,
\[
\widetilde{H}_N (\la) :=
\left[\begin{array}{ccccc}
-h_{p+1} & g_{p+2} &  & &  \\
& -h_{p+2} & g_{p+3} &  &  \\
& & \ddots & \ddots & \\
& & & -h_{N-3} & g_{N-2} \\
&&&& -h_{N-2}
\end{array} \right] \otimes I_r,
\]
which is invertible in $\mathcal{P}_{N-1}$ and has the same structure as $Z_{21} (\la)$ in \eqref{eq.appZN2}. The last block column of $\widetilde{H}_N(\la)^{-1}$ has a structure similar to \eqref{eq.appY22} and is denoted by $J(\la)$. Then, the first two equivalence transformations are
\begin{align*}
\widetilde{W}_N (\la) & := \diag \left( I_{(p+1)m} ,
\left[
\begin{array}{cc}
I_r & -g_{p+1} I_r \\
0 & I_r
\end{array}
\right]
, I_{(N-3-p)r} \right) \, \diag (I_{(p+1)m +r} , \widetilde{H}_N (\la)^{-1}) \,\, \widetilde{Z}_N (\la) \\ & =: \left[\begin{array}{c}
\widetilde{Z}_N^{(1)} (\la)  \\ \widetilde{W}_N^{(2)} (\la)
\end{array} \right],
\end{align*}
where
\[
\widetilde{W}_N^{(2)} (\la) =
\left[\begin{array}{ccc}
e_p^T \otimes (-h_p \widetilde{U}^T) & 0 & \left(\prod_{i=p+1}^{N-2} \frac{g_i}{h_i}\right) g_{N-1} I_r \\
0 & I_{(N-2-p)r} & g_{N-1} J(\la)
\end{array} \right] ,
\]
with $e_p^T = [0 \,\, \cdots \,\, 0 \,\, 1] \in \FF^{1 \times p}$. In order to describe the outcome of the next two transformations, we consider the following submatrix of $\widetilde{Z}_N^{(1)} (\la)$:
\[
\widetilde{E}_N (\la) :=\left[
\begin{array}{ccccc}
g_1  & &  &  & \\
-h_1 & g_2  & & &  \\
&  \ddots & & \ddots &  \\
&&& -h_{p-1}  & g_p
\end{array} \right]\otimes I_m.
\]
The next equivalence transformations in $\mathcal{P}_{N-1}$ are
\begin{align*}
\widetilde{X}_N (\la) & :=  \left[
\begin{array}{cc}
\frac{\beta_N}{h_{N-1}} I_m & 0 \\
0 & I_{pm+(N-1-p)r}
\end{array}
\right]
\widetilde{W}_N (\la)
\left[
\begin{array}{c|c|c}
I_{pm} & 0 & 0 \\ \hline
0 & I_{(N-2-p)r} & - g_{N-1} J (\la) \\ \hline
0& 0 & I_r
\end{array}
\right]\\
& =  \left[\begin{array}{c|c|c}
\begin{array}{ccc} \! \!\!
\frac{g_N}{h_{N-1}} \widetilde{D}_1 & \cdots & \! \!\! \frac{g_N}{h_{N-1}} \widetilde{D}_p
\end{array}
&
\begin{array}{ccc} \! \!\!
\frac{g_N}{h_{N-1}} \widetilde{L}_{p+1} & \cdots & \! \!\! \frac{g_N}{h_{N-1}} \widetilde{L}_{N-2}
\end{array}
&
\widetilde{R}_N^{(2)} (\la)  \phantom{\Big|} \\  \hline
\widetilde{E}_N (\la) \phantom{\Big|} & 0& 0 \\  \hline
e_p^T \otimes (-h_p \widetilde{U}^T) & 0 & \left(\prod_{i=p+1}^{N-2} \frac{g_i}{h_i}\right) g_{N-1} I_r \\  \hline
0 & I_{(N-2-p)r} & 0
\end{array} \right],
\end{align*}
where $\widetilde{R}_N^{(2)} (\la)$ is the rational matrix appearing in \eqref{eq.RNlowrank}. Observe that the structure of the last block row of $\widetilde{X}_N (\la)$ allows us to perform an equivalence transformation in $\mathcal{P}_{N-1}$ that turns the block $\left[ \begin{array}{ccc}
\frac{g_N}{h_{N-1}} \widetilde{L}_{p+1} & \cdots & \frac{g_N}{h_{N-1}} \widetilde{L}_{N-2}
\end{array} \right]$ into $0$ without changing the remaining blocks. The resulting matrix is called $\widehat{X}_N (\la)$. Now, denote by $E_{21} (\la)$ the matrix obtained from $\widetilde{E}_N (\la)$ by removing its first block row and its last block column, and observe that $E_{21} (\la)$ is invertible in $\mathcal{P}_{N-1}$ and has the same structure as $Z_{21} (\la)$ in \eqref{eq.appZN2} with $N-2$ replaced by $p-1$. The last block column of $E_{21} (\la)^{-1}$ is denoted by $\widetilde{Y}_{22} (\la)$. With this information, the following equivalence transformations are
\[
\widehat{W}_N (\la)  := \diag \left( I_{m} ,
\left[
\begin{array}{cc}
I_m & -g_{1} I_m \\
0 & I_m
\end{array}
\right]
, I_s \right) \, \diag (I_{2m} , E_{21} (\la)^{-1}, I_{(N-1-p)r}) \,\, \widehat{X}_N (\la),
\]
where $I_s = I_{(p-2)m+ (N-1-p)r}$, and
\begin{align*}
\widehat{S}_N (\la) & := \widehat{W}_N (\la)  \diag \left(
\left[
\begin{array}{cc}
I_{(p-1)m} & -g_{p} \widetilde{Y}_{22} (\la) \\
0 & I_m
\end{array}
\right]
, I_{(N-2-p)r} , I_r \right) \\
& = \left[\begin{array}{c|c|c|c}
\begin{array}{ccc} \! \!\!
\frac{g_N}{h_{N-1}} \widetilde{D}_1 & \cdots & \! \!\!\!\! \frac{g_N}{h_{N-1}} \widetilde{D}_{p-1}
\end{array}
& \widetilde{R}_N^{(1)} (\la) &
0
&
\widetilde{R}_N^{(2)} (\la)  \phantom{\Big|} \\  \hline
0 & \!\! \left(\prod_{i=1}^{p-1} \frac{g_i}{h_i}\right) g_{p} I_m \!\! & 0 & 0 \\  \hline
I_{(p-1)m} & 0 &0 & 0\\  \hline
0 & -h_p \widetilde{U}^T & 0 &\!\! \left(\prod_{i=p+1}^{N-2} \frac{g_i}{h_i}\right) g_{N-1} I_r \!\! \\  \hline
0 & 0 & I_{(N-2-p)r} & 0
\end{array} \right],
\end{align*}
where $\widetilde{R}_N^{(1)} (\la)$ is the rational matrix appearing in \eqref{eq.RNlowrank}.
Finally, the announced matrix $\widetilde{S}_N (\la)$ is obtained from $\widehat{S}_N (\la)$ by using its third block row to transform the block $\left[ \begin{array}{ccc} \! \!
\frac{g_N}{h_{N-1}} \widetilde{D}_1 & \cdots & \! \!\!\! \frac{g_N}{h_{N-1}} \widetilde{D}_{p-1}
\end{array} \right]$ into $0$ without changing the remaining blocks. The structure of $\widetilde{S}_N (\la)$ implies immediately that $\widetilde{S}_N(\xi_k)$ has full column rank for every $\xi_k \in \mathcal{P}_{N-1}$ if and only if $\widetilde{R}_N(\xi_k)$ in \eqref{eq.RNlowrank} has full column rank for every $\xi_k \in \mathcal{P}_{N-1}$.


\end{document}